\scrollmode

\documentclass{amsart}
   \usepackage{graphicx}
\usepackage{amssymb}
\usepackage[all]{xy}

\theoremstyle{plain}
\newtheorem{prop}{Proposition}
\newtheorem{thm}[prop]{Theorem}
\newtheorem{cor}[prop]{Corollary}
\newtheorem{lem}[prop]{Lemma}
\newtheorem{fact}[prop]{Fact}
\newtheorem{claim}{Claim}
\newtheorem{ques}{Question}

\newtheorem*{thmA}{Theorem A}
\newtheorem*{thmB}{Theorem B}

\theoremstyle{definition}

\theoremstyle{remark}

\newtheorem{rem}[prop]{Remark}
\newtheorem{example}[prop]{Example}

\numberwithin{prop}{section} 
\numberwithin{prob}{section} 
\numberwithin{claim}{prop}
\numberwithin{equation}{section}

\newcommand{\image}{\mathrm{im}}
\newcommand{\kernel}{\mathrm{ker}}
\newcommand{\coker}{\mathrm{coker}}
\newcommand{\Hom}{\mathrm{Hom}}

\newcommand{\dExt}{\mathrm{dExt}}
\newcommand{\dTor}{\mathrm{dTor}}
\newcommand{\dH}{\mathrm{dH}}
\newcommand{\dd}{\mathrm{d}}
\newcommand{\ff}{\mathrm{f}}
\newcommand{\iid}{\mathrm{id}}
\newcommand{\frk}{\mathrm{frk}}
\newcommand{\spn}{\mathrm{span}}
\newcommand{\stab}{\mathrm{stab}}
\newcommand{\QZ}{\mathrm{QZ}}
\newcommand{\ccd}{\mathrm{cd}}

\newcommand{\ob}{\mathrm{ob}}
\newcommand{\op}{\mathrm{op}}
\newcommand{\Aut}{\mathrm{Aut}}

\newcommand{\coind}{\mathrm{coind}}
\newcommand{\dcoind}{\mathrm{dcoind}}
\newcommand{\idn}{\mathrm{ind}}
\newcommand{\rst}{\mathrm{res}}
\newcommand{\ifl}{\mathrm{inf}}

\newcommand{\soc}{\mathrm{soc}}
\newcommand{\Osgrp}{\mathrm{Osgrp}}
\newcommand{\rk}{\mathrm{rk}}
\newcommand{\eurk}{\mathfrak{rk}}
\newcommand{\euF}{\eu{F}}
\newcommand{\euC}{\eu{C}}
\newcommand{\euN}{\eu{N}}
\newcommand{\euvN}{\eu{vN}}
\newcommand{\euO}{\eu{O}}
\newcommand{\uE}{\underline{E}}

\newcommand{\sdim}{\mathrm{sdim}}
\newcommand{\ev}{\mathrm{ev}}
\newcommand{\uev}{\underline{\ev}}
\newcommand{\End}{\mathrm{End}}

\newcommand{\OS}{\textup{OS}}
\newcommand{\FP}{\textup{FP}}
\newcommand{\tF}{\textup{F}}
\newcommand{\sF}{\textup{sF}}

\newcommand{\ca}[1]{\mathcal{#1}}
\newcommand{\eu}[1]{\mathfrak{#1}}

\newcommand{\Z}{\mathbb{Z}}
\newcommand{\Q}{\mathbb{Q}}
\newcommand{\R}{\mathbb{R}}
\newcommand{\F}{\mathbb{F}}
\newcommand{\N}{\mathbb{N}}

\newcommand{\caO}{\ca{O}}
\newcommand{\caU}{\ca{U}}
\newcommand{\caV}{\ca{V}}
\newcommand{\caB}{\ca{B}}
\newcommand{\caF}{\ca{F}}
\newcommand{\boh}{\mathbf{h}}
\newcommand{\CO}{\mathfrak{CO}}
\newcommand{\Hier}{\mathrm{Hier}}
\newcommand{\tW}{\widetilde{W}}
\newcommand{\Iw}{\mathrm{Iw}}
\newcommand{\bDelta}{\mathbf{\Delta}}
\newcommand{\bG}{\bar{G}}
\newcommand{\met}{\mathrm{d}}

\newcommand{\caC}{\mathcal{C}}

\newcommand{\ualpha}{\underline{\alpha}}

\newcommand{\caE}{\mathcal{E}}
\newcommand{\eue}{\mathbf{e}}
\newcommand{\beue}{\bar{\eue}}
\newcommand{\boV}{\mathbf{V}}
\newcommand{\boE}{\mathbf{E}}
\newcommand{\boL}{\mathbf{L}}
\newcommand{\eup}{\mathfrak{p}}
\newcommand{\caT}{\mathcal{T}}
\newcommand{\tSigma}{\widetilde{\Sigma}}
\newcommand{\hSigma}{\widehat{\Sigma}}
\newcommand{\eust}{\mathfrak{st}}

\newcommand{\ueue}{\underline{\eue}}
\newcommand{\ubeue}{\underline{\beue}}

\newcommand{\caA}{\mathcal{A}}
\newcommand{\HNN}{\mathrm{HNN}}

\newcommand{\sgn}{\mathrm{sgn}}
\newcommand{\hA}{\widehat{A}}
\newcommand{\Ort}{\mathrm{O}}

\newcommand{\Eub}{\underline{E}}
\newcommand{\der}{\partial}
\newcommand{\eps}{\varepsilon}
\newcommand{\tder}{\tilde{\partial}}
\newcommand{\tdelta}{\tilde{\delta}}
\newcommand{\tphi}{\tilde{\phi}}
\newcommand{\card}{\mathrm{card}}
\newcommand{\teta}{\tilde{\eta}}
\newcommand{\caP}{\ca{P}}
\newcommand{\Bor}{\mathrm{Bor}}
\newcommand{\uX}{\underline{X}}
\newcommand{\uY}{\underline{Y}}
\newcommand{\ux}{\underline{x}}
\newcommand{\uy}{\underline{y}}
\newcommand{\supp}{\mathrm{supp}}
\newcommand{\csupp}{\mathrm{csupp}}
\newcommand{\utSigma}{\underline{\tSigma}}

\newcommand{\CAT}{\mathrm{CAT}}

\newcommand{\baG}{\bar{G}}

\newcommand{\RG}{R[G]}
\newcommand{\RH}{R[H]}
\newcommand{\RbG}{R[\baG]}

\newcommand{\QG}{\Q[G]}
\newcommand{\QO}{\Q[\caO]}
\newcommand{\QH}{\Q[H]}
\newcommand{\QU}{\Q[U]}

\newcommand{\dbl}{[\![}
\newcommand{\dbr}{]\!]}

\newcommand{\dis}{\mathbf{dis}}
\newcommand{\Mod}{\mathbf{mod}}
\newcommand{\RMod}{{}_R\Mod}
\newcommand{\QGmod}{{}_{\QG}\Mod}
\newcommand{\RGdis}{{}_{R[G]}\dis}
\newcommand{\QOdis}{{}_{\QO}\dis}
\newcommand{\RGmod}{{}_{R[G]}\Mod}
\newcommand{\RHdis}{{}_{R[H]}\dis}
\newcommand{\RbGdis}{{}_{R[\baG]}\dis}

\newcommand{\QGdis}{{}_{\QG}\dis}
\newcommand{\QHdis}{{}_{\QH}\dis}
\newcommand{\QUdis}{{}_{\QU}\dis}

\newcommand{\disQG}{\dis_{\QG}}
\newcommand{\disQO}{\dis_{\QO}}

\newcommand{\Qvec}{{}_{\Q}\Mod}
\newcommand{\pdim}{\mathrm{proj}_{\Q}\mathrm{dim}}

\newcommand{\caR}{\ca{R}}
\newcommand{\caS}{\ca{S}}
\newcommand{\caL}{\ca{L}}
\newcommand{\biB}{\mathrm{Bi}}

\newcommand{\tr}{\mathrm{tr}}
\newcommand{\utr}{\underline{\tr}}
\newcommand{\ubG}{\underline{\biB}(G)}

\newcommand{\uder}{\underline{\partial}}

\newcommand{\uOmega}{\underline{\Omega}}
\newcommand{\uXi}{\underline{\Xi}}
\newcommand{\uomega}{\underline{\omega}}
\newcommand{\bomega}{\bar{\omega}}
\newcommand{\argu}{\hbox to 7truept{\hrulefill}}

\begin{document}

\title[Rational discrete cohomology for t.d.l.c. groups]{Rational discrete cohomology for\\
totally disconnected locally compact groups}
\author{I.~Castellano and Th.~Weigel}
\date{\today}
\address{I.~Castellano\\
Dipartimento di Matematica\\
Universit\`a di Bari "Aldo Moro"\\
Via E.~Orabona 4\\
70125 Bari, Italy}
\email{ilaria.castellano@uniba.it}

\address{Th. Weigel\\ Dipartimento di Matematica e Applicazioni\\
Universit\`a degli Studi di Milano-Bicocca\\
Ed.~U5, Via R.Cozzi 55\\
20125 Milano, Italy}
\email{thomas.weigel@unimib.it}

\begin{abstract}
Rational discrete cohomology and homology for a totally disconnected locally compact group $G$
is introduced and studied. The $\Hom$-$\otimes$ identities associated
to the rational discrete bimodule $\biB(G)$ allow to introduce the notion of rational duality groups
in analogy to the discrete case. It is shown that semi-simple groups defined over a non-discrete, non-archimedean
local field are rational t.d.l.c. duality groups, and the same is true for certain topological Kac-Moody groups.
However, Y.~Neretin's group of spheromorphisms of a locally finite regular tree is not even of finite rational discrete
cohomological dimension. For a unimodular t.d.l.c. group $G$ of type $\FP$ it is possible to define an Euler-Poincar\'e 
characteristic $\chi(G)$ which is a rational multiple of a Haar measure. 
This value is calculated explicitly for Chevalley groups
defined over a non-discrete, non-archimedean local field $K$ and some other examples.
\end{abstract}


\maketitle

%%%%%%%%%%
%%% intro %%%%
%%% 2/3/2015 %% 

\section{Introduction}
\label{s:intro}
For a topological group $G$ several cohomology
theories have been introduced and studied in the past.
In many cases the main motivation was to obtain an interpretation of the 
low-dimensional cohomology groups in analogy to discrete groups
(cf. \cite{aumo:conmea}, \cite{moore:grext2}, \cite{moore:grext3}, \cite{moore:grext4}).

The main subject of this paper is \emph{rational discrete cohomology}
for totally disconnected, locally compact (= t.d.l.c.) groups. 
Discrete cohomology can be defined in a very general frame work. We will call a topological group $G$ to be of type \OS, if
\begin{equation}
\Osgrp(G)=\{\,\caO\subseteq G\mid \text{$\caO$ an open subgroup of $G$}\,\}
\end{equation}
is a basis of neighbourhoods of $1$ in $G$, e.g., any profinite group is of type \OS, and,
by van Dantzig's theorem (cf. \cite{vdan:diss}),
any t.d.l.c. group is of type \OS. For any topological group $G$ of type \OS\
and any commutative ring $R$ the category $\RGdis$ of discrete left $\RG$-modules is an abelian category
with enough injectives. This allows to define $\dExt_{\RG}^\bullet(\argu,\argu)$
as the right derived functors of $\Hom_{\RG}(\argu,\argu)$, and $\dH^\bullet(\RG,\argu)$
as the right derived functors of the fixed point functor $\argu^G$. Even in this general
context one has a Hochschild-Lyndon-Serre spectral sequence for 
short exact sequences of topological groups of type \OS\ (cf. \S\ref{ss:LHS}), and a
generalized Eckmann-Shapiro type lemma in the first argument
(cf. \S\ref{ss:ecksha}).

In case that $G$ is a t.d.l.c. group it turns out 
that $\QGdis$ is also an abelian category with enough projectives (cf. Prop.~\ref{prop:permod}). 
Moreover, any projective rational discrete right $\QG$-module is flat with respect to short exact sequences
of rational discrete left $\QG$-modules (cf. \eqref{eq:flat4}).
Thus for a discrete left $\QG$-module $M$ one may define
\emph{rational discrete cohomology} with coefficients in $M$ by
\begin{align}
\dH^k(G,M)&=\dExt^k_G(\Q,M),&&k\geq 0,\label{eq:defdisco}
\intertext{and \emph{discrete rational homology} with coefficients in $M$ by}
\dH_k(G,M)&=\dTor_k^G(\Q,M),&&k\geq 0\label{eq:defdisho}
\end{align}
in analogy to discrete groups (cf. \S\ref{s:rathom}). 

The rational discrete cohomological dimension $\ccd_{\Q}(G)$ (cf. \eqref{eq:pdim3}) reflects
structural information on a t.d.l.c. group $G$, e.g.,
\begin{itemize}
\item[(1)] $G$ is compact if, and only if, $\ccd_{\Q}(G)=0$ (cf. Prop.~\ref{prop:cd}(a));
\item[(2)] the flat rank of $G$ is bounded by $\ccd_{\Q}(G)$ (cf. Prop.~\ref{prop:flat}(b));
\item[(3)] if $G$ is acting discretely on a tree with compact stabilizers then $\ccd_{\Q}(G)\leq 1$
(cf. Prop.~\ref{prop:treeact}(b));
\item[(4)] more generally, if $G$ is acting discretely on a contractable simplical complex $\Sigma$
with compact stabilizers, then $\ccd_{\Q}(G)\leq \dim(\Sigma)$ (cf. Prop.~\ref{prop:ftopcd}(a));
\item[(5)] a nested t.d.l.c. group $G$ satisfies $\ccd_{\Q}(G)\leq 1$ (cf. Prop.~\ref{prop:nested}).
\end{itemize}
The existence of finitely generated projective rational discrete $\QG$-modules
allows to define the \emph{\FP$_k$-property}, $k\in\N_0\cup\{\infty\}$, (cf. \S\ref{ss:FPinf})
in analogy to discrete groups 
(cf. \cite[\S VIII.5]{brown:coh}).
A t.d.l.c. group $G$ is compactly generated if, and only
if, it is of type $\FP_1$ (cf. Thm.~\ref{thm:FP1}).
Using the theory of rough Cayley graphs due to H.~Abels (cf. \cite{abels:spec})
and Bass-Serre theory,
one may also introduce the notion of a {\it compactly presented t.d.l.c. groups} (cf. \S\ref{ss:comppres}).
Such a group must be of type $\FP_2$. 

For a t.d.l.c. group $G$ there is not necessarily a group algebra 
but two natural rational discrete $\QG$-bimodules $\biB(G)$ and $\caC_c(G,\Q)$
and both turn out to be useful in this context.
If $G$ is unimodular, they are isomorphic, but not necessarily
canonically isomorphic; if $G$ is not unimodular, they are non-isomorphic (cf. Prop.~\ref{prop:compar}(b)).
The $\QG$-bimodule $\caC_c(G,\Q)$ just consists of the continuous functions 
with compact support from $G$ to the discrete field $\Q$.
The rational discrete left $\QG$-module $\biB(G)$ is defined by
\begin{equation}
\label{eq:BiG}
\biB(G)=\varinjlim_{\caO\in\CO(G)} \Q[G/\caO],
\end{equation}
where $\CO(G)=\{\,\caO\in\Osgrp(G)\mid \caO\ \text{compact}\,\}$ and the direct limit
is taken along a rational multiple of the transfers (cf. \eqref{eq:biB2}). 
It carries naturally also the structure of a discrete right $\QG$-module.

The rational discrete $\QG$-bimodule $\biB(G)$ can be used to establish
$\Hom$-$\otimes$ identities which are well known for discrete groups (cf. \S\ref{ss:homten}).
This allows to define the notion of a {\it rational t.d.l.c. duality group}. In more detail,
a t.d.l.c. group $G$ will be said to be a rational duality group of dimension $d\geq 0$, if
\begin{itemize}
\item[(i)] $G$ is of type $\FP_\infty$;
\item[(ii)] $\ccd_{\Q}(G)<\infty$;
\item[(iii)] $\dH^k(G,\biB(G))=0$ for $k\not=d$.
\end{itemize}
Note that any discrete virtual duality group in the sense of R.~Bieri and
B.~Eckmann (cf. \cite{biec:dual}) of virtual cohomological
dimension $d\geq 0$ is indeed a rational t.d.l.c. duality group of dimension $d\geq 0$,
and the same is true for discrete duality groups of dimension $d$ over $\Q$
(cf. \cite[\S 9.2]{bieri:hom}). 
For these t.d.l.c. groups one has a non-trivial relation between 
discrete cohomology and discrete homology (cf. Prop.~\ref{prop:dual}).
However, it differs slightly from the usual form.
A t.d.l.c. group $G$ satisfying (i) and (ii) will be said to be of type $\FP$.
Note that a t.d.l.c. group $G$ satisfying (i), (ii) and (iii) must also satisfy
$\ccd_{\Q}(G)=d$ (cf. Prop.~\ref{prop:FP}). 
The rational discrete right $\QG$-module 
\begin{equation}
\label{eq:dualint}
D_G=\dH^d(G,\biB(G))
\end{equation}
will be called the {\it rational dualizing module}.
A compact t.d.l.c. group $G$ is a rational t.d.l.c. duality group of dimension $0$
with dualizing module isomorphic to the trivial right $\QG$-module $\Q$.
If $G$ is a compactly generated t.d.l.c. group satisfying $\ccd_{\Q}(G)=1$,
then $G$ is a rational t.d.l.c. duality group of dimension $1$.
Provided that $G$ acts discretely on a locally finite 
tree $\caT$ with compact stabilizers one has an isomorphism of rational discrete left $\QG$-modules
\begin{equation}
\label{eq:dualinttree}
{}^\times D_G\simeq H^1_c(|\caT|,\Q)\otimes\Q(\Delta)
\end{equation}
(cf. \S\ref{sss:fund}), where $|\caT|$ denotes the {\it topological realization}\footnote{Although this topological space is
also known as the {\it realization}  of the tree $\caT$ (cf. \cite[\S I.2, p.~14]{serre:trees}), 
we have chosen this name in order to avoid the linguistic ambiguity with the {\it geometric realization} of buildings.}
of $\caT$(cf. \S\ref{ss:geore}),
and $H^1_c(\argu,\Q)$ denotes cohomology with compact support with coefficients in the discrete field $\Q$.
If $\Sigma$ is a locally finite $\euC$-discrete simplicial $G$-complex (cf. \S\ref{ss:Fdissimp}),
then $H_c^k(|\Sigma|,\Q)$ carries naturally the structure of a rational discrete left $\QG$-module
(cf. \S\ref{ss:cosup}).
Here $\Q(\Delta)$ denotes the $1$-dimensional rational discrete left $\QG$-module which action is given by
the modular function\footnote{Throughout this paper Haar measures are considered to be
left invariant.} $\Delta\colon G\to\Q^+$ of $G$ (cf. Remark~\ref{rem:biC}).
For a right $\QG$-module $M$ we denote by ${}^\times M$ the associated left $\QG$-module,
i.e., for $m\in M$ and $g\in G$ one has $g\cdot m=m\cdot g^{-1}$.

Let $G$ be a t.d.l.c. group.
The family of compact open subgroups $\euC$ of $G$
is closed under conjugation and finite intersections.
Hence there exists a \emph{$G$-classifying space $\Eub_\euC(G)$}
for the family $\euC$ (cf. \cite[\S 2]{misva:proper}, \cite[Chap.~1, Thm.~6.6]{tdieck:trans})
which is unique up to $G$-homotopy equivalence. 
E.g., if $\Pi=\pi_1(\caA,\Lambda,x_0)$ is the fundamental group of a 
graph of profinite groups, then the tree $\caT=\caT(\caA,\Lambda,\Xi,\caE^+)$
arising in Bass-Serre theory (cf. \cite[\S I.5.1, Thm.~12]{serre:trees}) is an
$\uE_{\euC}(\Pi)$-space (cf. \S\ref{sss:fund}).
For algebraic groups defined over a non-discrete non-archimedean t.d.l.c. field $K$ a 
celebrated theorem of A.~Borel and {J-P.}~Serre yields the following (cf. \S\ref{sss:alg}).

\begin{thmA}
Let $G$ be a simply-connected semi-simple algebraic group defined over a 
non-discrete non-archimedian local field $K$, let $(\caC,S)$ be 
the affine building associated to $G$, and let $|\Sigma(\caC,S)|$
denote the topological realization of the Bruhat-Tits realization of $(\caC,S)$.
Then $G(K)$ is a rational t.d.l.c. duality group of dimension $d=\rk(G)=\dim(\Sigma(\caC,S))$, and
\begin{equation}
\label{eq:dualalg}
{}^\times D_{G(K)}\simeq H^{d}_c(|\Sigma(\caC,S)|,\Q).
\end{equation}
Moreover, $|\Sigma(\caC,S)|$ is an $\uE_{\euC}(G(K))$-space.
\end{thmA}

Here $\rk(G)$ denotes the rank of $G$ as algebraic group defined over $K$,
i.e., $\rk(G)$ coincides with the rank of a maximal split torus of $G$.
As already mentioned in \cite{bose:top} one may think of ${}^\times D_{G(K)}$
as a {\it generalized Steinberg module} of $G(K)$. 

Another source of rational t.d.l.c. duality groups arises in the context of 
topolo\-gical Kac-Moody groups as introduced by B.~R\'emy and M.~Ronan
in \cite[\S 1B]{rr:KM}. In principal, these t.d.l.c. groups behave very similarly
as the examples described in Theorem~B, but there are also quite notable differences.
Any split Kac-Moody group $G(\F)$ defined over a finite field $\F$ has an action
on the positive part $\Xi^+$ of the associated twin building.
Assuming that the associated Weyl group $(W,S)$ is infinite one defines
$\bG(\F)$ as the permutation closure of $G(\F)$ on its action on the chambers of $\Xi^+$.
Then $\bG(\F)$ acts also on $\Sigma(\Xi^+)$, the Davis-Moussang realization of $\Xi^+$,
which is a finite dimensional locally finite simplicial complex.
The subgroup $\bG(\F)^\circ$ of $\bG(\F)$,
which is the closure of the image of the type preserving
automorphisms of $\Xi$, is of finite index in $\bG(\F)$.
This group has the following properties.

\begin{thmB}
Let $\F$ be a finite field, let $\bG(\F)$ be a topological Kac-Moody group
obtained by completing the split Kac-Moody group $G(\F)$, and let $(W,S)$
denote its associated Weyl group.
\begin{itemize}
\item[(a)] Then $\Sigma(\Xi)$ is a tame $\euC$-discrete simplicial $\bG(\F)^\circ$-complex,
and $|\Sigma(\Xi)|$ is an $\uE_{\euC}(\bG(\F)^\circ)$-space.
\item[(b)] $\ccd_{\Q}(\bG^\circ(\F))=\ccd_{\Q}(\bG(\F))=\ccd_{\Q}(W)$.
\item[(c)] For $d=\ccd_{\Q}(\bG^\circ(\F))$ one has
${}^\times D_{\bG(\F)^\circ}\simeq H^{d}_c(|\Sigma(\Xi)|,\Q)$.
\item[(d)] $\bG(\F)^\circ$ (and hence $\bG(\F)$) is a rational t.d.l.c. duality group if,
and only if, $W$ is a rational duality group. 
\end{itemize}
\end{thmB}

One may verify easily that whenever $W$ is a hyperbolic Coxeter group, 
the topological Kac-Moody group $\bG(\F)^\circ$
is indeed a rational t.d.l.c. duality group (cf. Remark~\ref{rem:hyp}).
However, one may also construct examples of topological Kac-Moody groups
which are not rational t.d.l.c. duality groups (cf. Remark~\ref{rem:notdu}).

We close our investigations with a short discussion of
the {\it Euler-Poincar\'e characteristic} $\chi(G)$ one may define
for a unimodular t.d.l.c. group $G$ of type $\FP$.
The value $\chi(G)$ is a rational multiple of a left invariant
Haar measure $\mu_\caO$, $\mu_\caO(\caO)=1$,
for some compact open subgroup $\caO$ of $G$.
Indeed, for a discrete group $\Pi$ of type $\FP$ it is straightforward to verify that
$\chi(\Pi)=\chi_{\Pi}\cdot\mu_{\{1\}}$,
where $\chi_{\Pi}$ denotes the classical Euler-Poincar\'e characteristic (cf. \cite[\S XI.7]{brown:coh}).
Although we present some explicit calculations in section~\ref{ss:euler}, we are still very far away
from understanding the generic behaviour of the Euler-Poincar\'e characteristic.
We conjecture that a unimodular compactly generated t.d.l.c. group $G$ satisfying
$\ccd_{\Q}(G)=1$ should have non-positive Euler-Poincar\'e characteristic
(cf. Question~\ref{ques:cd1}). An affirmative answer to this question
would resolve the accessibility problem, and thus would yield
a complete description of unimodular compactly generated t.d.l.c. groups
of rational discrete cohomological dimension $1$.
\vskip6pt

\noindent
{\bf Acknowledgements:} Sections 2, 3, 4 and 5 are part
of the first author's PhD thesis (cf. \cite{ila:phd}).
The authors would like to thank M.~Bridson for some helpful
discussions concerning $\CAT(0)$-spaces.

%%%%%%%%%%%
%%% disco %%%%%
%%%%%%%%%%%
%% 2/3/2014

\section{Discrete cohomology}
\label{s:condis}

\subsection{Topological groups of type \OS}
\label{ss:TopOS} 
Let $G$ be a topological group of type \OS.
If $H$ is a closed subgroup of $G$, then
$\{\,\caO\cap H\mid \caO\in\Osgrp(G)\,\}$ is a basis of neighbourhood of $1$ in $H$.
Hence $H$ is also of type $\OS$. Moreover, if $\pi\colon G\to\baG$ is a continuous, open,
surjective homomorphism of topological groups,
then $\{\,\pi(\caO)\mid \caO\in\Osgrp(G)\,\}$ is a basis of neighbourhoods
of $1$ in $\baG$. Thus $\baG$ is also of type \OS.
Finally, if $G$ is homeomorphic to an open subgroup
of a topological group $Y$, then $Y$ is of type \OS\ as well.

Any compact totally-disconnected topological group is profinite
and therefore, by definition, of type \OS\ (cf. \cite[Prop.~1.1.0]{ser:gal}).
By van Dantzig's theorem (cf. \cite{vdan:diss}), every t.d.l.c. group  
contains a compact open subgroup. 
Hence from the previously mentioned properties one conlcudes the following.

\begin{fact}
\label{fact:tdlcOS}
Any totally-disconnected locally-compact group is of type \OS.
\end{fact}

%%%%%

\subsection{Discrete modules of topological groups of type $\OS$}
\label{ss:disOS}
Let $R$ be a commutative ring with $1\in R$.
For a group $G$ we denote by $\RG$ its {\it $R$-group algebra},
and by $\RGmod$ the abelian category of left $\RG$-modules.
Let $G$ be a topological group of type $\OS$. 
For $M\in\ob(\RGmod)$ the subset
\begin{equation}
\label{eq:defd}
\dd(M)=\{\,m\in M\mid \text{$\stab_G(m)$ open in $G$}\,\},
\end{equation}
where $\stab_G(m)=\{\,g\in G\mid g\cdot m=m\,\}$, is an $\RG$-submodule of $M$,
the {\it largest discrete $\RG$-submodule of $M$}.
Here we call a left $\RG$-module $B$ {\it discrete}, if the map
$\argu\cdot\argu\colon G\times B\to B$ is continuous, where $B$ carries the discrete topology.
If $\phi\colon B\to M$ is a morphism of left $\RG$-modules
and $B$ is discrete, one concludes that $\image(\phi)\subseteq \dd(M)$.
This has the following consequences (cf. \cite[Prop. 2.3.10]{weib:hom}).

\begin{fact}
\label{fact:prin}
Let $G$ be a topological group of type \OS, and let
$\RGdis$ denote the full subcategory of $\RGmod$ the objects of which are
discrete left $\RG$-modules. Then
\begin{itemize}
\item[(a)] $\RGdis$ is an abelian category;
\item[(b)] $\dd\colon \RGmod\to\RGdis$
is a covariant additive left exact functor, which is the
right-adjoint of the forgetful functor
$\ff\colon\RGdis\to\RGmod$; 
\item[(c)] $\dd$ is mapping injectives to injectives;
\item[(d)] $\RGdis$ has enough injectives.
\end{itemize}
\end{fact}

The abelian category $\RGmod$ admits minimal injective envelopes,
i.e., for $M\in\ob(\RGmod)$ the maximal essential extension $i_M\colon M\to I(M)$
is a minimal injective envelope (cf. \cite[\S III.11]{mcl:hom}).
Thus, if $M\in\ob(\RGdis)$, then $\dd(i_M)\colon M\to \dd(I(M))$ is  a minimal injective
envelope in $\RGdis$. This has the following consequence.

\begin{fact}
\label{fact:mininjres}
The abelian category $\RGdis$ admits minimal injective resolutions.
\end{fact}

If $M$ and $N$ are left $\RG$-modules, then $M\otimes N=M\otimes_R N$
is again a left $\RG$-module. Moreover, if $G$ is a topological group of
type \OS\ and $M$ and $N$ are discrete left $\RG$-modules, then $M\otimes N$
is discrete as well.

%%%%%%%

\subsection{Closed subgroups}
\label{ss:closed}
Let $H$ be a closed subgroup of the topological group $G$ of type \OS.
Then one has a {\it restriction functor}
\begin{equation}
\label{eq:rst}
\rst^G_H(\argu)\colon\RGdis\longrightarrow\RHdis
\end{equation}
which is exact. The {\it discrete coinduction functor} 
\begin{equation}
\label{eq:dcoind}
\dcoind_H^G(\argu)=\dd\circ\coind_H^G(\argu)\colon\RHdis\longrightarrow\RGdis
\end{equation}
where $\coind_H^G(M)=\Hom_{\RH}(\RG,M)$, $M\in\ob({}_{\RH}\Mod)$,
is the coinduction functor, is left exact.
It is the right-adjoint of the restriction functor.
In particular, $\dcoind_H^G(\argu)$ is mapping injectives to injectives
(cf. \cite[Prop. 2.3.10]{weib:hom}).

Let $N$ be a closed normal subgroup of the topological group $G$ of type \OS.
The canonical projection $\pi\colon G\to \baG$, where $\baG=G/N$,
is surjective, continuous and open. One has an {\it inflation functor}
\begin{equation}
\label{eq:ifl}
\ifl_{\baG}^G(\argu)\colon \RbGdis\longrightarrow\RGdis
\end{equation}
which is exact. The inflation functor is left-adjoint to the
{\it $N$-fixed point functor}
\begin{equation}
\label{eq:Nfix}
\argu^N\colon\RGdis\longrightarrow\RbGdis
\end{equation}
which is left exact. In particular, one has the following (cf. \cite[Prop. 2.3.10]{weib:hom}).

\begin{fact}
\label{fact:extra}
Let $G$ be a topological group of type \OS, and
let $N$ be a closed normal subgroup of $G$. 
Then the $N$-fixed point functor 
$\argu^N\colon\RGdis\longrightarrow\RbGdis$, $\baG=G/N$,
is mapping injectives to injectives.
\end{fact}

%%%%%%

\subsection{Open subgroups}
\label{ss:open}
Let $H\subseteq G$ be an open subgroup of the topological group $G$ of type \OS, and let
$M\in\ob(\RHdis)$. Then for $y=\sum_{1\leq i\leq n} r_i g_i\otimes m_i\in\idn_H^G(M)$ one has
$\bigcap_{1\leq i\leq n} {}^g \stab_H(m_i)\subseteq \stab_G(y)$. In particular, $M$ is a discrete left
$\RG$-module, and one has an exact {\it induction functor}
\begin{equation}
\label{eq:ind}
\idn_H^G(\argu)\colon\RHdis\longrightarrow\RGdis.
\end{equation}
which is the left adjoint of the restriction functor $\rst^G_H(\argu)$. 
Hence one has the following (cf. \cite[Prop. 2.3.10]{weib:hom}).

\begin{fact}
\label{fact:resopen}
Let $H\subseteq G$ be an open subgroup of the topological group $G$ of type \OS.
Then $\rst^G_H(\argu)$ is mapping injectives to injectives.
\end{fact}

%%%%%%

\subsection{Discrete cohomology}
\label{ss:discoh}
Let $G$ be a topological group of type \OS. For $M\in\ob(\RGdis)$ we denote by
\begin{equation}
\label{eq:dExt}
\dExt^k_{\RG}(M,\argu)=\caR^k\Hom_{\RGdis}(M,\argu)
\end{equation}
the right derived functors of $\Hom_{\RG}(M,\argu)$ in $\RGdis$.
The contra-/covariant bifunctors 
\begin{equation}
\label{eq:dExt2}
\dExt^k_{\RG}(\argu,\argu)\colon
\RGdis^{\op}\times\RGdis\longrightarrow \RMod,
\end{equation}
$k\geq 0$, admit long exact sequences in the first and in the second
argument. Moreover, $\dExt^0_{\RG}(M,\argu)\simeq \Hom_{\RG}(M,\argu)$,
and $\dExt^k_{\RG}(M,I)=0$ for every injective, discrete left $\RG$-module $I$.
We define the {\it $k^{th}$ discrete cohomology group} of $G$
with coefficients in $\RGdis$ by
\begin{equation}
\label{eq:dExt3}
\dH^k(\RG,\argu)=\dExt_{\RG}^k(R,\argu),\qquad k\geq 0,
\end{equation}
where $R$ denotes the trivial left $\RG$-module.

\begin{example}
\label{ex:dH}
Let $R=\Z$.

\noindent
(a) If $G$ is a discrete group, $G$ is of type \OS\ and ${}_{\Z[G]}\dis={}_{\Z[G]}\Mod$.
In particular, $\dH^\bullet(\Z[G],\argu)=H^\bullet(G,\argu)$ coincides
with ordinary group cohomology (cf. \cite{brown:coh}).

\noindent
(b) If $G$ is a profinite group, $G$ is of type \OS\ and ${}_{\Z[G]}\dis$ coincides with
the abelian category of discrete $G$-modules. Thus
$\dH^\bullet(\Z[G],\argu)=H^\bullet(G,\argu)$ coincides
with the Galois cohomology groups discussed in \cite{ser:gal}.
\end{example}

%%%%%%%

\subsection{The restriction functor}
\label{ss:resdH}
Let $H\subseteq G$ be a closed subgroup of the topological group $G$ of type \OS, and let
$M\in\ob(\RGdis)$. Let $(I^\bullet,\der^\bullet,\mu_G)$ be an injective resolution
of $M$ in $\ob(\RGdis)$, and let $(J^\bullet,\delta^\bullet,\mu_H)$ be an injective
resolution of $\rst^G_H(M)$ in $\RHdis$. By the comparison theorem in homological algebra, one has
a morphism of cochain complexes of left $\RH$-modules
$\phi^\bullet\colon \rst^G_H(I^\bullet)\to J^\bullet$ such that the diagram
\begin{equation}
\label{eq:resdH}
\xymatrix{
\rst^G_H(M)\ar[r]^{\mu_G}\ar@{=}[d]&\rst^G_H(I^0)\ar[r]^{\der^0}\ar[d]^{\phi^0}
&\rst^G_H(I^1)\ar[r]^{\der^1}\ar[d]^{\phi^1}&
\rst^G_H(I^2)\ar[r]^{\der^2}\ar[d]^{\phi^2}&\ldots\\
\rst^G_H(M)\ar[r]^{\mu_H}&J^0\ar[r]^{\delta^0}&J^1\ar[r]^{\delta^1}&J^2\ar[r]^{\delta^2}&\ldots
}
\end{equation}
commutes.
Moreover, $\phi^\bullet$ is unique up to chain homotopy equivalence
with respect to this property. Let 
\begin{equation}
\label{eq:resdH1}
\tphi^k\colon \xymatrix{(I^k)^G\ar[r]& (I^k)^H\ar[r]^{(\phi^k)^H}&(J^k)^H.}
\end{equation}
Then $\tphi^\bullet\colon ((I^\bullet)^G,\tder^\bullet)\to (J^\bullet)^H,\tdelta^\bullet)$ is a
mapping of cochain complexes of $R$-modules. Applying the cohomology functor yields a 
mapping of cohomological functors
\begin{equation}
\label{eq:resdH2}
\rst^\bullet_{G,H}(M)\colon\dH^\bullet(\RG,M)\longrightarrow\dH^\bullet(\RH,\rst^G_H(M)).
\end{equation}
The map $\rst^\bullet_{G,H}(\argu)$ will be called the {\it restriction functor} in discrete
cohomology.

%%%%%%

\subsection{The inflation functor}
\label{ss:infdH}
Let $N$ be a closed normal subgroup of the topological group of type \OS,
and let $\pi\colon G\to\baG$, $\baG=G/N$, denote the canonical projection.
Let $M\in\ob(\RbGdis)$, and let $(I^\bullet,\der^\bullet,\mu_{\baG})$ be an injective
resolution of $M$ in $\RbGdis$.
Let $(J^\bullet,\delta^\bullet,\mu_G)$ be an injective resolution of
$\ifl_{\baG}^G(M)$ in $\RGdis$. By the comparison theorem in homological algebra, one has a
mapping of cochain complexes of left $\RG$-modules
$\phi^\bullet\colon \ifl_{\baG}^G(I^\bullet)\to J^\bullet$
which is unique up to chain homotopy equivalence such that the diagram
\begin{equation}
\label{eq:ifldH}
\xymatrix@C=1.1truecm{
\ifl_{\baG}^G(M)\ar[r]^{\ifl_{\baG}^G(\mu_{\baG})}\ar@{=}[d]&
\ifl_{\baG}^G(I^0)\ar[r]^{\der^0}\ar[d]^{\phi^0}&
\ifl_{\baG}^G(I^1)\ar[r]^{\der^1}\ar[d]^{\phi^1}&
\ifl_{\baG}^G(I^2)\ar[r]^{\der^2}\ar[d]^{\phi^2}&\ldots\\
\ifl_{\baG}^G(M)\ar[r]^{\mu_G}&
J^0\ar[r]^{\delta^0}&
J^1\ar[r]^{\delta^1}&
J^2\ar[r]^{\delta^2}&\ldots
}
\end{equation}
commutes. Applying first the fixed point functor $\argu^G$ and then the cohomology functor
one obtains a functor
\begin{equation}
\label{eq:ifldH2}
\ifl^\bullet_{\baG,G}(M)\colon\dH^\bullet(\RbG,M)\longrightarrow \dH^\bullet(\RG,\ifl_{\baG}^G(M)),
\end{equation}
which will be called the {\it inflation functor} in discrete cohomology.

%%%%%%

\subsection{The Hochschild - Lyndon - Serre spectral sequence}
\label{ss:LHS}
Let $N$ be a closed normal subgroup of the topological group $G$ of type \OS,
and let $\pi\colon G\to\baG$, $\baG=G/N$, denote the canonical projection.
The functor
$\argu^G\colon\RGdis\longrightarrow\RMod$ can be decomposed by
$\argu^G=(\argu^{\baG})\circ(\argu^N)$. Moreover,
$\argu^N\colon\RGdis\longrightarrow\RbGdis$ is mapping injectives to injectives (cf. Fact~\ref{fact:extra}).
Thus for $M\in\ob(\RGdis)$ the Grothendieck spectral sequence (cf. \cite[Thm.~5.8.3]{weib:hom})
yields a {\it Hochschild-Lyndon-Serre
spectral sequence}
\begin{equation}
\label{eq:HLS}
E_2^{s,t}(M)=\dH^s(\RbG,\dH^t(R[N],\rst^G_N(M)))\ \Longrightarrow \dH^{s+t}(\RG,M)
\end{equation}
of cohomological type concentrated in the first quadrant which is functorial in $M$. The edge homomorphisms are given by
\begin{equation}
\label{eq:HLS1}
\begin{aligned}
e_h^s\colon&\xymatrix@C=1.8truecm{\dH^s(\RbG,M^N)\ar[r]^-{\ifl_{\baG,G}^s(M^N)}&E_2^{s,0}(M)\ar[r]&
E_\infty^{s,0}(M),}\\
e_v^t\colon&\xymatrix@C=1.8truecm{E_\infty^{0,t}(M)\ar[r]&E_2^{0,t}(M)\ar[r]^-{\rst^t_{G,N}(M)^N}&
\dH^t(R[N],M)^{G/N},}\\
\end{aligned}
\end{equation}
and one has a $5$-term exact sequence
\begin{equation}
\label{eq:HLS2}
\xymatrix{
0\ar[r]&\dH^1(\RbG,M^N)\ar[r]^{e_h^1}&
\dH^1(\RG,M)\ar[r]^{e_v^1}&
\dH^1(R[N],M)^{G/N}\ar[d]^{d_2^{0,1}}\\
&&\dH^2(\RG,M)&
\dH^2(\RbG,M^N)\ar[l]_{e_h^2}
}
\end{equation}

%%%%%%%

\subsection{An Eckmann-Shapiro type lemma}
\label{ss:ecksha}
Let $H$ be an open subgroup of the topological group $G$ of type \OS,
and let $M\in\ob(\RGdis)$. Let $(I^\bullet,\der^\bullet,\mu_G)$ be an injective
resolution of $M$ in $\RGdis$. Then
$(\rst^G_H(I^\bullet),\rst^G_H(\der^\bullet),\rst^G_H(\mu_G))$ is an injective
resolution of $\rst^G_H(M)$ in $\RHdis$ (cf. Fact~\ref{fact:resopen}).
For $B\in\ob(\RHdis)$ one has natural isomorphisms
\begin{equation}
\label{eq:ecksha1}
\Hom_{\RG}(\idn_H^G(B),I^k)\simeq \Hom_{R[H]}(B,\rst^G_H(I^k))
\end{equation}
$k\geq 0$, which yield isomorphisms
\begin{equation}
\label{eq:ecksha2}
\dExt^k_{\RG}(\idn_H^G(B),M)\simeq\dExt^k_{R[H]}(B,\rst^G_H(M)),
\end{equation}
$k\geq 0$. In particular, for $B=R$ one obtains natural isomorphisms
\begin{equation}
\label{eq:ecksha3}
\eps^k\colon\dExt^k_{\RG}(\idn_H^G(R),\argu)\longrightarrow\dH^k(R[H],\argu)
\end{equation}
such that 
\begin{equation}
\label{eq:ecksha4}
\xymatrix{
&\dExt^\bullet_{\RG}(\idn^G_H(R),\argu)\ar[rd]^{\eps^\bullet}&\\
\dExt_{\RG}^\bullet(R,\argu)\ar[0,2]^{\rst^\bullet_{G,H}(\argu)}\ar[ru]^{\dExt^\bullet_{\RG}(\eps)}&&\dH^\bullet(R[H],\rst^G_H(\argu))
}
\end{equation}
is a commutative diagram of cohomological functors with values in $\RGdis$,
where $\eps\colon\idn_H^G(R)\to R$ is the canonical map.

%%%%%%%%%%%%
%%% ratdisco %%%%
%%%%%%%%%%%%
%%% 2/3/2015

\section{Rational discrete cohomology for t.d.l.c. groups}
\label{s:ratdis}
Throughout this section we will assume that 
$G$ is a t.d.l.c. group. Furthermore, we
choose $R=\Q$ to be the field of rational numbers.
From now on we will omit the appearance of the field $\Q$ in the notation,
i.e., we put $\dH^\bullet(G,\argu)=\dH^\bullet(\Q[G],\argu)$,
$\dExt^\bullet_{\Q[G]}(\argu,\argu)=\dExt^\bullet_G(\argu,\argu)$, etc.

%%%%%%%%%

\subsection{Profinite groups}
\label{ss:prgr}
It is well known that a t.d.l.c. group is compact if, and only if, it is profinite 
(cf. \cite[Prop.~1.1.0]{ser:gal}). The main goal of this subsection is to establish the following property
which will turn out to be essential for our purpose.

\begin{prop}
\label{prop:prof}
Let $G$ be a profinite group. Then every discrete left $\QG$-module $M$ is injective and projective in $\QGdis$.
\end{prop}

\begin{proof} The proposition is a consequence of several facts.

\begin{claim}
\label{cl:prf1}
One has $\dH^k(G,M)=0$ for every $M\in \ob(\QGdis)$ and $k\geq 1$.
In particular, every finite dimensional discrete left $\QG$-module is projective.
\end{claim}

\begin{proof}[Proof of Claim~\ref{cl:prf1}]
By \cite[\S 2.2, Prop.~8]{ser:gal}, one has 
\begin{equation}
\label{eq:sergal8}
\textstyle{\dH^k(G,M)=\varinjlim_{U\triangleleft_\circ G} \dH^k(G/U,M^U),}
\end{equation} 
where the projective limit is running over all open normal subgroups $U$ of $G$.
By Maschke's theorem, one has $\dH^k(G/U,M^U)=0$ for $k\geq 1$,
and thus, by \eqref{eq:sergal8}, $\dH^k(G,M)=0$ for all $k\geq 1$.
Let $U$ be an open subgroup of $G$. Since $\dExt_U^1(\Q,\argu)=0$,
$\Hom_U(\Q,\argu)\colon\QUdis\to\Qvec$ is an exact
functor. In particular, $\Q\in\ob(\QUdis)$ is projective. Thus, as $\rst^G_U(\argu)$ is exact,
$\Q[G/U]\simeq\idn_U^G(\Q)$ is projective 
for every open subgroup $U$ of $G$ (cf. \cite[Prop.~2.3.10]{weib:hom}).
By Maschke's theorem, every finite dimensional discrete left $\QG$-module is
a direct summand of $\Q[G/U]$ for some open subgroup $U$ of $G$. This yields the claim.
\end{proof}

As a consequence one obtains the following.

\begin{claim}
\label{cl:prf2}
Let $S\in \ob(\QGdis)$ be irreducible, and let $M\in \ob(\QGdis)$. Then $\dExt^k_G(S,M)=0$ for all $k\geq 1$.
\end{claim}

\begin{proof}[Proof of Claim~\ref{cl:prf2}]
Note that every irreducible discrete left $\QG$-module $S$ is of finite dimension,
and therefore projective (cf. Claim~\ref{cl:prf1}).
\end{proof}

For a discrete left $\QG$-module $M$ we denote by
$\soc(M)\subseteq M$ the {\it socle} of $M$, i.e., $\soc(M)$ is the
sum of all irreducible left $\QG$-submodules of $M$.
Moreover, as $m\in M$ is contained in the finite dimensional discrete left $\QG$-module $\QG\cdot m$, one has that 
$\soc(M)=0$ if, and only if, $M=0$.

\begin{claim} 
\label{cl:prf3}
Every discrete left $\QG$-module $M$ is injective in $\ob(\QGdis)$. 
\end{claim}

\begin{proof}[Proof of Claim~\ref{cl:prf3}]
Let $(I^\bullet,\delta^\bullet)$ be a minimal injective resolution of $M$ in $\QGdis$ (cf. Fact~\ref{fact:mininjres}). 
By Claim~\ref{cl:prf2}, for any irreducible $S\in \ob(\QGdis)$ one has for $k\geq 1$ that
$0=\dExt^k_G(S,M)=\Hom_G(S,I^k)$
(cf. \cite[Cor.~2.5.4(ii)]{ben:coho1}). In particular, $\soc(I^k)=0$, and therefore $I^k=0$. Hence $M$ is injective. 
\end{proof}

\begin{claim}
\label{cl:prf4}
The category $\QGdis$ has enough projectives.
\end{claim}

\begin{proof}[Proof of Claim~\ref{cl:prf4}]
Let $M\in \ob(\QGdis)$.
By definition, for $m\in M$ the left $\QG$-submodule $\Q[G]\cdot m\subseteq M$ is finite dimensional and discrete.
Thus, by Claim~\ref{cl:prf1}, $\Q[G]\cdot m$ is projective in $\QGdis$, 
and therefore, $P=\coprod_{m\in M}\Q[G]\cdot m$ is projective in $\QGdis$.
The canonical map $\pi\colon P\to M$ is surjective, and this yields the claim.
\end{proof}

\begin{claim}
\label{cl:prf5}
Every discrete left $\Q[G]$-module is projective in $\QGdis$. 
\end{claim}

\begin{proof}[Proof of Claim~\ref{cl:prf5}] 
Let $M\in\ob(\QGdis)$. 
By Claim~\ref{cl:prf4}, there exist a projective discrete left $\QG$-module $P$ and a map $\pi\colon P\to M$ such that 
\begin{equation}
\label{eq:seqprf}
\xymatrix{
0\ar[r]& \kernel(\pi)\ar[r]& P\ar[r]^{\pi}& M\ar[r]& 0
}
\end{equation}
is a short exact sequence in $\QGdis$. By Claim~\ref{cl:prf3}, $\kernel(\pi)$ is injective.
Hence \eqref{eq:seqprf} splits, and $M$ is isomorphic to a direct summand of $P$, i.e.,
$M$ is projective.
\end{proof}
\end{proof}

%%%%%%%%%%%%%%%%%%%%

\subsection{Discrete permutation modules}
\label{ss:projdis}
Let $G$ be a t.d.l.c. group. A left $G$-set $\Omega$ 
is said to be {\it discrete}, if $\stab_G(\omega)$ is an open subgroup of $G$ for any $\omega\in\Omega$.
In this case $\argu\cdot\argu\colon G\times\Omega\to\Omega$ is continuous, where $\Omega$ is considered
to be a discrete topological space and $G\times\Omega$ carries the product topology.
We say that $\Omega$ is a {\it discrete left $G$-set with compact stabilizers}, if
$\stab_G(\omega)$ is a compact and open subgroup for any $\omega\in\Omega$.

Let $\Omega$ be a left $G$-set. Then $\Q[\Omega]$ - the free $\Q$-vector space over the set $\Omega$ -
carries canonically the structure of left $\QG$-module.
Moreover, $\Omega$ is a discrete left $G$-set if, and only if, $\Q[\Omega]$ is a discrete left $\QG$-module.
For a discrete left $G$-set $\Omega$, $\Q[\Omega]$ is also called a {\it discrete left $\QG$-permutation module}.
Additionally, one has the following property.

\begin{prop}\label{prop:permod}
Let $G$ be a t.d.l.c. group, and let $\Omega$ be a discrete left $G$-set with compact stabilizers.
Then $\Q[\Omega]$ is projective in $\QGdis$. 
In particular, the abelian category $\QGdis$ has enough projectives.
\end{prop}

\begin{proof}
Let $\Omega=\bigsqcup_{i\in I}\Omega_i$ be the decomposition of $\Omega$ into $G$-orbits.
Since $\Q[\Omega]=\coprod_{i\in I}\Q[\Omega_i]$ and as the coproduct of projective objects
remains projective, it suffices to prove the claim for transitive discrete left $G$-sets $\Omega$
with compact stabilizers, i.e., we may assume that for $\omega\in \Omega$ one has an isomorphism of left $G$-sets
$\Omega\simeq G/\caO$, and an isomorphism of left $\QG$-modules 
$\Q[\Omega]\simeq\idn_{\caO}^G(\Q)$, where
$\caO=\stab_G(\omega)$.
The induction functor is left adjoint to the restriction functor which is exact. 
Since the trivial left $\QO$-module is projective in $\QOdis$ (cf. Prop.~\ref{prop:prof}),
and as $\idn_{\caO}^G(\argu)$ is mapping projectives to projectives
(cf. \S\ref{ss:open} and \cite[Prop. 2.3.10]{weib:hom}), one concludes that $\Q[\Omega]$ is projective.

Let $M\in \ob(\QGdis)$. For each element $m$ in $M$ let $\caO_m$ be a compact open subgroup of $\stab_G(m)$. 
Then $\pi_m\colon\idn^G_{\caO_m}(\Q)\to M$ given by
$\pi_m(g\otimes1)=g\cdot m$ is a mapping of left $\Q[G]$-modules.
It is straightforward to verify that 
\begin{equation}
\label{eq:coprodM}
\textstyle{\pi=\coprod_{m\in M}\pi_m\colon\coprod_{m\in M} \idn_{\caO_m}^G(\Q)\longrightarrow M}
\end{equation}
is a surjective mapping of discrete left $\QG$-modules, and 
$\coprod_{m\in M} \idn_{\caO_m}^G(\Q)\simeq\Q[\bigsqcup_{m\in M} G/\caO_m]$
is projective.
\end{proof}

It is somehow astonishing that for a t.d.l.c. group $G$, 
the category $\QGdis$ is a full subcategory of $\QGmod$ with enough projectives and enough injectives 
(cf. Fact~\ref{fact:prin}),
but the reader may have noticed that this is a direct consequence of Proposition~\ref{prop:prof}.

\begin{cor}
\label{cor:charproj}
Let $G$ be a t.d.l.c. group.
A discrete left $\QG$-module $M$ is projective if, and only if, 
it is a direct summand of a discrete left $\QG$-permutation module $\Q[\Omega]$
for some discrete left $G$-set $\Omega$ with compact stabilizers.
\end{cor}

From Proposition~\ref{prop:permod}
and Corollary~\ref{cor:charproj} one deduces the following properties.

\begin{prop}
\label{prop:rstidn}
Let $G$ be a t.d.l.c. group and, let $H\subseteq G$ be a closed subgroup of $G$.
Then $\rst^G_H\colon\QGdis\rightarrow\QHdis$ is mapping projectives to projectives.
Moreover, if $H$ is open in $G$, then
$\idn^G_H\colon\QHdis\rightarrow\QGdis$ is mapping projectives to projectives.
\end{prop}

\begin{proof}
By Corollary~\ref{cor:charproj}, it suffices to prove 
that $\rst^G_H(\Q[\Omega])$ is projective for any discrete left
$\Q[G]$-permutation with compact stabilizers. Since $\rst^G_H(\argu)$
commutes with coproducts, we may also assume that $\Omega$ is a transitive
discrete left $G$-set with compact stabilizers, i.e.,  for $\omega\in\Omega$
and $\caO=\stab_G(\omega)$ one has $\Omega\simeq G/\caO$. Let $\caR\subseteq G$
be a set of representatives of the $(H,\caO)$-double cosets in $G$, i.e., 
$G=\bigsqcup_{r\in \caR} H\,r\,\caO$. Then
\begin{equation}
\label{eq:rstcl}
\textstyle{\rst^G_H(\Q[G/\caO])\simeq\coprod_{r\in\caR} \Q[H/(H\cap {}^r\caO)].}
\end{equation}
Since $H$ is closed and $\caO$ compact and open, $H\cap{}^r\caO$ is compact and open in $H$.
Hence, by Proposition~\ref{prop:permod}, $\rst^G_H(\Q[G/\caO])$ is a projective
discrete left $\QH$-module.

Assume that $H$ is open in $G$. Since $\idn_H^G(\argu)\colon\QHdis\to\QGdis$
is the left adjoint of the exact functor $\rst^G_H(\argu)\colon\QGdis\to\QHdis$,
$\idn_H^G(\argu)$ is mapping projectives to projectives (cf. \cite[Prop. 2.3.10]{weib:hom}).
\end{proof}

We close this subsection with the following
characterization of compact t.d.l.c. groups.

\begin{prop}
\label{prop:ccom}
Let $G$ be a t.d.l.c. group, and let $\caO$ be a compact open subgroup of $G$.
Then $\Hom_G(\Q,\Q[G/\caO])\not=0$ if, and only if, $G$ is compact.
Moreover, if $G$ is compact, then $\Hom_G(\Q,\Q[G/\caO])\simeq \Q$.
\end{prop}

\begin{proof}
If $G$ is not compact, then $|G:\caO|=\infty$. Hence
$\Hom_G(\Q,\Q[G/\caO])=0$ (cf. \cite[\S III.5, Ex.~4]{brown:coh}).
If $G$ is compact, then $|G\colon \caO|<\infty$,
and $\Q[G/\caO]$ coincides with the coinduced module
$\coind_\caO^G(\Q)$ (cf. \cite[\S III, Prop.~5.9]{brown:coh}).
Hence $\Hom_G(\Q,\Q[G/\caO])\simeq\Hom_{\caO}(\Q,\Q)\simeq \Q$.
This yields the claim.
\end{proof}

%%%%%%%%%%%%%%
\subsection{Signed rational discrete permutation modules}
\label{ss:signedperm}
For some application it will be useful to consider
{\it signed discrete permutation modules}.
Let $(\Omega,\bar{\argu},\cdot)$ be a {\it signed discrete left $G$-set}, i.e.,
$(\Omega,\bar{\argu})$ is a signed set (cf. \S\ref{ss:ssets}),
and $(\Omega,\cdot)$ is a left $G$-set satisfying
\begin{equation}
\label{eq:sdpm}
\overline{g\cdot \omega}=g\cdot \bar{\omega},
\end{equation}
for all $g\in G$ and $\omega\in\Omega$. Then (cf. \eqref{eq:sset1})
\begin{equation}
\label{eq:sdpm2}
\Q[\uOmega]=\Q[\Omega]/\spn_{\Q}\{\,\omega+\bomega\mid\omega\in\Omega\,\}
\end{equation}
is a discrete left $\QG$-module. 
For $\omega\in\Omega$ let $\uomega$ denote its canonical image in $\Q[\uOmega]$.

Let $\Omega=\bigsqcup_{j\in J}\Omega_r$ be a 
decomposition of $\Omega$ in $G\times C_2$-orbits, where $C_2=\Z/2\Z$ and $\sigma=1+2\Z\in C_2$
acts by application of $\bar{\argu}$. We say that $G$ acts {\it without inversion} on $\Omega_j$, if $g\cdot\omega\not=\omega$
for all $g\in G$ and $\omega\in\Omega_j$, otherwise we say that $G$ acts {\it with inversion}.
In this case, there exists $g\in G$ and $\omega\in\Omega_j$ such that $g\cdot \omega=\bomega$. Moreover,
the action of $G_{\pm\omega}=\stab_G(\{\omega,\bomega\})$ on $\Q\cdot\uomega\subseteq\Q[\uOmega]$ is given by
a {\it sign character}
\begin{equation}
\label{eq:sdpm3}
\sgn_\omega\colon G_{\pm\omega}\longrightarrow\{\pm 1\}.
\end{equation}
Moreover, by choosing representatives of the $G\times C_2$-orbits,
one obtains an isomorphism of discrete left $\QG$-modules
\begin{equation}
\label{eq:sdpm4}
\Q[\uOmega]\simeq\coprod_{\omega\in\caR_{\mathrm{w}}}\idn_{G_{\pm\omega}}^G(\Q(\sgn_\omega))\oplus
\coprod_{\varpi\in\caR_{\mathrm{wo}}}\idn_{G_{\varpi}}^G(\Q)).
\end{equation}
Note that by the $\argu^{\pm}$-construction (cf. \S\ref{ss:ssets}) any rational
discrete permutation module $\Q[\Omega]$ can be considered as the signed
rational discrete permutation module $\Q[\uOmega^{\pm}]$.
%%%%%%%%%%%%%%

\subsection{The rational discrete cohomological dimension} 
\label{ss:cdq}
Let $M$ be a discrete left $\QG$-module for a t.d.l.c. group $G$.
Then $M$ is said to have {\it projective dimension} less or equal to $n\leq\infty$, if $M$ admits a projective resolution 
$(P_\bullet,\der_\bullet,\eps)$ in $\QGdis$ satisfying $P_k=0$ for $k>n$, i.e., the sequence
\begin{equation}
\label{eq:pdim1}
\xymatrix{
0\ar[r]&P_n\ar[r]^{\der_n}&P_{n-1}\ar[r]^{\der_{n-1}}&\ldots\ar[r]^{\der_2}&P_1\ar[r]^{\der_1}
&P_0\ar[r]^{\eps}&M\ar[r]&0}
\end{equation}
is exact. 
The {\it projective dimension} $\pdim(M)\in\N_0\cup\{\infty\}$ of $M\in\ob(\QGdis)$ 
is defined to be the minimum of such $n\leq\infty$.
As in the discrete case one has the following property.
	
\begin{lem} 
\label{lem:pdim}
Let $G$ be a t.d.l.c. group, and let $M$ be a discrete left $\QG$-module.
Then the following are equivalent.
\begin{itemize}
\item[(i)] $\pdim(M)\leq n$;
\item[(ii)] $\dExt^i(M,\argu)=0$ for all $i>n$;
\item[(iii)] $\dExt^{n+1}(M,\argu)=0$;
\item[(iv)] In any exact sequence 
\begin{equation}
\label{eq:pdim2}
\xymatrix{
0\ar[r]&K\ar[r]&P_{n-1}\ar[r]^{\der_{n-1}}&\ldots\ar[r]^{\der_2}&P_1\ar[r]^{\der_1}
&P_0\ar[r]^{\eps}&M\ar[r]&0}
\end{equation} 
in $\QGdis$ with $P_i$, $0\leq i\leq n-1$, projective, one has that $K$ is projective.
\end{itemize}
\end{lem}
	
\begin{proof}
The proof of \cite[Chap.~VIII, Lemma~2.1]{brown:coh} can be transferred verbatim.
\end{proof}

For a t.d.l.c. group $G$
\begin{equation}
\label{eq:pdim3}
\ccd_{\Q}(G)=\pdim(\Q)
\end{equation}
will be called the {\it rational discrete cohomological dimension} of $G$, i.e., one has
\begin{equation}
\label{eq:cd1}
\begin{aligned}
	\ccd_\Q(G) =& \pdim(\mathbb{Q}) \\
	= & \min(\{\,n\in\N_0\mid \dH^i(G,\argu)=0\ \text{for all $i>n$}\,\}\cup\{\infty\}) \\
	= & \sup\{\ n\in\N_0 \mid \dH^n(G,M)\neq 0\ \text{for some $M\in\ob(\QGdis)$}\,\}. 
\end{aligned}
\end{equation}

\begin{prop}
\label{prop:cd} 
Let $G$ be a t.d.l.c. group. 
\begin{itemize}
\item[(a)] One has $\ccd_\Q (G)=0$ if, and only if, $G$ is compact.
\item[(b)] If $N$ is a closed normal subgroup of $G$, then 
\begin{equation}
\label{eq:cd2}
\ccd_\Q (G)\leq \ccd_\Q(N) +\ccd_\Q (G/N).
\end{equation}
\item[(c)] If $H$ is a closed subgroup of $G$, then 
\begin{equation}
\label{eq:cd3}
\ccd_\Q (H)\leq \ccd_\Q (G).
\end{equation} 
In particular, if $H$ is normal and co-compact, then equality holds in \eqref{eq:cd3}.
\end{itemize}
\end{prop}

\begin{proof} (a) If $G$ is compact, then $G$ is profinite, and thus, by
Proposition~\ref{prop:prof}, $\ccd_{\Q}(G)=\pdim(\Q)=0$. Suppose that
$G$ satisfies $\ccd_{\Q}(G)=0$, i.e., $\Q\in\ob(\QGdis)$ is projective.
By van Dantzig's theorem, there exists a compact open subgroup $\caO$ of $G$.
By hypothesis, the canonical projection $\eps\colon\Q[G/\caO]\to\Q$ admits a section.
In particular, $\Hom_G(\Q,\Q[G/\caO])\not=0$. Hence one concludes that $|G:\caO|<\infty$, and $G$ is compact.
(b) is a direct consequence of the Hochschild-Lyndon-Serre spectral sequence (cf. \eqref{eq:HLS})
and \eqref{eq:cd1}.
(c) By Proposition~\ref{prop:rstidn}, the restriction of a projection resultion of $\Q$ in $\QGdis$
is a projective resolution of $\Q$ in $\QHdis$ showing the inequality \eqref{eq:cd3}.
The final remark is a consequence of \eqref{eq:cd2}, \eqref{eq:cd3} and part (a).
\end{proof}

\begin{rem}
\label{rem:finind}
Let $G$ be a t.d.l.c. group, and let $H\subseteq G$ be such that $|G:H|<\infty$.\break
One concludes from Proposition~\ref{prop:cd}(c) that $\ccd_{\Q}(H)=\ccd_{\Q}(G)$.
However, the following question will remain unanswered in this paper.

\begin{ques}
Suppose that $H$ is a closed co-compact subgroup of the t.d.l.c. group $G$.
Is it true that $\ccd_\Q(H)=\ccd_\Q(G)$?
\end{ques}
\end{rem}

\begin{rem}
\label{rem:discsgrp}
Proposition~\ref{prop:cd}(c) implies that for
any discrete group $D$ of a t.d.l.c. group $G$
one has 
\begin{equation}
\label{eq:discd}
\ccd_{\Q}(D)\leq\ccd_{\Q}(G).
\end{equation}
\end{rem}
%%%%%

\subsection{The flat rank of a t.d.l.c. group}
\label{ss:frk}
Let $G$ be a t.d.l.c. group, and let $H$ be a closed subgroup of $G$.
Then $H$ is called a {\it flat subgroup} of $G$, if
there exists a compact open subgroup $\caO$ of $G$ which is
{\it tidy} for all $h\in H$ (cf. \cite[\S2, Def.~1]{geud:flat}).
For such a subgroup $H(1)=\{\,h\in H\mid s(h)=1\,\}$ is a normal subgroup, and
$H/H(1)$ is a free abelian group over some set $I_H$
(cf. \cite[Cor.~6.15]{geog:nyj}).
Moreover, $\eurk(H)=\card(I_H)\in\N\cup\{\infty\}$
is called the {\it rank of the flat subgroup $H$}.
The {\it flat rank} $\frk(G)$ of $G$ is defined by 
\begin{equation}
\label{eq:frkG}
\frk(G)=\sup(\{\,\eurk(H)\mid \text{$H$ a flat subgroup of $G$}\,\}\cup\{0\})\in\N_0\cup\{\infty\}
\end{equation}
(cf. \cite[\S 1.3]{geud:frk}).
We call the t.d.l.c. group $G$ to be {\it $N$-compact},
if for any compact open subgroup $\caO$ of $G$, the open subgroup
$N_G(\caO)$ is also compact.
One has the following property.

\begin{prop}
\label{prop:flat}
Let $G$ be an $N$-compact t.d.l.c. group.
\begin{itemize}
\item[(a)] Let $H$ be a flat subgroup of $G$.
Then $\eurk(H)=\ccd_{\Q}(H)$.
\item[(b)] $\frk(G)\leq\ccd_{\Q}(G)$.
\end{itemize}
\end{prop}

\begin{proof} (a)
For any abelian group $A$ which is a free over a set $I_A$ one has
\begin{equation}
\label{eq:abel1}
\ccd_{\Q}(A)=\ccd_{\Z}(A)=\card(I_A)\in\N_0\cup\{\infty\}.
\end{equation}
Let $\caO$ be a compact open subgroup such that $\caO$ is tidy for all $h\in H$.
By hypothesis, $N_G(\caO)$ is compact, and therefore 
$C=H\cap N_G(\caO)$ is compact. Then, by Proposition~\ref{prop:cd}(a) and the 
Hochschild-Lyndon-Serre spectral sequence (cf. \eqref{eq:HLS}),
one has natural isomorphisms 
$\dH^k(H,\argu)\simeq\dH^k(H/C,\argu^C)$ for all $k>0$. 
Hence, by \eqref{eq:abel1}, $\ccd_{\Q}(H)=\ccd_{\Q}(H/C)=\eurk(H)$.
(b) is a direct consequence of (a).
\end{proof}

\begin{rem} 
\label{rem:nere}
(Neretin's group of spheromorphisms)
Let $\caT_{d+1}$ be a $d+1$-regular tree, and let
$G=\Hier_\circ(\caT_{d+1})=N_{d+1,d}$ be the group of all {\it almost
automorphisms} (or spheromorphisms) of $\caT_{d+1}$
introduced by Y.~Neretin in \cite{ner:hier}.
It is well known that $G$ is a compactly generated
simple (in particular topologically simple) t.d.l.c. group.
The group $G$ contains a discrete subgroup $D$ which is isomorphic to
the Higman-Thompson group $F_{d+1,d}$.
From this fact one concludes that $G$ contains discrete subgroups 
isomorphic to $\Z^n$ for all $n\geq 1$. Hence, by Remark~\ref{rem:discsgrp}, $\ccd_{\Q}(G)=\infty$.
\end{rem}

%%%%

\subsection{Totally disconnected locally compact groups of type \FP$_\infty$}
\label{ss:FPinf}
Let $G$ be a t.d.l.c. group, and let $M$ be a discrete left $\QG$-module.
Then $M$ is said to be {\it finitely generated}, if there exist compact open subgroups
$\caO_1,\ldots,\caO_n$ of $G$ and a surjective homomorphism
$\pi\colon\coprod_{1\leq j\leq n} \Q[G/\caO_j]\to M$.
A partial projective resolution
\begin{equation}
\label{eq:papr1}
\xymatrix{
P_n\ar[r]^{\der_n}&P_{n-1}\ar[r]^{\der_{n-1}}&\ldots\ar[r]^{\der_2}&P_1\ar[r]^{\der_1}
&P_0\ar[r]^{\eps}&M\ar[r]&0}
\end{equation}
of $M$ will be called to be of {\it finite type}, if $P_j$ is finitely generated for all $0\leq j\leq n$.
One has the following property.

\begin{prop}
\label{prop:FPn}
Let $G$ be a t.d.l.c. group, and let $M$ be a discrete left $\QG$-module.
Then the following are equivalent:
\begin{itemize}
\item[(i)] There is a partial projective resolution $(P_j,\der_j,\eps)_{0\leq j\leq n}$ of finite type
of $M$ in $\QGdis$;
\item[(ii)] the discrete left $\QG$-module $M$
is finitely generated and for every partial projective resolution of finite type
\label{eq:papr2}
\begin{equation}
\xymatrix{
Q_k\ar[r]^{\eth_k}&Q_{k-1}\ar[r]^{\eth_{k-1}}&\ldots\ar[r]^{\eth_2}&Q_1\ar[r]^{\eth_1}
&Q_0\ar[r]^{\eps^\prime}&M\ar[r]&0}
\end{equation}
in $\QGdis$ with $k<n$,  $\ker(\eth_k)$ is finitely generated.
\end{itemize}
\end{prop}

\begin{proof}
The proposition is a direct consequence of generalized form of Schanuel's lemma 
(cf. \cite[Chap.~VIII, Lemma~4.2 and Prop.~4.3]{brown:coh}).
\end{proof}

The discrete left $\QG$-module $M$ will be said to be {\it of type $\FP_n$}, $n\geq 0$, if $M$
satisfies either of the hypothesis (i) or (ii) in Proposition~\ref{prop:FPn},
i.e., $M$ is of type $\FP_0$ if, and only if, $M$ is finitely generated.
If $M$ is of type $\FP_n$ for all $n\geq 0$, then $M$ will be said to be of type
$\FP_\infty$. The t.d.l.c. group $G$ will be called to be of type $\FP_n$, $n\in\N\cup\{\infty\}$,
if the trivial left $\QG$-module $\Q$ is of type $\FP_n$.

%%%%%%%%
%% 2/3/2015
%%%%%%%%
\section{Rational discrete homology for t.d.l.c. groups}
\label{s:rathom}
For a profinite group $\caO$ any short exact sequence
$0\rightarrow L\overset{\alpha}{\to} M\overset{\beta}{\to} N\to 0$ 
of discrete left $\Q[\caO]$-modules splits
(cf. Prop.~\ref{prop:prof}). Hence, if $\Q\in\ob(\disQO)$ denotes the trivial right $\QO$-module,
the sequence
\begin{equation}
\label{eq:flat1}
\xymatrix{
0\ar[r]&\Q\otimes_\caO L\ar[r]^{\iid_{\Q}\otimes\alpha}&
\Q\otimes_\caO M\ar[r]^{\iid_{\Q}\otimes\beta}&
\Q\otimes_\caO N\ar[r]&0
}
\end{equation} 
is exact.
Let $G$ be a t.d.l.c. group, let $\caO\subseteq G$ be a compact open subgroup,
and let $0\rightarrow L\overset{\alpha}{\to} M\overset{\beta}{\to} N\to 0$ be a short exact
sequence of discrete left $\QG$-modules.
Since $\idn_\caO^G(\argu)\colon\disQO\to\disQG$ is an exact functor, and as
\begin{equation}
\label{eq:flat2}
\idn_{\caO}^G(\Q)\otimes_G\argu\simeq\Q\otimes_{\caO}\rst^G_\caO(\argu)\colon
\QGdis\longrightarrow\Qvec
\end{equation}
are naturally isomorphic additive functors, the sequence
\begin{equation}
\label{eq:flat3}
\xymatrix{
0\ar[r]&\idn_\caO^G(\Q)\otimes_G L\ar[r]^{\iid\otimes\alpha}&
\idn_\caO^G(\Q)\otimes_G M\ar[r]^{\iid\otimes\beta}&
\idn_\caO^G(\Q)\otimes_G N\ar[r]&0
}
\end{equation}
is exact, i.e.,
$\idn_\caO^G(\Q)\otimes_G\argu\colon\QGdis\to\Qvec$ is an exact functor.
Moreover, as $\argu\otimes_G\argu$ commutes with direct limits in the first
and in the second argument (cf. \cite[\S VIII.4, p.~195]{brown:coh}), the sequence
\begin{equation}
\label{eq:flat4}
\xymatrix{
0\ar[r]&Q\otimes_G L\ar[r]^{\iid_Q\otimes\alpha}&
Q\otimes_G M\ar[r]^{\iid_Q\otimes\beta}&
Q\otimes_G N\ar[r]&0
}
\end{equation}
is exact for any projective discrete right $\QG$-module $Q$ (cf. Cor.~\ref{cor:charproj}),
i.e., $Q$ is {\it flat}. For a discrete right $\QG$-module $C$, we denote
by 
\begin{equation}
\label{eq:dTor}
\dTor^G_k(C,\argu)\colon\QGdis\longrightarrow\Qvec
\end{equation} 
the left derived functors
of the right exact functor $C\otimes_G\argu\colon \QGdis\to\Qvec$, and define the
{\it rational discrete homology} of $G$ by
\begin{equation}
\label{eq:dhom}
\dH_k(G,\argu)=\dTor^G_k(\Q,\argu).
\end{equation}
By definition, one has
the following properties.

\begin{prop}
\label{prop:dhom}
Let $G$ be a t.d.l.c. group,
let $C$ be a discrete right $\QG$-module, and let $M$ be a 
discrete left $\QG$-module. Then
\begin{itemize}
\item[(a)] $\dH_k(G,M)=0$ for all $k>\ccd_\Q(G)$;
\item[(b)] if $P\in\ob(\QGdis)$ is projective, one has $\dTor_k^G(C,P)=0$ for all $k\geq 1$;
\item[(c)] if $G$ is compact, one has $\dTor_k^G(C,M)=0$ for all $k\geq 1$.
\end{itemize}
\end{prop}

\begin{proof}
(a) is a direct consequence of the definition of $\ccd_\Q(G)$ (cf. \eqref{eq:pdim3}).
(b) Using the previously mentioned arguments one concludes that $\argu\otimes_G P\colon
\disQG\to \Qvec$ is an exact functor. This property implies that for the left derived functors
$\caL_k^M$ of $\argu\otimes_G M$ one has that $\dTor^G_k(C,M)\simeq\caL_k^M(C)$.
Since $P$ is projective, $\caL_k^P=0$ for all $k>1$. This yields the claim.
(c) is a direct consequence of Proposition~\ref{prop:prof}.
\end{proof}

%%%%

\subsection{Invariants and coinvariants}
\label{ss:inco}
Let $\caO$ be a profinite group, and let $M$ be a discrete left $\QO$-module.
The $\Q$-vector space
\begin{equation}
\label{eq:invM}
M^\caO=\{\,m\in M\mid g\cdot m=m\ \text{for all $g\in\caO$}\,\}\subseteq M
\end{equation} 
is called the {\it space of $\caO$-invariants} of $M$, and 
\begin{equation}
\label{eq:coinv}
M_{\caO}=\Q\otimes_\caO M,
\end{equation}
is called the {\it space of $\caO$-coinvariants} of $M$. By definition, one has
a canonical map
\begin{equation}
\label{eq:natcoinv}
\phi_{M,\caO}\colon M^\caO\longrightarrow M_\caO
\end{equation}
given by $\phi_{M,\caO}(m)=1\otimes_{\caO} m$ for $m\in M$.
One has the following.

\begin{prop}
\label{prop:phinat}
Let $\caO$ be a profinite group,
and let $M$ be a discrete left $\QO$-module. Then $\phi_{M,\caO}$ is an isomorphism.
Moreover, if $U\subseteq\caO$ is an open subgroup, one has commutative diagrams
\begin{equation}
\label{eq:coinvsquare}
\xymatrix@C=1.5truecm{
M^\caO\ar[r]^{\phi_{M,\caO}}\ar[d]& M_\caO\ar[d]^{\rho_{\caO,U}}\\
M^U\ar[r]^{\phi_{M,U}}& M_U}
\qquad\qquad
\xymatrix@C=1.5truecm{
M^\caO\ar[r]^{\phi_{M,\caO}}& M_\caO\\
M^U\ar[r]^{\phi_{M,U}}\ar[u]^{\lambda_{U,\caO}}& M_U\ar[u]}
\end{equation}
where 
\begin{align}
\lambda_{U,\caO}(m)&=\frac{1}{|\caO:U|}\sum_{r\in \caR} r\cdot m,\label{eq:lambdadef}\\
\rho_{\caO,U}(1\otimes_\caO n)&=\frac{1}{|\caO:U|}\sum_{r\in\caR} 1\otimes_U r^{-1}\cdot n,\label{eq:mudef}
\end{align}
$m\in M^U$, $n\in N$, and $\caR\subseteq\caO$ is any set of coset representatives of $\caO/U$.
\end{prop}

\begin{proof}
The space of $\caO$-invariants $M^\caO\subseteq M$ is a discrete left $\QO$-submodule of $M$.
Hence, by Claim~\ref{cl:prf3}, $M=M^\caO\oplus C$ for some discrete left $\QO$-submodule $C$ of $M$.
Since $C$ is a discrete left $\QO$-module, one has $C=\varinjlim_{i\in I} C_i$, where
$C_i$ are finite-dimensional discrete left $\QG$-submodules of $C$.
By construction, $C_i^{\caO}=0$, and thus, by Maschke's theorem, $(C_i)_{\caO}=0$.
Since $\Q\otimes_{\caO}\argu$ commutes with direct limits, this implies that $C_\caO=0$.
Hence $C$ is contained in the kernel of the canonical map $\tau_M\colon M\to M_{\caO}$.
Let $\phi_{M,\caO}=(\tau_M)\vert_{M^\caO}$ denote the restriction of $\tau_M$ to the $\QO$-submodule
$M^{\caO}$. Since $\dH_1(\caO,\argu)=0$ (cf. Prop~\ref{prop:cd}(a), Prop.~\ref{prop:dhom}(a)),
the long exact sequence for $\dH_\bullet(\caO,\argu)$ implies that
the map $(M^\caO)_{\caO}\to M_{\caO}$ is injective.
As $\tau_{M^\caO}\colon M^\caO\to (M^\caO)_{\caO}$ is a bijection, the commutativity of the diagram
\begin{equation}
\label{eq:diainv}
\xymatrix@C=1.5truecm{
M^\caO\ar[r]\ar[d]_{\tau_{M^\caO}}\ar[dr]^{\phi_{M,\caO}}&M\ar[d]^{\tau_M}\\
(M^\caO)_{\caO}\ar[r]&M_{\caO}}
\end{equation}
yields that $\phi_{M,\caO}$ is injective. 
As $\tau_M$ is surjective and $C\subseteq\kernel(\tau_M)$,
one concludes that $\phi_{M,\caO}$ is bijective, and $C=\kernel(\tau_M)$.
The commutativity of the diagrams \eqref{eq:coinvsquare} can be verified by a straightforward calculation.
\end{proof}

%%%%%%

\subsection{The rational discrete standard bimodule $\biB(G)$}
\label{ss:bimo} 
For a t.d.l.c. group $G$ the set of compact open subgroups
$\CO(G)$ of $G$ with the inclusion relation "$\subseteq$" is a directed set, 
i.e., $(\CO(G),\subseteq)$ is a partially ordered set, and for 
$U,V\in\CO(G)$ one has also $U\cap V\in \CO(G)$.
For $U,V\in\CO(G)$, $V\subseteq U$, one has an injective mapping of discrete left $\Q[U]$-modules
\begin{equation}
\label{eq:biB1}
\teta_{U,V}\colon\Q\longrightarrow\Q[U/V],\qquad\teta_{U,V}(1)=\frac{1}{|U:V|}\sum_{r\in\caR} r\, V,
\end{equation}
where $\caR\subset U$ is any set of coset representatives of $U/V$.
As $\idn_V^G=\idn_U^G\circ\idn_V^U$, $\teta_{U,V}$ induces an injective mapping
\begin{equation}
\label{eq:biB2}
\eta_{U,V}\colon\Q[G/U]\longrightarrow\Q[G/V],\ \ \ 
\eta_{U,V}(x\,U)=\frac{1}{|U:V|}\sum_{r\in\caR} x\, r\, V,\ \ \ x\in G.
\end{equation}
By construction, one has for $W\in\CO(G)$, $W\subseteq V\subseteq U$, that 
$\eta_{U,W}=\eta_{V,W}\circ\eta_{U,V}$.
For $g\in G$ one has an isomorphism of discrete left $\QG$-modules
\begin{equation}
\label{eq:biB3}
c_{U,g}\colon \Q[G/U]\longrightarrow\Q[G/U^g],\qquad c_{U,g}(x\,U)= x\,g\,U^g,
\qquad x\in G,
\end{equation}
where $U^g=g^{-1} U g$.
Moreover, for $g,h\in G$ and $U,V\in\CO(G)$, $V\subseteq U$, one has
\begin{align}
c_{V,g}\circ\eta_{U,V}&=\eta_{U^g,V^g}\circ c_{U,g},\label{eq:biB4}\\
c_{U^g,h}\circ c_{U,g}&=c_{U,gh}.\label{eq:biB5}
\end{align}
Let $\biB[G]=\varinjlim_{U\in\CO(G)} (\Q[G/U],\eta_{U,V})$.
Then, by definition, $\biB[G]$ is a discrete left $\QG$-module.
By \eqref{eq:biB4}, the assignment
\begin{equation}
\label{eq:biB6}
(xU)\cdot g= c_{U,g}(x\,U)=xg\,U^g, \ \ g,x\in G,\ \ U\in\CO(G),
\end{equation}
defines a $\Q$-linear map $\argu\cdot\argu\colon \biB(G)\times \Q[G]\to\biB(G)$, and, by \eqref{eq:biB5},
this map defines a right $\QG$-module structure on $\biB(G)$.
By \eqref{eq:biB6}, $\biB(G)$ is a discrete right $\QG$-module making $\biB(G)$
a rational discrete $\QG$-bimodule. Therefore, we will call $\biB(G)$
the {\it rational discrete standard bimodule} of $G$. It has the following straightforward
properties.

\begin{prop}
\label{prop:biG}
Let $G$ be a t.d.l.c. group. 
\begin{itemize}
\item[(a)] $\biB(G)$ is a flat rational discrete right $\QG$-module, and a flat
rational discrete left $\QG$-module, i.e.,
\begin{equation}
\label{eq:dTorbiB}
\dTor_k^G(A,\biB(G))=\dTor_k^G(\biB(G),B)=0
\end{equation}
for all $A\in\ob(\disQG)$, $B\in \ob(\QGdis)$ and $k\geq 1$.
\item[(b)] One has
\begin{equation}
\label{eq:invbiB}
\Hom_G(\Q,\biB(G))\simeq
\begin{cases}
\Q&\ \text{if $G$ is compact,}\\
0&\ \text{if $G$ is not compact.}
\end{cases}
\end{equation}
\end{itemize}
\end{prop}

\begin{proof}
(a) The direct limit of flat objects is flat (cf. \cite[\S VIII.4, p.~195]{brown:coh}).
Thus, by Proposition~\ref{prop:dhom}(b), $\biB(G)$ is a flat rational discrete left $\QG$-module.
A similar argument shows that $\biB(G)$ is also a flat rational discrete right $\QG$-module.

\noindent
(b) is a direct consequence of Proposition~\ref{prop:ccom}.
\end{proof}

Let $M\in\ob(\QGdis)$ and $U\in\CO(G)$.
The composition of maps (cf. Prop.~\ref{prop:phinat})
\begin{equation}
\label{eq:biB7}
\xymatrix@C=2truecm{
\theta_{M,U}\colon \Q[U\backslash G]\otimes_G M\ar[r]^-{\xi_{M,U}}&
\Q\otimes_U M\ar[r]^-{\phi^{-1}_{M,U}}&M^U},
\end{equation}
where $\xi_{M,U}(Ug\otimes_G m)=1\otimes_U gm$,
$g\in G$, $m\in M$, is an isomorphism. Moreover, if $V\in\CO(G)$, $V\subseteq U$,
one has by \eqref{eq:coinvsquare} a commutative diagram
\begin{equation}
\label{eq:biB8}
\xymatrix@C=2truecm{
\Q[U\backslash G]\otimes_G M\ar[r]^-{\theta_{M,U}}\ar[d]_{\nu_{U,V}\otimes_G\iid_M}& M^U\ar[d]\\
\Q[V\backslash G]\otimes_G M\ar[r]^-{\theta_{M,V}}& M^V}
\end{equation}
which is natural in $M$, where  $\nu_{U,V}\colon \Q[U\backslash G]\to \Q[V\backslash G]$
is given by 
\begin{equation}
\label{eq:biB9}
\nu_{U,V}(U x)=\frac{1}{|U:V|}\sum_{r\in\caR} V r^{-1} x,
\end{equation} 
$x\in G$, and $\caR\subseteq U$
is a set of coset representatives of $U/V$, i.e., $U=\bigsqcup_{r\in\caR} r V$.
Thus the family of maps $(\theta_{M,U})_{U\in\CO(G)}$ induces an isomorphism
\begin{equation}
\label{eq:biB10}
\theta_M\colon\biB(G)\otimes_G M\longrightarrow M
\end{equation}
which is natural in $M$. Note that the map $\theta_M$ can be described quite explicitly.
For $g\in G$ and $m\in M$ the subgroup $W=\stab_G(g\cdot m)$ is open.
In particular, for $U\in\CO(G)$ the intersection $U\cap W$
is of finite index in $U$. Let $U=\bigsqcup_{r\in\caR} r\cdot (U\cap W)$. Then
\begin{equation}
\label{eq:thetex}
\theta_M(Ug\otimes_G m)=\frac{1}{|U:(U\cap W)|}\cdot \sum_{r\in\caR} r\cdot m.
\end{equation}
From the commutative diagram \eqref{eq:biB8} 
one concludes the following property.

\begin{prop}
\label{prop:fundbi}
Let $G$ be a t.d.l.c. group.
Then one has a natural isomorphism
$\theta\colon \biB(G)\otimes_G\argu\longrightarrow\iid_{\QGdis}$.
\end{prop}

\begin{rem}
\label{rem:gralg}
Let $G$ be a t.d.l.c. group. One verifies easily using \eqref{eq:thetex} that
\begin{equation}
\label{eq:hopf1}
\argu\cdot\argu=\theta_{\biB(G)}\colon\biB(G)\otimes_G\biB(G)\to\biB(G)
\end{equation}
defines the structure of an associative algebra on $(\biB(G),\cdot)$.
However, $(\biB(G),\cdot)$ contains a unit $1_{\biB(G)}\in\biB(G)$ if, and only
if, $G$ is compact. 
Nevertheless, the algebra $(\biB(G),\cdot)$ has an {\it augmentation} or {\it co-unit}, i.e.,
the canonical map 
\begin{equation}
\label{eq:hopf2}
\eps\colon\biB(G)\to\Q,\qquad 
\eps(xU)=1,\qquad x\in G,\ U\in\CO(G)
\end{equation}
is a mapping of $\QG$-bimodules satisfying
$\eps(a\cdot b)=\eps(a)\cdot\eps(b)$ for all $a,b\in\biB(G)$.
The algebra $\biB(G)$ comes also equipped with an {\it antipode}
\begin{equation}
\label{eq:hopf3}
S=\argu^\times\colon\biB(G)\to\biB(G),\qquad S(Ug)=g^{-1}U,\qquad g\in G,\ U\in\CO(G).
\end{equation}
Again by \eqref{eq:thetex} one verifies also that $\theta_M\colon \biB(G)\otimes_G M\to M$
defines the structure of a $\biB(G)$-module for any rational discrete $\QG$-module $M$.
\end{rem}

%%%%%

\subsection{The $\Hom$-$\otimes$-identity}
\label{ss:homten}
Let $G$ be a t.d.l.c. group, and let $Q,M\in\ob(\QGdis)$. Then one has
a canonical map (cf. \eqref{eq:biB10})
\begin{equation}
\label{eq:nathom}
\begin{gathered}
\zeta_{Q,M}\colon \Hom_G(Q,\biB(G))\otimes_G M\longrightarrow \Hom_G(Q,M),\\
\zeta_{Q,M}(h\otimes_G m)(q)=\theta_M(h(q)\otimes_G m),
\end{gathered}
\end{equation}
for $h\in\Hom_G(Q,\biB(G))$, $q\in Q$, $m\in M$.
Let $U\subseteq G$
be a compact open subgroup of $G$. From now on we will assume that the 
canonical embedding
\begin{equation}
\label{eq:homt1}
\eta_U\colon\Q[G/U]\longrightarrow\biB(G)
\end{equation}
is given by inclusion.
By the Eckmann-Shapiro-type lemma (cf. \S\ref{ss:ecksha}), the map
\begin{equation}
\label{eq:ecky1}
\hat{\argu}\colon\Hom_G(\Q[G/U],\biB(G))\longrightarrow\biB(G),\ \ \
\hat{h}=h(U),
\end{equation}
$h\in \Hom_G(\Q[G/U],\biB(G))$,
is injective satisfying 
\begin{equation}
\label{eq:ecky2}
\image(\hat{\argu})=\spn_{\Q}\{\,x\,U^x\mid x\in G\,\}\subseteq\biB(G).
\end{equation}
In particular, one has an isomorphism 
\begin{equation}
\label{eq:ecky3}
j_U\colon \Hom_G(\Q[G/U],\biB(G))\to\Q[U\backslash G]
\end{equation}
induced by $\hat{\argu}$.  Hence, by 
the Eckmann-Shapiro-type lemma (cf. \S\ref{ss:ecksha}), \eqref{eq:biB7},
and the commutativity of the diagram
\begin{equation}
\label{dia:homten}
\xymatrix{
\Hom_G(\Q[G/U],\biB(G))\otimes_G M
\ar[0,2]^-{\zeta_{\Q[G/U],M}}
\ar[d]_{j_U\otimes_G\iid_M}&&\Hom_G(\Q[G/U],M)\\
\Q[U\backslash G]\otimes_G M\ar[r]&M_U\ar[r]^-{\phi_{M,U}^{-1}}&M^U\ar[u]\\
}
\end{equation}
one concludes that $\zeta_{\Q[G/U],M}$ is an isomorphism
for all $M\in\ob(\QGdis)$. The additivity $\Hom_G(\argu,M)$ yields the following.

\begin{prop}
\label{prop:homten}
Let $G$ be a t.d.l.c. group, let $M, P\in\ob(\QGdis)$, and
assume further that $P$ is finitely generated and projective.
Then 
\begin{equation}
\label{eq:homten}
\zeta_{P,M}\colon\Hom_G(P,\biB(G))\otimes_G M\longrightarrow \Hom_G(P,M)
\end{equation}
is an isomorphism.
\end{prop}

%%%%%

\subsection{Dualizing finitely generated projective rational discrete $\QG$-mo\-du\-les}
\label{ss:dualcom}
From \eqref{eq:ecky3} one concludes that for any finitely generated projective
rational discrete left $\QG$-module $P$, $\Hom_G(P,\biB(G))$ is
a finitely generated projective rational discrete right $\QG$-module. In order 
to obtain a functor $\argu^\circledast\colon\QGdis^{\op}\to\QGdis$ we put for
$M\in\ob(\QGdis)$
\begin{equation}
\label{eq:ecky4}
\begin{aligned}
M^\circledast&=\Hom_G^\times(M,\biB(G))\\
&=
\{\,\varphi\in\Hom_{\Q}(M,\biB(G))\mid\forall g\in G\ \forall m\in M: \varphi(g\cdot m)=m\cdot g^{-1}\,\}.
\end{aligned}
\end{equation}
Obviously, $\argu^\times\circ\argu\colon \Hom_G(M,\biB(G))\to M^\circledast$ (cf. \eqref{eq:hopf3}) yields an isomorphism
of $\Q$-vector spaces. 
Note that for $g\in G$, $\varphi\in M^\circledast$ and $m\in M$ one has
\begin{equation}
\label{eq:ecky4.1}
(g\cdot\varphi)(m)=g\cdot\varphi(m).
\end{equation}
By construction, if $P$ is a finitely generated projective
rational discrete left $\QG$-module, then $P^\circledast$ is a finitely generated projective
rational discrete left $\QG$-module as well.
Let $M\in\ob(\QGdis)$, let $m\in M$ and let $\varphi\in M^\circledast$. Put
\begin{equation}
\label{eq:natd1}
\vartheta_M(m)(\varphi)=\varphi(m)^\times.
\end{equation}
Then for $g\in G$ one has
\begin{align}
\vartheta_M(m)(g\cdot\varphi)&=((g\cdot\varphi)(m))^\times=(g\cdot \varphi(m))^\times\notag\\
&=\varphi(m)^\times\cdot g^{-1}=\vartheta_M(m)(\varphi)\cdot g^{-1},\label{eq:natd2}
\end{align}
i.e., $\vartheta_M(m)\in M^{\circledast\circledast}$, and thus one has a $\Q$-linear map
$\vartheta_M\colon M\to M^{\circledast\circledast}$.
As
\begin{align}
\vartheta_M(g\cdot m)(\varphi)&=\varphi(g\cdot m)^\times=(\varphi(m)\cdot g^{-1})^\times\notag\\
&=g\cdot \varphi(m)^\times= (g\cdot \vartheta_M(m))(\varphi),\label{eq:natd3}
\end{align}
$\vartheta_M\colon M\to M^{\circledast\circledast}$ is a mapping of left $\QG$-modules.
It is straightforward to verify that one obtains a natural transformation
\begin{equation}
\label{eq:natd4}
\vartheta\colon \iid_{\QGdis}\longrightarrow \argu^{\circledast\circledast}.
\end{equation}
By \eqref{eq:ecky1} and \eqref{eq:ecky3} one has for $U\in\CO(G)$
that $\vartheta_{\Q[G/U]}(U)=j_U^\times$, where $j_U^\times$ is the map making
the diagram
\begin{equation}
\label{eq:natd5}
\xymatrix{
\Hom_G(\Q[G/U),\biB(G))\ar[r]^-{j_U}\ar[d]_{\argu^\times\circ\ldots}&\Q[U\backslash G]\ar[d]^{\argu^\times}\\
\Q[G/U]^\circledast\ar[r]^-{j_U^\times}&\biB(G)
}
\end{equation}
commute. This yields that $\vartheta_{\Q[G/U]}$ is an isomorphism. 
Hence the additivity of $\vartheta$ and the functors
$\iid_{\QGdis}$ and $\argu^{\circledast\circledast}$ imply that 
$\vartheta_P\colon P\to P^{\circledast\circledast}$ is an isomorphism
for every finitely generated projective rational discrete $\QG$-module $P$.

%%%%%%%%%%%

\subsection{T.d.l.c. groups of type $\FP$}
\label{ss:FP}
A t.d.l.c. group $G$ is said to be {\it of type $\FP$},
if 
\begin{itemize}
\item[(i)] $\ccd_\Q(G)=d<\infty$;
\item[(ii)] $G$ is of type $\FP_\infty$
\end{itemize}
(cf. \cite[\S VIII.6]{brown:coh}). By Lemma~\ref{lem:pdim} and Proposition~\ref{prop:FPn}, for such a group $G$
the trivial left $\QG$-module $\Q$ possesses a projective resolution $(P_\bullet,\der_\bullet,\eps)$
which is finitely generated and concentrated in degrees $0$ to $d$. One has the following property.

\begin{prop}
\label{prop:FP}
Let $G$ be a t.d.l.c. group of type $\FP$. Then
\begin{equation}
\label{eq:FP-1}
\ccd_\Q(G)=\max\{\,k\geq 0\mid \dH^k(G,\biB(G))\not=0\,\}.
\end{equation}
\end{prop}

\begin{proof}
Let $m=\max\{\,k\geq 0\mid \dH^k(G,\biB(G))\not=0\,\}$. Then $m\leq d=\ccd_\Q(G)$.
By definition, there exists a discrete left $\QG$-module $M$ such that $\dH^d(G,M)\not=0$.
As $G$ is of type $\FP_\infty$, the functor $\dH^d(G,\argu)$ commutes with direct limits 
(cf. \cite[\S VIII.4, Prop.~4.6]{brown:coh}).
Since $M=\varinjlim_{i\in I} M_i$, where $M_i$ are finitely generated $\QG$-submodules of $M$,
we may assume also that $M$ is finitely generated, i.e.,
there exist finitely many elements $m_1,\ldots, m_n\in M$ such that
$M=\sum_{1\leq j\leq n} \Q[G]\cdot m_j$. Let $\caO_j$ be a compact open subgroup contained
in $\stab_G(m_j)$. Then one has a canonical surjective map $\coprod_{1\leq j\leq n} \Q[G/\caO_j]\to M$.
Thus, as $\dH^d(G,\argu)$ is right exact, $\dH^d(G,\coprod_{1\leq j\leq n}\Q[G/\caO_j])\not=0$,
and therefore $\dH^d(G,\Q[G/\caO_j])\not=0$ for some element $j\in\{1,\ldots,d\}$.
Let $\caO=\caO_j$.
For any pair of compact open subgroups $U, V$ of $G$, $V\subseteq U$, which are contained in $\caO$ the map
$\eta_{U,V}\colon\Q[G/U]\to\Q[G/V]$ (cf. \eqref{eq:biB2}) is split injective (cf. Prop.~\ref{prop:prof}).
From the additivity of $\dH^d(G,\argu)$ one concludes that
$\dH^d(\eta_{U,V})\colon \dH^d(G,\Q[G/U])\to\dH^d(G,\Q[G/V])$ is injective. Hence
\begin{equation}
\label{eq:FP-2}
\textstyle{\dH^d(G,\biB(G))=\varinjlim_{U\in\CO_\caO(G)}\dH^d(G,\Q[G/U])\not=0,}
\end{equation}
i.e., $m\geq d$, and this yields the claim.
\end{proof}

For a t.d.l.c. group $G$ of type $\FP$ with $d=\ccd_\Q(G)$
we will call the rational discrete right $\QG$-module
\begin{equation}
\label{eq:dual}
D_G=\dH^d(G,\biB(G))
\end{equation}
the {\it rational dualizing module} of $G$.
It has the following fundamental property.

\begin{prop}
\label{prop:dual1}
Let $G$ be a t.d.l.c. group of type $\FP$ with $d=\ccd_\Q(G)$.
Then there exists a canonical isomorphism
\begin{equation}
\label{eq:dual2}
\upsilon\colon \dH^d(G,\argu)\longrightarrow D_G\otimes_G\argu
\end{equation}
of covariant additive right exact functors.
\end{prop}

\begin{proof}
Let $(P_\bullet,\der_\bullet^P,\eps^\Q)$ be a projective resolution of the
trivial left $\QG$-module $\Q$, which is finitely generated
and concentrated in degrees $0$ to $d$, and let $(Q_\bullet,\eth_\bullet,\eps^{D_G})$
be a projective resolution of the discrete right $\QG$-module $D_G$.
Let $(C_\bullet,\der_\bullet^C)$ be the chain complex of discrete right $\QG$-modules
given by
\begin{equation}
\label{eq:dual3}
C_k=\Hom_G(P_{d-k},\biB(G)),\ \ \ \der_k^C=\Hom_G(\der_{d-k+1},\biB(G))\colon C_k\to C_{k-1},\ \ \ 0\leq k\leq d,
\end{equation}
(cf. \eqref{eq:ecky4}). In particular, by construction, $H_0(C_\bullet,\der_\bullet^C)$ is canonically
isomorphic to $D_G$. Since $(C_\bullet,\der_\bullet)$ is a chain complex of projectives,
and as $(Q_\bullet,\eth_\bullet,\eps^{D_G})$ is exact, the comparison theorem in homological algebra implies that
there exists a mapping of chain complexes $\chi_\bullet\colon (C_\bullet,\der_\bullet^C)\to (Q_\bullet,\eth_\bullet)$
which is unique up to chain homotopy equivalence. For $M\in\ob(\QGdis)$, $\chi_\bullet$ induces a chain map
\begin{equation}
\label{eq:dual4}
\chi_{\bullet,M}=\chi_\bullet\otimes_G\iid_M\colon (C_\bullet\otimes_G M,\der_\bullet^C\otimes_G\iid_M)
\longrightarrow (Q_\bullet\otimes_G M,\eth_\bullet\otimes_G\iid_M).
\end{equation}
The right exactness of $\argu\otimes_G M$ implies that
\begin{equation}
\label{eq:dual5}
H_0(\chi_{\bullet,M})\colon H_0(C_\bullet\otimes_G M)\longrightarrow D_G\otimes_G M
\end{equation}
is an isomorphism. By the $\Hom$-$\otimes$ identity \eqref{eq:homten}, one has a natural isomorphism
of chain complexes 
\begin{equation}
\label{eq:dual6}
C_\bullet\otimes_G M\simeq \Hom_G(P_{d-\bullet},M).
\end{equation}
Hence $H_0(C_\bullet\otimes_G M)\simeq \dH^d(G,M)$, and this yields the claim.
\end{proof}

\begin{rem}
\label{rem:dual}
A straightforward modification of the proof of Proposition~\ref{prop:dual1}
shows that for a t.d.l.c. group $G$ of type $\FP$ one has natural isomorphisms
\begin{equation}
\label{eq:dual7}
\bar{\chi}_{\bullet,\argu}\colon \dH^{d-\bullet}(G,\argu)\longrightarrow
\underline{\dTor}^G_\bullet(C_\bullet,\argu\dbl 0\dbr),
\end{equation}
where $\underline{\dTor}^G_\bullet(\argu,\argu)$ denotes the {\it hyper-homology}
(cf. \cite[\S 2.7]{ben:coho1}, \cite[\S VIII.4, Ex.~7]{brown:coh}).
\end{rem}

%%%%

\subsection{Rational duality groups}
\label{ss:dual}
A t.d.l.c. group $G$ is said to be a {\it rational duality
group of dimension $d\geq 0$}, if
\begin{itemize}
\item[(i)] $G$ is of type $\FP$;
\item[(ii)] $\dH^k(G,\biB(G))=0$ for all $k\not=d$.
\end{itemize}
By Proposition~\ref{prop:FP}, for such a group one has $d=\ccd_\Q(G)$.
These groups have the following fundamental property.

\begin{prop}
\label{prop:dual}
Let $G$ be t.d.l.c. group which is a rational duality group of di\-men\-sion $d\geq 0$,
and let $D_G$ denote its rational dualizing module. Then one has
natural isomorphisms of (co)homological functors
\begin{equation}
\label{eq:dual8}
\begin{aligned}
\dH^\bullet(G,\argu)&\simeq\dTor_{d-\bullet}^G(D_G,\argu),\\
\dH_\bullet(G,\argu)&\simeq\dExt^{d-\bullet}_G({}^\times D_G,\argu).
\end{aligned}
\end{equation}
\end{prop}

\begin{proof}
Let $(P_\bullet,\der_\bullet,\eps^\Q)$ be a projective resolution of the
trivial left $\QG$-module $\Q$, which is finitely generated
and concentrated in degrees $0$ to $d$. Then
\begin{equation}
\label{eq:dual9}
Q_k=\Hom_G(P_{d-k},\biB(G)),\ \ 
\eth_k=\Hom_G(\der_{d-k+1},\biB(G)),\ \ k\in\{0,\ldots,d\},
\end{equation}
is a chain complex of projective rational discrete right $\QG$-modules 
concentrated in degrees $0,\ldots, d$ which satisfies 
\begin{equation}
\label{eq:dual10}
H_k(Q_\bullet,\eth_\bullet)\simeq
\begin{cases}
D_G &\ \text{for $k=0$,}\\
\hfil 0\hfil &\ \text{for $k\not=0$,}
\end{cases}
\end{equation}
By Proposition~\ref{prop:homten}, one has for $M\in\ob(\QGdis)$ isomorphisms
\begin{equation}
\label{eq:dual11}
H_k(Q_\bullet\otimes_G M,\eth_\bullet\otimes_G\iid_M)\simeq
\dTor_k^G(D_G,M)\simeq \dExt_G^{d-k}(\Q,M)
\end{equation}
which are natural in $M$. This yields the first natural isomorphism in \eqref{eq:dual8}.

Let $(P_\bullet^\circledast[d],\der_\bullet^\circledast[d])$
be the chain complex $(P_\bullet^\circledast,\der_\bullet^\circledast)$
concentrated in homological degrees $-d,\ldots, -0$
moved $d$-places to the left such that it is
concentrated in homological degrees $0,\ldots, d$, i.e.,
\begin{equation}
\label{eq:dual12}
H_k(P_\bullet^\circledast[d],\der_\bullet^\circledast[d])\simeq
\begin{cases}
{}^\times D_G &\ \text{for $k=0$,}\\
\hfil 0\hfil &\ \text{for $k\not=0$.}
\end{cases}
\end{equation}
In particular, $(P_\bullet^\circledast[d],\der_\bullet^\circledast[d])$ is a projective
resolution of the rational discrete left $\QG$-module ${}^\times D_G$.
Then, by \eqref{eq:natd4} and the remarks following it, one has
an isomorphism of chain complexes of projective rational discrete left $\QG$-modules
\begin{equation}
\label{eq:dual13}
(P_\bullet^\circledast[d]^\circledast[d],\der_\bullet^\circledast[d]^\circledast[d])\simeq (P_\bullet,\der_\bullet).
\end{equation}
By the same arguments as used before, one obtains 
for $M\in\ob(\QGdis)$ and 
\begin{equation}
\label{eq:dual14}
R_k=\Hom_G(P^\circledast[d]_{d-k},\biB(G)),\ 
\delta_k=\Hom_G(\der^\circledast[d]_{d-k+1},\biB(G)),\ k\in\{0,\ldots,d\},
\end{equation}
isomorphisms
\begin{equation}
\label{eq:dual15}
H_k(R_\bullet\otimes_G M,\delta_\bullet\otimes_G\iid_M)\simeq
\dTor_k^G(\Q,M)\simeq \dExt_G^{d-k}({}^\times D_G,M)
\end{equation}
which are natural in $M$. This yields the claim.
\end{proof}

%%%%

\subsection{Locally constant functions with compact support}
\label{ss:lcf}
Let $G$ be a t.d.l.c. group, and let
$\caC(G,\Q)$ denote the $\Q$-vector space of continuous
functions from $G$ to $\Q$, where $\Q$ is considered as a discrete topological space. Then $\caC(G,\Q)$ is a 
$\Q[G]$-bimodule, where the $G$-actions are given by
\begin{equation}
\label{eq:biB11}
(g\cdot f)(x)=f(g^{-1}\,x),\ \ 
(f\cdot g)(x)=f(x\,g^{-1}),\ \ g,x\in G,\ \ f\in \caC(G,\Q).
\end{equation}
For any compact open set $\Omega\subseteq G$ let $I_\Omega\colon G\to\Q$
denote the continuous function given by $I_\Omega(x)=1$ for $x\in\Omega$
and $I_\Omega(x)=0$ for $x\in G\setminus\Omega$.
Then $g\cdot I_\Omega=I_{g\,\Omega}$ and $I_\Omega\cdot g=I_{\Omega\,g}$. Moreover,
\begin{equation}
\label{eq:biB12}
\caC_c(G,\Q)=\spn_{\Q}\{\,I_{gU}\mid g\in G,\,U\in\CO(G)\,\}\subseteq \caC(G,\Q),
\end{equation}
coincides with the {\it set of locally constant functions from $G$ to $\Q$ with compact support}.
It is a rational discrete $\QG$-bisubmodule of $\caC(G,\Q)$.

The rational discrete left $\QG$-module $\biB(G)$ is isomorphic to 
the left $\QG$-module $\caC_c(G,\Q)$, but the isomorphism is not canonical.
Let $\caO\subseteq G$ be a fixed compact open subgroup of $G$, and put
$\CO_\caO(G)=\{\,U\in\CO(G)\mid U\subseteq \caO\,\}$.
The
$\Q$-linear map $\psi_U^{(\caO)}\colon\Q[G/U]\to \caC_c(G,\Q)$ given by
\begin{equation}
\label{eq:biB13}
\psi_U^{(\caO)}(xU)=|\caO: U|\cdot I_{xU}
\end{equation}
is an injective homomorphism of rational discrete left $\QG$-modules.
For $U,V\in \CO_\caO(G)$, $V\subseteq U$, one has $\psi_V^{(\caO)}\circ\eta_{U,V}=\psi^{(\caO)}_U$
(cf. \eqref{eq:biB2}), and thus $(\psi_U^{(\caO)})_{U\in\CO_{\caO}(G)}$ induces
an isomorphism of left $\QG$-modules 
\begin{equation}
\label{eq:biB14}
\psi^{(\caO)}\colon\biB(G)\to \caC_c(G,\Q),\qquad \psi^{(\caO)}\vert_{\Q[G/U]}=\psi_U^{(\caO)}.
\end{equation} 
This isomorphism has the following properties.

\begin{prop}
\label{prop:compar}
Let $G$ be a t.d.l.c. group with modular function $\Delta\colon G\to \Q^+$, i.e.,
if $\mu\colon \Bor(G)\to\R_0^+\cup\{\infty\}$ is a left invarinat Haar measure on $G$,
one has $\mu(S\cdot g)=\Delta(g)\cdot \mu(S)$ for all $S\in\Bor(G)$. Let $\caO\in\CO(G)$.
\begin{itemize}
\item[(a)]  If $\caU$ is a compact open subgroup of $G$ containing $\caO$, then
\begin{equation}
\label{eq:biB15}
\psi^{(\caU)}=|\caU:\caO|\cdot \psi^{(\caO)}.
\end{equation}
\item[(b)] For all $g,h\in G$ and $u\in\biB(G)$ one has
\begin{equation}
\label{eq:biB16}
\psi^{(\caO)}(g\cdot W\cdot h)= \Delta(h^{-1})\cdot g\cdot \psi^{(\caO)}(W)\cdot h.
\end{equation}
\end{itemize}
\end{prop}

\begin{proof}
(a) Let $g\in G$ and $W\in\CO(G)$. Then, by subsection~\ref{ss:bimo},
\begin{equation}
\label{eq:bimo1}
gW=\frac{1}{|W:W\cap\caO|}\cdot \sum_{r\in\caR} g\cdot r\cdot(W\cap\caO).
\end{equation}
Here we omitted the mappings $\eta_{W,W\cap\caO}$ in the notation.
Hence
\begin{align}
\psi^{(\caO)}(gW)&=\frac{|\caO:W\cap\caO|}{|W:W\cap\caO|}\cdot I_{gW}
=\frac{\mu(\caO)}{\mu(W)}\cdot I_{gW},\label{eq:bimo2}\\
\intertext{for some left invariant Haar measure $\mu$ on $G$, and}
\psi^{(\caU)}(gW)&=\frac{\mu(\caU)}{\mu(W)}\cdot I_{gW}.\label{eq:bimo3}
\end{align}
As $\mu(\caU)=|\caU:\caO|\cdot \mu(\caO)$, this yields the claim.

\noindent
(b) As $\psi^{(\caO)}$ is a homomorphism of left $\QG$-modules, it suffices
to prove the claim for $g=1$ and $u=W\in\CO(G)$.
If $W^h=h^{-1}\cdot W\cdot h=\bigsqcup_{s\in\caS} s\cdot (W^h\cap\caO)$, one obtains by (a) that
\begin{align}
\psi^{(\caO)}(Wh)&=h\cdot \psi^{(\caO)}(W^h)\notag\\
&=h\cdot \frac{1}{|W^h:W^h\cap\caO|}\cdot \sum_{s\in\caS}  \psi^{(\caO)}(s\cdot (W^h\cap\caO)).\notag\\
&=\frac{|\caO:W^h\cap\caO|}{|W^h:W^h\cap\caO|}\cdot I_{Wh}=\frac{\mu(\caO)}{\mu(W^h)}\cdot I_{Wh}.\label{eq:bimo5}
\end{align}
In particular, $\psi^{(\caO)}(Wh)=\mu(W)/ \mu(W^h)\cdot \psi^{(\caO)}(W)\cdot h$.
Since
\begin{equation}
\label{eq:bimo6}
\mu(W^h)=\mu(h^{-1}Wh)=\mu(Wh)=\Delta(h)\cdot \mu(W),
\end{equation}
this completes the proof.
\end{proof}

L
\begin{rem}
\label{rem:cosualg}
With a particular choice of Haar measure $\mu\colon \Bor(G)\to\R_0^+\cup\{\infty\}$ on $G$ one can make 
$(\caC_c(G,\Q),\ast_\mu)$ an associative algebra where $\ast_\mu$ is convolution with respect to $\mu$.
However, we have seen that $\biB(G)$ carries an algebra structure which is independent
of the choosen Haar measure $\mu$ (cf. Remark~\ref{rem:gralg}).
Both rational discrete $\QG$-bimodules $\biB(G)$ and $\caC_c(G,\Q)$ will turn out to be useful. 
The standard rational $\QG$-bimodule $\biB(G)$ seem to be the canonical choice for studying
$\Hom$-$\otimes$ identities or dualizing functors (cf. \S\ref{ss:homten}, \S\ref{ss:dualcom});
while $\caC_c(G,\Q)$ seem to be the right choice for relating the cohomology groups
$\dH^\bullet(G,\caC_c(G,\Q))$ with the cohomology with compact support of certain topological spaces
associated to $G$ (cf. \S\ref{ss:rattop}).
\end{rem}

\begin{rem}
\label{rem:biC}
Let $G$ be a t.d.l.c. group, and let $\Delta\colon G\to\Q^+$ denote its modular function.
Let $\Q(\Delta)$ denote the rational discrete left $\QG$-module which is isomorphic 
- as abelian group - to $\Q$ and which $G$-action is given by 
\begin{equation}
\label{eq:biC1}
g\cdot q=\Delta(g)\cdot q,\qquad g\in G,\ q\in\Q(\Delta).
\end{equation}
Let $\Q(\Delta)^{\times}$ denote the rational discrete right $\QG$-module associated to $\Q(\Delta)$,
i.e., $\Q(\Delta)^{\times}=\Q(\Delta)$ and for $g\in G$ and $q\in\Q(\Delta)^{\times}$ one has $q\cdot g=g^{-1}\cdot q$.
Then, by Proposition~\ref{prop:compar}(b), one has a (non-canonical) isomorphism
\begin{equation}
\label{eq:biC2}
\biB(G)\simeq_G\caC_c(G)\otimes\Q(\Delta)^\times
\end{equation}
of rational discrete right $\QG$-modules. Considering $\Q(\Delta)^\times$ as a trivial left $\QG$-module,
the isomorphism \eqref{eq:biC2} can be interpreted as an isomorphism of rational discrete $\QG$-bimodules.
\end{rem}

%%%%%

\subsection{The trace map}
\label{ss:trace}
For a compact open subgroup $\caO$ of $G$
let $\mu_\caO$ denote the left-invariant Haar measure on $G$ satisfying $\mu_\caO(\caO)=1$, i.e.,
if $\caU$ is a compact open subgroup of $G$ containing $\caO$, then
$\mu_\caO=|\caU\colon\caO|\cdot\mu_{\caU}$.
We also denote by 
\begin{equation}
\label{eq:defH}
\boh(G)=\Q\cdot\mu_\caO
\end{equation}
the $1$-dimensional $\Q$-vector space generated by all Haar measures $\mu_\caO$, $\caO\in\CO(G)$.
One has a $\Q$-linear map
\begin{equation}
\label{eq:trace1}
\tr=\psi^{(\caO)}(\argu)(1)\cdot\mu_{\caO}\colon\biB(G)\longrightarrow\boh(G),
\end{equation}
which is independent of the choice of the compact open subgroup $\caO$ of $G$, i.e., 
for all $u\in\biB(G)$ and all $\caO,\caU\in\CO(G)$ one has
\begin{equation}
\label{eq:trace2}
\tr(u)=\psi^{(\caO)}(u)(1)\cdot\mu_{\caO}=\psi^{(\caU)}(u)(1)\cdot\mu_{\caU}
\end{equation}
(cf. \eqref{eq:biB15}). In particular, by definition,
\begin{equation}
\label{eq:trace3}
\tr(\caO g)=
\begin{cases}
\mu_\caO&\ \text{if $\caO g=\caO$,}\\
\hfil 0\hfil&\ \text{if $\caO g\not=\caO$.}
\end{cases}
\end{equation}
In case that $G$ is unimodular one has the following.

\begin{prop}
\label{prop:trace}
Let $G$ be a unimodular t.d.l.c. group. Then one has
\begin{equation}
\label{eq:trace4}
\tr(g\cdot u)=\tr(u\cdot g)\qquad\text{for all $u\in\biB(G)$ and $g\in G$.}
\end{equation}
In particular, putting $\ubG=\biB(G)/\langle\, g\cdot u-u\cdot g\mid u\in\biB(G),\ g\in G\,\rangle_{\Q}$,
the $\Q$-linear map $\tr$ induces a $\Q$-linear map $\utr\colon\ubG\longrightarrow \boh(G)$.
\end{prop}

\begin{proof}
Since $(g\cdot f)(1)=(f\cdot g)(1)=f(g^{-1})$ for all $g\in G$ and $f\in C(G,\Q)$, one concludes from
\eqref{eq:biB16} that
\begin{equation}
\label{eq:trace5}
\tr(g\cdot u)-\tr(u\cdot g)=(\psi^{(\caO)}(u)(g^{-1})-\psi^{(\caO)}(u)(g^{-1}))\cdot\mu_\caO=0.
\end{equation}
This yields the claim.
\end{proof}

%%%%%%%%%%%%
% last update 2/3/2015
%%%%%%%%%%%%

\subsection{Homomorphism into the bimodule $\caC_c(G,\Q)$}
\label{ss:permbi} 
Let $G$ be a t.d.l.c. group.
In order to simplify notations we put 
$\caC_c(G)=\caC_c(G,\Q)$ 
(cf. \eqref{eq:biB12}).
For a rational discrete left $\QG$-module $M$ we put also 
\begin{equation}
\label{eq:defodot}
M^\odot=\Hom_G(M,\caC_c(G)),
\end{equation}
and consider $M^\odot$ as left $\QG$-module, i.e., for $g\in G$ and $h\in M^\odot$ one has
\begin{equation}
\label{eq:HomC1}
(g\cdot h)(m)= h(m)\cdot g^{-1}.
\end{equation}
One has a homomorphism of $\Q$-vector spaces
\begin{equation}
\label{eq:HomC2}
\argu_M^\vee\colon\Hom_G(M,\caC_c(G))\longrightarrow\Hom_{\Q}(M,\Q),
\end{equation}
which is given by evaluation in $1\in G$, i.e., for $m\in M$ and $h\in M^\odot$ one has
\begin{equation}
\label{eq:HomC3}
h^\vee_M(m)=h(m)(1).
\end{equation}
The $\Q$-vector space $M^\ast=\Hom_{\Q}(M,\Q)$ carries canonically the structure of a left
$\QG$-module, i.e., for $g\in G$ and $f\in M^\ast$ one has
\begin{equation}
\label{eq:HomC4}
(g\cdot f)(m)=f(g^{-1}\cdot m).
\end{equation}
However, $M^\ast$ is not necessarily a discrete $\QG$-module. For
$m\in M$, $g\in G$ and $h\in \Hom_G(M,\caC_c(G))$ one concludes from \eqref{eq:HomC1}
and \eqref{eq:HomC4} 
that
\begin{align}
(g\cdot h)^\vee_M(m)&=\big((g\cdot h)(m)\big)(1)=(h(m)\cdot g^{-1})(1)=h(m)(g)\notag\\
&=\big(g^{-1}\cdot h(m)\big)(1)=h(g^{-1}\cdot m)(1)=h^\vee_M(g^{-1}\cdot m),\label{eq:HomC5}
\end{align}
i.e., $\argu^\vee_M$ is a homomorphism of left $\QG$-modules.
This homomorphism has the following property.

\begin{prop}
\label{prop:mcom}
Let $G$ be a t.d.l.c. group, and let $M\in\ob(\QGdis)$.
Then $\argu^\vee_M\colon M^\odot\to M^\ast$ is injective.
\end{prop}

\begin{proof}
Let $h\in M^\odot$, $h\not=0$, i.e., there exists
$m\in M$ with $h(m)\not=0$. In particular, there exists $g\in G$ such that $h(m)(g)\not=0$.
Since 
\begin{equation}
\label{eq:inj1}
h(m)(g)= (g^{-1}\cdot h(m))(1)=h(g^{-1}\cdot m)(1)=h^\vee_M(g^{-1}\cdot m),
\end{equation}
this implies $h^\vee_M\not=0$.
\end{proof}

Let $\alpha\colon B\to M$ be a homomorphism of rational discrete left $\QG$-modules.
Then for $h\in M^\odot$ and $b\in B$ one has (cf. \eqref{eq:HomC3})
\begin{align}
\alpha^\odot(h)^\vee_B(b)&=
\big(\alpha^\odot(h)(b)\big)(1)
=h(\alpha(b))(1)\notag\\
&=h^\vee_M(\alpha(b))=\alpha^\ast(h^\vee_M)(b),\label{eq:HomC6}
\end{align}
i.e., the diagram 
\begin{equation}
\label{eq:HomC7}
\xymatrix{
M^\odot\ar[r]^{\alpha^\odot}\ar[d]_{\argu^\vee_M}& B^\odot\ar[d]^{\argu^\vee_B}\\
M^\ast\ar[r]^{\alpha^\ast}& B^\ast
}
\end{equation}
is commutative. 

%%%

\subsection{Proper signed discrete left $G$-sets}
\label{ss:propG}
Let $G$ be a t.d.l.c. group.
A signed discrete left $G$-set $\Omega$ (cf. \S\ref{ss:signedperm})
will be said to be {\it proper}, if $G$ has finitely many orbits on $\Omega$ and
$\stab_G(\omega)$ is compact for all $\omega\in\Omega$. 
In particular,
$\Q[\uOmega]$ is a finitely generated projective rational discrete left $\QG$-module (cf. \eqref{eq:sdpm4}). 
For $\omega\in\Omega$ define
$\omega^\ast\in\Q[\Omega]^\ast$ by
\begin{equation}
\label{eq:propG1}
\omega^\ast(\xi)=\begin{cases}
\hfil 1\hfil&\ \text{if $\xi=\omega$, }\\
\hfil-1\hfil&\ \text{if $\xi=\bomega$}, \qquad and\\
\hfil 0\hfil&\ \text{if $\xi\not\in\{\omega,\bomega\}$.}
\end{cases}
\end{equation}
In particular, if $g\in G$ then
\begin{equation}
\label{eq:propG2}
g\cdot \omega^\ast=(g\cdot \omega)^\ast.
\end{equation}
For a proper signed discrete left $G$-set $\Omega$ we put also
\begin{equation}
\label{eq:propG3}
\Q[\uOmega^\ast]=\spn_{\Q}\{\,\omega^\ast\mid \omega\in\Omega\}\subseteq\Q[\Omega]^\ast.
\end{equation}
In particular, as $\Omega^\ast=\{\,\omega^\ast\mid\omega\in\Omega\,\}$ is a proper signed discrete left $G$-set, $\Q[\uOmega^\ast]$ is a 
finitely generated projective discrete left $\QG$-module.
One has the following property.

\begin{prop}
\label{prop:propG1}
Let $G$ be a t.d.l.c. group, and let $\Omega$
be a proper signed discrete left $G$-set. Then
$\image(\argu^\vee_{\Q[\uOmega]})=\Q[\uOmega^\ast]$. In particular, 
$\argu^\vee_{\Q[\uOmega]}$ induces an isomorphism
$\argu_{\Omega}^\vee\colon \Q[\uOmega]^\odot\longrightarrow\Q[\uOmega^\ast]$.
\end{prop}

\begin{proof}
In suffices to prove the claim for a proper signed discrete left $G$-set $\Omega$ with 
one $G\times C_2$-orbit (cf. \eqref{eq:sdpm4}). We distinguish the two cases.

{\bf Case 1:} $G$ acts without inversion on $\Omega$. 
Let $\omega\in\Omega$ and put $\caO=\stab_G(\omega)$.
By hypothesis,
$\Omega=G\cdot \omega\sqcup G\cdot\bomega$. 
Let $\caR\subseteq G$ be a set of representatives for
$G/\caO$. For $x\in\caR$ there exists $h_x\in \Q[\uOmega]^\odot=\Hom_G(\Q[\uOmega],\caC_c(G))$ given by
$h_x(\omega)=I_{\caO x^{-1}}$. Moreover,
as ${}^\caO\caC_c(G)=\spn_{\Q}\{\,I_{\caO x^{-1}}\mid x\in\caR\,\}$ (cf. \eqref{eq:ecky3} and \eqref{eq:biB14}), one has
$\Q[\uOmega]^\odot=\spn_{\Q}\{\,h_x\mid x\in\caR\,\}$.
For $\varpi\in G\cdot\omega$ there exists a unique element
$y\in\caR$ such that $\varpi= y\cdot\omega$. Hence
\begin{equation}
\label{eq:propG4}
(h_x)^\vee_{\Q[\uOmega]}(\varpi)=
(h_x)^\vee_{\Q[\uOmega]}(y\cdot\omega)
=h_x(y\cdot \omega)(1)=I_{y\caO x^{-1}}(1)=I_{y\caO}(x)=\delta_{x,y}.
\end{equation}
Thus $(h_x)^\vee_{\Q[\uOmega]}=(x\cdot\omega)^\ast$. From this fact one concludes that
$\argu^\vee_{\Q[\Omega]}$ induces a mapping
$\argu_{\Omega}\colon \Q[\uOmega]^\odot\longrightarrow\Q[\uOmega^\ast]$,
and that $\argu_{\Omega}$ is surjective. Hence the claim follows from
Proposition~\ref{prop:mcom}.

{\bf Case 2:} $G$ acts with inversion on $\Omega$.
Let $\omega\in\Omega$ and $\sigma\in G$ be such that $\sigma\cdot\omega=\bomega$.
Put $\caU=\stab_G(\{\omega,\bomega\})$ and $\caO=\stab_G(\omega)$.
In particular, $\caO$ is normal in $\caU$ and $\sigma\in\caU$.
Let $\caS\subseteq G$ be a set of representatives for
$G/\caU$, and put $\caR=\caS\sqcup\caS\sigma$. Then $\caR$ is a set of representatives for $G/\caO$.
For $x\in\caS$ there exists $k_x\in \Q[\uOmega]^\odot=\Hom_G(\Q[\uOmega],\caC_c(G))$ given by
$k_x(\omega)=I_{\caO x^{-1}}-I_{\caO\sigma^{-1} x^{-1}}$. Moreover, as before one has
${}^\caO\caC_c(G)=\spn_{\Q}\{\,I_{\caO z^{-1}}\mid z\in\caR\,\}$, and
$\spn_{\Q}\{\,k_x\mid x\in\caS\,\}$ coincides with the eigenspace of the endomorphism
$\sigma_\circ\in\End({}^\caO\caC_c(G))$ with respect to the eigenvalue $-1$.
Thus $\Q[\uOmega]^\odot=\spn_{\Q}\{\,k_x\mid x\in\caS\,\}$.

By hypothesis, $\Omega=G\cdot\omega$.
Let $\Omega^+=\caS\cdot\omega$. Then $\Omega=\Omega^+\sqcup\Omega^-$,
where $\Omega^-=\{\,\bar{\varpi}\mid \varpi\in\Omega^+\,\}$.
Hence for $\varpi\in \Omega^+$ there exists a unique element
$y\in\caS$ such that $\varpi= y\cdot\omega$. This yields
\begin{align}
(k_x)^\vee_{\Q[\uOmega]}(\varpi)&=
(k_x)^\vee_{\Q[\uOmega]}(y\cdot\omega)
=k_x(y\cdot \omega)(1)=I_{y\caO x^{-1}}(1)-I_{y\caO\sigma^{-1} x^{-1}}(1).\notag\\
&=I_{y\caO x^{-1}}(1)=\delta_{x,y}\label{eq:propG5}
\end{align}
Here we used the fact that
$I_{y\caO\sigma^{-1} x^{-1}}(1)=I_{y\caU x^{-1}}(1)-I_{y\caO x^{-1}}(1)=0$. Hence
$(k_x)^\vee_{\Q[\uOmega]}=(x\cdot\omega)^\ast$ and the claim follows by the
same argument as in the previous case.
\end{proof}

One also concludes the following.

\begin{prop}
\label{prop:propG2}
Let $G$ be a t.d.l.c. group, let $\Xi$ and $\Omega$ be proper signed discrete
left $G$-sets, and let $\alpha\colon\Q[\uXi]\to\Q[\uOmega]$ be a proper 
homomorphism of left $\QG$-modules (cf. \eqref{eq:sset5}).
Then one has a commutative diagram
\begin{equation}
\label{eq:propG6}
\xymatrix{
\Q[\uOmega]^\odot\ar[r]^{\alpha^\odot}\ar[d]_{\argu^\vee_\Omega}&
\Q[\uXi]^\odot\ar[d]^{\argu^\vee_\Xi}\\
\Q[\uOmega^\ast]\ar[r]^{\ualpha^\ast}&\Q[\uXi^\ast],
}
\end{equation}
where the vertical maps are isomorphisms.
\end{prop}

\begin{proof}
This is an immediate consequence of Proposition~\ref{prop:propG1} and the commutativity
of the diagram \eqref{eq:HomC7}.
\end{proof}

%%%%%%%
%% 3/3/2015
%%%%%%%

\section{Discrete actions of t.d.l.c. groups on graphs}
\label{s:dag}
The notion of {\it graph} which will be used throughout the paper
coincides with the notion used by J-P.~Serre in \cite{serre:trees}, i.e.,
a {\it graph} $\Gamma=(\caV(\Gamma),\caE(\Gamma))$ will consist
of a {\it set of vertices} $\caV(\Gamma)$, a {\it set of edges} $\caE(\Gamma)$, an
{\it origin mapping} $o\colon \caV(\Gamma)\to \caE(\Gamma)$,
a {\it terminus mapping} $t\colon \caV(\Gamma)\to \caE(\Gamma)$ and an {\it edge inversion
mapping} $\bar{\phantom{x}}\colon \caE(\Gamma)\to \caE(\Gamma)$ satisfying
\begin{equation}
\label{eq:grid}
t(\beue)=o(\eue),\qquad
o(\beue)=t(\eue),\qquad
\bar{\beue}=\eue,\qquad
\beue\not=\eue
\end{equation}
for all $\eue\in \caE(\Gamma)$ (cf. \cite[\S I.2.1]{serre:trees}).
In particular, $\caE(\Gamma)$ is a {\it signed set} (cf. \S\ref{ss:ssets}).
Such a graph is said to be {\it combinatorial}, if the
map
\begin{equation}
\label{eq:graph1}
(t,o)\colon\caE(\Gamma)\longrightarrow\caV(\Gamma)\times\caV(\Gamma)
\end{equation}
is injective.
Moreover, $\Gamma$ is said to be
{\it locally finite}, if $\eust_\Gamma(v)=\{\,\eue\in\caE(\Gamma)\mid o(\eue)=v\,\}$ is finite
for every vertex $v\in\caV(\Gamma)$.

%%%%

\subsection{The exact sequence associated to a graph}
\label{ss:exgraph}
For a graph $\Gamma$ let 
\begin{align}
\boV(\Gamma)&=\Q[\caV(\Gamma)]\label{eq:defV}\\
\intertext{denote the free $\Q$-vector space over the set of vertices of $\Gamma$,
and let}
\boE(\Gamma)&=\Q[\underline{\caE(\Gamma)}]=
\Q[\caE(\Gamma)]/\langle\, \eue+\beue\mid\eue\in\caE(\Gamma)\,\rangle,\label{eq:defE}
\end{align}
(cf. \eqref{eq:sset1}), 
i.e., if $\ueue\in\boE[\Gamma]$ denotes the image of $\eue\in\caE(\Gamma)$ in $\boE(\Gamma)$,
one has $\ubeue=-\ueue$. One has a canonical $\Q$-linear mapping  
$\der\colon\boE(\Gamma)\to\boV(\Gamma)$
given by
\begin{equation}
\label{eq:canG}
\der(\ueue)=t(\eue)-o(\eue),\qquad\eue\in\caE(\Gamma),
\end{equation}
which has the following well known properties (cf. \cite[\S 2.3, Cor.~1]{serre:trees}).

\begin{fact}
\label{fact:graph}
Let $\Gamma=(\caV(\Gamma),\caE(\Gamma))$ be a graph,
and let $\der\colon\boE(\Gamma)\to\boV(\Gamma)$ be the map given by
\eqref{eq:canG}.
Then
\begin{itemize}
\item[(a)] $\kernel(\der)\simeq H_1(|\Gamma|,\Q)$,
where $|\Gamma|$ denotes the topological realization of $\Gamma$.
\item[(b)] $\coker(\der)\simeq\Q[\caV(\Gamma)/\!\!\sim]$, where $\sim$ is the connectedness relation, i.e.,
$\Gamma$ is connected, if and only if, $\coker(\der)\simeq\Q$.
\end{itemize}
In particular, $\Gamma$ is a tree if, and only if, $\kernel(\der)=0$ and
$\coker(\der)\simeq \Q$.
\end{fact} 

Let $\boL(\Gamma)=\kernel(\der)$, i.e., if $\Gamma$ is a connected graph the sequence
\begin{equation}
\label{eq:graph2}
\xymatrix{
0\ar[r]&\boL(\Gamma)\ar[r]&\boE(\Gamma)\ar[r]^{\der}&\boV(\Gamma)\ar[r]^{\eps}&\Q\ar[r]&0
}
\end{equation}
is exact.

%%%%

\subsection{Rough Cayley graphs}
\label{ss:rcgraph}
Let $G$ be a t.d.l.c. group, let $\caO$ be a compact open subgroup of $G$, and let $S\subseteq G\setminus \caO$
be a symmetric subset of $G$ intersecting $\caO$ trivially, i.e., $s\in S$ implies $s^{-1}\in S$.
The {\it rough Cayley graph} $\Gamma=\Gamma(G,S,\caO)$ associated with $(G,S,\caO)$ is the graph
given by $\caV(\Gamma)=G/\caO$ and 
\begin{equation}
\label{eq:rcg1}
\caE(\Gamma)=\{\,(g\caO,gs\caO)\mid g\in G,\ s\in S\,\}.
\end{equation}
The origin mapping is given by the projection on the first coordinate,
the terminus mapping is given by the projection on the second coordinate,
while the edge inversion mapping permutes the first and second coordinate.
The definition we have chosen here follows the approach used in \cite[\S 2]{bero:rcg}.
However, in our setup edges of a graph are directed, and we do not claim $\Gamma$ to be connected.
By construction, $G$ has a discrete left action on $\Gamma=\Gamma(G,S,\caO)$.
It is straightforward to verify that $\Gamma$ is combinatorial. One has the following properties.

\begin{prop}
\label{prop:rcg}
Let $G$ be a t.d.l.c. group, let $\caO$ be a compact open subgroup of $G$, let $S$
be a symmetric subset of $G\setminus \caO$, and let $\Gamma=\Gamma(G,S,\caO)$ be
the rough Cayley graph associated with $(G,S,\caO)$.
\begin{itemize}
\item[(a)] $G$ acts transitively on the set of vertices of $\Gamma$;
\item[(b)] if $S$ is a finite set, then $\Gamma$ is locally finite;
\item[(c)] $\Gamma$ is connected if, and only if,
$G$ is generated by $\caO$ and $S$, i.e., $G=\langle\, S,\caO\,\rangle$.
\end{itemize}
\end{prop}

\begin{proof}
(a) is obvious.

\noindent
(b) Suppose that $S$ is finite. 
By (a), it suffices to show that $\eust_\Gamma(\caO)$ is a finite set.
By definition, $\eust_\Gamma(\caO)=\{\,(\omega\caO,\omega s\caO\mid \omega\in\caO,\ s\in S\,\}$.
Since $\caO$ is compact open in $G$, any double coset $\caO s\caO$ is the union
of finitely many right cosets $g_1(s)\caO$,\ldots, $g_{\alpha(s)}(s)\caO$, i.e., 
 $t(\eust_\Gamma(\caO))=\{\,g_j(s)\caO\mid s\in S,\ 1\leq j\leq\alpha(s)\,\}$.
As $\Gamma$ is combinatorial, the restriction of $t$ to $\eust_\Gamma(\caO)$ is injective.
This yields the claim.

\noindent
(c) Suppose that $\Gamma$ is connected, and let $g\in G$. 
Then there exists a path $\eup=\eue_0\cdots\eue_n$ from $\caO$ to $g\caO$,
i.e., $o(\eue_0)=\caO$, $t(\eue_n)=g\caO$ and $o(\eue_j)=t(\eue_{j-1})$ for $1\leq j\leq n$.
Let $g_j\in G$, $s_j\in S$ such that $\eue_j=(g_j\caO,g_js_j\caO)$ for $0\leq j\leq n$, i.e.,
$g_0=\omega_0\in\caO$. By induction, one concludes that $g_j\in \langle\, S,\caO\,\rangle$
for all $1\leq j\leq n$. By construction, there exists $\omega_n\in\caO$ such that
$g_ns_n\omega_n=g$. In particular, $g\in \langle\, S,\caO\,\rangle$, and thus $G=\langle\, S,\caO\,\rangle$.

Suppose now that $G=\langle\, S,\caO\,\rangle$.
By (a), it suffices to show
that for any vertex $g\caO\in\caV(\Gamma)$, there is a path from $\caO$ to $g\caO$.
By hypothesis, there exists elements $s_1,\ldots, s_n\in S$ and $\omega_0,\ldots,\omega_n\in\caO$ such that
$g=\omega_0s_1\cdots s_n\omega_n$. Hence for $\eue_0=(\omega_0\caO,\omega_0s_1\caO)$ and
$\eue_k=(\omega_0s_1\cdots s_{k-1}\omega_{k-1}\caO,\omega_0s_1\cdots s_{k}\caO)$ for $1\leq k\leq n$
one verifies easily that
$\eup=\eue_0\cdots\eue_n$ is a path from $\caO$ to $g\caO$. Thus $\Gamma$ is connected. 
\end{proof}

%%%%%

\subsection{T.d.l.c. groups of type \FP$_1$}
\label{ss:FP1}
It is well known that a discrete group is fini\-te\-ly generated if, and only if, it is of type $\FP_1$ 
(cf. \cite[Ex.~1d of \S I.2 and Ex.~1 of \S VIII]{brown:coh}). 
For t.d.l.c. groups one has the following.

\begin{thm}
\label{thm:FP1}
Let $G$ be a t.d.l.c. group. Then the following are equivalent.
\begin{itemize}
\item[(i)] The group $G$ is compactly generated.
\item[(ii)]  The discrete $\QG$-module 
$\ker(\Q[G/\caO]\stackrel{\eps}{\rightarrow} \Q)$ is finitely generated
for any compact open subgroup $\caO$ of $G$, where $\eps$ is the augmentation map.
\item[(iii)]  There exists a compact open subgroup $\caO$ of $G$ such that  
$\ker(\Q[G/\caO]\stackrel{\eps}{\rightarrow} \Q)$ is finitely generated.
\item[(iv)] The group $G$ is of type $\FP_1$.
\end{itemize}
\end{thm}
	
\begin{proof}
The implication (ii)$\Rightarrow$(iii) is trivial, and, by Proposition~\ref{prop:FPn}, (iii) and (iv) are equivalent.
Suppose that $G$ is compactly generated, and let $\caO$ be any compact open subgroup of $G$.
Then there exists a finite symmetric subset $S\subseteq G\setminus \caO$ such that $G=\langle\, S,\caO\,\rangle$.
Let $\Gamma=\Gamma(G,S,\caO)$ be the rough Cayley graph associated with $(G,S,\caO)$.
Then, as $\Gamma$ is connected (cf. Prop.~\ref{prop:rcg}), one has a short exact sequence (cf. \eqref{eq:graph2})
\begin{equation}
\label{eq:FP1-1}
\xymatrix{
\boE(\Gamma)\ar[r]^-{\der}&\Q[G/\caO]\ar[r]^-{\eps}&\Q\ar[r]&0.}
\end{equation}
Moreover, as $\boE(\Gamma)$ is a rational discrete left $\QG$-module being generated by the finite subset
$\{\,(\underline{\caO,s\caO})\mid s\in S\,\}$, $\kernel(\eps)$ is finitely generated. This shows 
(i)$\Rightarrow$(ii).

Let $\caO$ be a compact open subgroup of $G$ such that $K=\ker(\eps_\caO\colon\Q[G/\caO]\to \Q)$ is a finitely generated discrete left $\QG$-module, i.e., there exist finitely many elements $u_1,\ldots, u_n\in K$ generating
$K$ as $\QG$-module. The set $\caB=\{\,g\caO-\caO\mid g\in G\setminus\caO\,\}$ is a generating set of $K$ as
$\Q$-vector space, i.e., any element $u_j$ is a linear combination of finitely many elements of $\caB$.
Hence there exists a finite subset $S\subset G\setminus\caO$ such that
any element $u_j$ is a linear combination of $\Sigma=\{\,s\caO-\caO\mid s\in S\,\}$. Adjoining the inverses of elements
in $S$ if necessary, we may also assume that $S$ is symmetric, i.e., $s\in S$ implies $s^{-1}\in S$.
Note that, by construction, one has 
\begin{equation}
\label{eq:FP1-2}
\textstyle{K=\sum_{1\leq j\leq n}\QG\cdot u_j =\sum_{s\in S} \QG\cdot (s\caO-\caO).}
\end{equation}
Let $\Gamma=\Gamma(G,S,\caO)$ be the rough Cayley graph associated with $(G,S,\caO)$,
and let $\der\colon\boE(\Gamma)\to\Q[G/\caO]$ denote the canonical map, where
we have identified $\boV(\Gamma)$ with $\Q[G/\caO]$.
Then $\image(\der)\subseteq\kernel(\eps_\caO)$, and $s\caO-\caO\in\image(\der)$ for all $s\in S$
(cf. \eqref{eq:canG}, \eqref{eq:rcg1}).
In particular, by \eqref{eq:FP1-2}, $\image(\der)=K$, and $\coker(\der)\simeq\Q$, i.e.,
$\Gamma$ is connected (cf. Fact~\ref{fact:graph}(b)). 
Therefore, $G=\langle\, S,\caO\,\rangle$ (cf. Prop.~\ref{prop:rcg}(c)),
and $G$ is compactly generated.
This yields the implication (iii)$\Rightarrow$(i) and completes the proof.
\end{proof}

%%%%%%

\subsection{T.d.l.c. groups acting discretely on a tree}
\label{ss:tdlctree}
Let $G$ be a group acting discretely on a graph $\Gamma$, i.e.,
all vertex stabilizers $G_v=\stab_G(v)$, $v\in\caV(\Gamma)$, and all
edge stabilizers $G_{\eue}=\stab_G(\eue)$, $\eue\in\caE(\Gamma)$ are open.
For $\eue\in\caE(\Gamma)$ the set $\{\eue,\beue\}$ will be called
a {\it geometric edge} of $\Gamma$. By $\caE^g(\Gamma)$ we denote
the set of all geometric edges of $\Gamma$, i.e.,
$\caE^g(\Gamma)=\caE(\Gamma)/C_2$, where $C_2$ is the cyclic
group of order $2$ where the involution is acting by edge inversion.
Let $G_{\{\eue\}}$ denote the stabilizer of the geometric edge $\{\eue,\beue\}$, i.e.,
one has a {\it sign character}
\begin{equation}
\label{eq:signgeo}
\sgn_\eue\colon G_{\{\eue\}}\longrightarrow\{\pm1\}.
\end{equation}
By $\Q_\eue$ we denote the $1$-dimensional rational discrete left $\Q[G_{\{\eue\}}]$-module
with action given by \eqref{eq:signgeo}, i.e.,
for $g\in G_{\{\eue\}}$ and $q\in\Q_\eue$ one has $g\cdot q=\sgn_{\eue}(g)\cdot q$.

For a tree $\caT$ one has $\boL(\caT)=0$. Hence
the exact sequence \eqref{eq:graph2} specializes to a short exact sequence.
This fact implies the following version
of a result of I.~Chiswell (cf. \cite[p.~179]{brown:coh}, \cite{ian:tree}).

\begin{prop}
\label{prop:treeact}
Let $G$ be a t.d.l.c. group acting discretely on a tree $\caT$.
Let $\caR_{\caV}\subseteq\caV(\caT)$ be a set of representatives
of the $G$-orbits on $\caV(\caT)$, and let 
$\caR_{\caE}\subseteq\caE(\caT)$ be a set of representatives
of the $G\times C_2$-orbits on $\caE(\caT)$.
\begin{itemize}
\item[(a)] For $M\in\ob(\QGdis)$ one has a long exact sequence
\begin{equation}
\label{eq:chis1}
\xymatrix@C=.45cm{
\ldots\ar[r]&\dH^k(G,M)\ar[r]&\prod_{v\in\caR_{\caV}} \dH^k(G_v,M)\ar[r]&
\prod_{\eue\in\caR_{\caE}}\dH^k(G_{\{\eue\}},M\otimes\Q_\eue)\ar[r]&\ldots
}
\end{equation}
\item[(b)] Suppose that $G_v$ is compact for all $v\in\caV(\caT)$. Then
$\ccd_\Q(G)\leq 1$.
In particular, \eqref{eq:chis1} specializes to an exact sequence
\begin{equation}
\label{eq:chis2}
\xymatrix{
0\ar[r]&\dH^0(G,M)\ar[r]&
\prod_{v\in\caR_{\caV}} \dH^0(G_v,M)\ar[d]\\
0&\dH^1(G,M)\ar[l]&\prod_{\eue\in\caR_{\caE}}\dH^0(G_{\{\eue\}}, M\otimes\Q_{\eue})\ar[l] 
}
\end{equation} 
\end{itemize}
\end{prop}

\begin{proof}
(a) By hypothesis, one has isomorphisms
\begin{equation}
\label{eq:graph3}
\begin{aligned}
\boV(\caT)&\textstyle{\simeq \coprod_{v\in\caR_{\caV}} \idn_{G_v}^G(\Q)},\\
\boE(\caT)&\textstyle{\simeq \coprod_{\eue\in\caR_{\caE}} \idn_{G_{\{\eue\}}}^G(\Q_{\eue})},
\end{aligned}
\end{equation}
and a short exact sequence 
\begin{equation}
\label{eq:graph4}
\xymatrix{
0\ar[r]&\boE(\caT)\ar[r]^-{\der}&\boV(\caT)\ar[r]^-{\eps}&\Q\ar[r]&0.
}
\end{equation}
of rational discrete left $\QG$-modules.
Hence \eqref{eq:chis1} is a direct consequence
of the
long exact sequence associated to $\dExt^\bullet_G(\argu,M)$ and \eqref{eq:graph4},
the Eckmann-Shapiro lemma (cf. \S\ref{ss:ecksha}),
and the isomorphisms
$\dExt_{G_{\{\eue\}}}^k(\Q_\eue,M)\simeq\dH^k(G_{\{\eue\}},M\otimes\Q_{\eue})$ which are natural in $M$.

\noindent
(b) By hypothesis, $G_v$ is compact open, and $G_\eue$ is open for all $v\in\caV(\caT)$
and $\eue\in\caE(\caT)$. As $G_\eue\subseteq G_{t(\eue)}$, this implies that $G_\eue$ and $G_{\{\eue\}}$ are also
compact and open. Hence \eqref{eq:graph4} is a projective resolution of $\Q$ in $\QGdis$, and thus
$\ccd_{\Q}(G)\leq 1$. The exact sequence \eqref{eq:chis2} follows from \eqref{eq:chis1} and the fact
that in this case one has
$\ccd_{\Q}(G_v)=\ccd_{\Q}(G_{\{\eue\}})=0$ (cf. Prop.~\ref{prop:cd}(a)). 
\end{proof}

\begin{rem}
\label{rem:stall}
If $G$ is a compactly generated t.d.l.c. group, then $\dH^1(G,\biB(G))$ must have dimension $0$, $1$ or $\infty$.
It is shown in \cite{ila:stall} that in the latter case $G$ is either isomorphic to a free product with amalgamation
$H\coprod_{\caU}\caO$, where $H$ is an open subgroup of $G$, $\caO$ is a compact open subgroup of $G$ and
$|\caO:\caO\cap H|\not=1$,
or $G$ is isomorphic to an HNN-extension $\HNN(H,\varphi)$ for some open subgroup $H$, $\caO$ and $\caU$
are compact open subgroups of $H$ with at least one of them different from $H$
and $\varphi\colon\caO\to\caU$ is an isomorphism, i.e.,
Stallings' decomposition theorem holds for compactly generated t.d.l.c. groups.
As a consequence a compactly generated t.d.l.c. group $G$ satisfying $\dim(\dH^1(G,\biB(G)))\geq 2$ must have
a non-trivial discrete action on some tree $\caT$.
\end{rem}

%%%%%

\subsection{Fundamental groups of graphs of profinite groups}
\label{ss:fund}
Let $\Lambda$ be a connected graph. 
A {\it graph of profinite groups} $(\caA,\Lambda)$ based on the graph $\Lambda$
consists of
\begin{itemize}
\item[(i)] a profinite group $\caA_v$ for every vertex $v\in\caV(\Lambda)$;
\item[(ii)] a profinite group $\caA_{\eue}$ for every edge $\eue\in\caE(\Lambda)$
satisfying $\caA_{\eue}=\caA_{\beue}$;
\item[(iii)] an open embedding $\iota_{\eue}\colon\caA_{\eue}\rightarrow\caA_{t(\eue)}$ 
for every edge $\eue\in\caE(\Lambda)$.
\end{itemize}
For a graph of groups $(\caA,\Lambda)$ one defines the group
$F(\caA,\Lambda)$ (cf. \cite[\S I.5.1]{serre:trees})
as the group generated by $\caA_v$, $v\in\caV(\Lambda)$, 
and $\eue\in\caE(\Lambda)$ subject to the relations
\begin{equation}
\label{eq:relF}
\eue^{-1}=\beue\ \ \text{and}\ \ \eue\,\iota_\eue(a)\eue^{-1}=\iota_{\beue}(a)
\ \ \text{for all $\eue\in\caE(\Lambda)$, $a\in\caA_\eue$.}
\end{equation}
Let $\Xi\subseteq\Lambda$ be a maximal subtree of $\Lambda$.
The {\it fundamental group} $\pi_1(\caA,\Lambda,\Xi)$ of the graph of groups $(\caA,\Lambda)$ 
with respect to $\Xi$
is  given by
\begin{equation}
\label{eq:fund}
\Pi=\pi_1(\caA,\Lambda,\Xi)=F(\caA,\Lambda)/\langle\, \eue\mid\eue\in\caE(\Xi)\,\rangle
\end{equation}
(cf. \cite[\S I.5.1]{serre:trees}).
Choosing an orientation $\caE^+\subset\caE(\Lambda)$ of $\Lambda$ 
one may construct a tree $\caT=\caT(\caA,\Lambda,\Xi,\caE^+)$
with a canonical left $\Pi$-action (cf. \cite[\S I.5.1, Thm.~12]{serre:trees}).
The fundamental group of a graph of profinite groups carries naturally the structure of 
t.d.l.c. group. Indeed, the set of all open subgroups of vertex stabilizers are a 
neighborhood basis of the identity element of the topology.
Thus from the Main Theorem of Bass-Serre theory one concludes the following.

\begin{prop}
\label{prop:fgrgr}
Let $(\caA,\Lambda)$ be a graph of profinite groups, let $\Xi\subseteq\Lambda$ be a maximal subtree
of $\Lambda$, and let $\caE^+\subset\caE(\Lambda)$ be an orientation of $\Lambda$.
Then $\Pi=\pi_1(\caA,\Lambda,\Xi)$ carries naturally the structure of a t.d.l.c. group,
and its action on $\caT=\caT(\caA,\Lambda,\Xi,\caE^+)$ is discrete and proper,
i.e., vertex stabilizers and edge stabilizers are compact and open.
In particular, $\ccd_{\Q}(\Pi)\leq 1$.
\end{prop}

\begin{proof}
This follows from the Main Theorem of Bass-Serre theory and Proposition~\ref{prop:treeact}(b).
\end{proof}

\begin{rem}
\label{rem:fgrgr} 
The Main Theorem of Bass-Serre theory would yield more information
than stated in Proposition~\ref{prop:fgrgr}. Indeed, a discrete and proper action
of a t.d.l.c. group $G$ on a tree $\caT$ can be used to define a graph of profinite groups
$(\caA,\Lambda)$ such that $G\simeq\pi_1(\caA,\Lambda,\Xi)$.
\end{rem}

%%%%

\subsection{Nested t.d.l.c. groups}
\label{ss:nn}
A t.d.l.c. group $G$ is said to be {\it nested}, if there exists
a countable ascending sequence of compact open subgroups $(\caO_k)_{k\geq 0}$ of $G$
such that 
\begin{equation}
\label{eq:nest0}
G=\textstyle{\bigcup_{k\geq 0} \caO_k.}
\end{equation}
Such a group can be represented as the fundamental group
of the graph of profinite groups 
\begin{equation}
\label{eq:nest1}
(\caA,\Lambda)\colon\qquad\xymatrix@M=0pt@R=3pt{
\caO_1&\caO_2&\caO_3&\caO_4&\caO_5\\
\bullet\ar@{-}[r]_{\caO_1}&\bullet\ar@{-}[r]_{\caO_2}&\bullet\ar@{-}[r]_{\caO_3}&
\bullet\ar@{-}[r]_{\caO_4}&\bullet\ar@{-}[r]_{\caO_5}&\ldots
}
\end{equation}
based on the straight infinite half line $\Lambda$, where all injections are the canonical ones, i.e., one has
\begin{equation}
\label{eq:nest2}
G\simeq\pi_1(\caA,\Lambda,\Lambda).
\end{equation}
Thus from Proposition~\ref{prop:fgrgr} one concludes the following.

\begin{prop}
\label{prop:nested}
Let $G$ be a nested t.d.l.c. group. Then $\ccd_\Q(G)\leq 1$.
\end{prop}

\begin{rem}
\label{rem:Qp}
Let $p$ be a prime number. Then 
$\Q_p=\bigcup_{k\geq 0} p^{-k}\Z_p$,
the additive group of the {\it $p$-adic numbers} 
is a nested t.d.l.c. group, but
$\Q_p$ is not compact. Hence $\ccd_{\Q}(\Q_p)=1$.
\end{rem}

%%%%%%%%%%%%%%%%%%%%%%

\subsection{Generalized presentations of a t.d.l.c. groups}
\label{ss:genrep}
Let $G$ be a t.d.l.c. group. 
A {\it generalized presentation} of $G$ is a graph of profinite groups $(\caA,\Lambda)$ together with a continuous open surjective
homomorphism
\begin{equation}
\label{eq:pres}
\phi\colon\pi_1(\caA,\Lambda,\Xi)\longrightarrow G,
\end{equation}
such that $\phi\vert_{\caA_v}$ is injective for all $v\in\caV(\Lambda)$.

\begin{prop} 
\label{prop:pres}
Let $G$ be a t.d.l.c. group.
\begin{itemize}
\item[(a)] $G$ has a generalized presentation
$((\caA,\Lambda_0),\phi)$, where $\Lambda_0$ is a graph with a single vertex.
\item[(b)] For any generalized presentation $((\caA,\Lambda),\phi)$, $K=\kernel(\phi)$
is a discrete free group contained in the {\it quasi-center}
\begin{equation}
\label{eq:qz}
\QZ(\Pi)=\{\,g\in\Pi\mid C_\Pi(g)\ \text{is open in $\Pi$}\,\}.
\end{equation}
of $\Pi=\pi_1(\caA,\Lambda,\Xi)$. In particular, 
\begin{equation}
\label{eq:relmod}
R(\phi)=K/[K,K]\otimes_\Z\Q
\end{equation}
is a rational discrete left $\QG$-module.
\item[(c)] If $((\caA_i,\Lambda_i),\phi_i)$, $i=1,2$, are two presentations of $G$, then
$R(\phi_1)$ and $R(\phi_2)$ are stably equivalent, i.e., there exist projective rational discrete
$\QG$-modules $P_1$ and $P_2$ such that $R(\phi_1)\oplus P_1\simeq R(\phi_2)\oplus P_2$.
\end{itemize}
\end{prop}

\begin{proof} (a) By van Dantzig's theorem there exists a
compact open subgroup $\caO$ of $G$.
Choose a subset $S$ of $G$ such that $G=<\caO,S>$. 
Without loss of generality we may assume that $S=S^{-1}$, and choose a subset $S^+\subseteq S$ such that
$S=S^+\cup(S^+)^{-1}$.
Let $\Lambda_0$ be the graph with a single vertex $v$ with a loop attached for any $s\in S^+$, i.e., $\caV(\Lambda_0)=\{v\}$
and $\caE(\Lambda_0)=\{\,\eue_s,\beue_s\mid s\in S^+\,\}$.
Let $(\caA,\Lambda_0)$ be the graph of profinite groups based on $\Lambda_0$ defined as follows:
\begin{itemize}
\item[-] $\caA_v=\caO$;
\item[-] $\caA_{\eue_s}=\caA_{\beue_s}=\caO\cap s^{-1}\caO s$ for all $s\in S^+$;
\item[-] $\alpha_{\eue_s}\colon\caA_{\eue_s}\longrightarrow s^{-1}\caO s\stackrel{i_s}{\longrightarrow }\caA_v$ for all $s\in S^+$, where $i_s$ is left conjugation by $s$;
\item[-] $\alpha_{\beue_s}\colon\caA_{\eue_s}\to\caA_v$ is the canonical inclusion for all $s\in S^+$.
\end{itemize}
Clearly, $(\caA,\Lambda_0)$ is a graph of profinite groups, and for the canonical map
\begin{equation}
\label{eq:pres1}
\phi_0\colon\caE(\Lambda_0)\cup \caO\rightarrow G,\qquad \phi(\eue_s)=\phi(\beue_s)^{-1}=s,\ \phi_0(w)=w,\ w\in\caO,
\end{equation}
the relations \eqref{eq:relF} are satisfied. The edge set of the maximal subtree $\Xi$ is empty.
Hence $\caF(\caA,\Lambda_0)\to\pi_1(\caA,\Lambda_0,\Xi)$ is an isomorphism. This show that
there exists a unique group homomorphism $\phi\colon \pi_1(\caA,\Lambda_0,\Xi)\to G$
such that $\phi\vert_{\caE(\Lambda_0)\cup \caO}=\phi_0$.
By construction, $\phi$ is surjective and $\phi\vert_{\caA_v}$ is injective. Hence $((\caA,\Lambda_0),\phi)$
is a generalized presentation.

\noindent
(b) Let $\caT=\caT(\caA,\Lambda,\Xi,\caE^+)$. Then, by construction, any non-trivial element of $K$ acts without inversion of edges and without fixed points on $\caT$. Thus, by Stallings' theorem (cf. \cite[\S I.3.3, Thm.~4]{serre:trees}), $K$ is a free discrete group.
As $K$ is normal in $\Pi$ and discrete, $K$ must be contained in $\QZ(\Pi)$.
This implies that $R(\phi)$ is a rational discrete left $\QG$-module.

\noindent
(c) Let $\Pi_i=\pi_1(\caA_i,\Lambda_i,\Xi_i)$, and let $\caT_i=\caT(\caA_i,\Lambda_i,\Xi_i,\caE^+_i)$ be trees associated to
$(\caA_i,\Lambda_i)$. In particular, $\Pi_i$ is acting without inversion on $\caT_i$.
Put $K_i=\kernel(\phi_i)$. Then for the quotient graphs $\Gamma_i=\caT_i/\!\!/K_i$,
\begin{equation}
\label{eq:pres2}
\xymatrix{
\boE(\Gamma_i)\ar[r]^-{\der_i}&\boV(\Gamma_i)\ar[r]^-{\eps_i}&\Q}
\end{equation}
are partial projective resolutions of $\Q$ in $\QGdis$. Hence, by
Schanuel's lemma (cf. \cite[\S VIII.4, Lemma~4.2]{brown:coh}),
$\kernel(\der_1)$ and $\kernel(\der_2)$ are stably equivalent.
By Fact~\ref{fact:graph}(a), one has
$\kernel(\der_1)\simeq R(\phi_1)$ and $\kernel(\der_2)\simeq R(\phi_2)$. 
This yields the claim.
\end{proof}

\begin{rem}
\label{rem:relmod}
Let $((\caA,\Lambda),\phi)$ be a presentation of the t.d.l.c. group $G$.
One should think of $R(\phi)$ as the {\it relation module} of the 
presentation $((\caA,\Lambda),\phi)$. By Proposition~\ref{prop:pres},
its stable isomorphism type is an invariant of the t.d.l.c. group $G$.
\end{rem}

\subsection{Compactly presented t.d.l.c. groups}
\label{ss:comppres}
A t.d.l.c. group $G$ is said to be {\it compactly presented},
if there exists a presentation $((\caA,\Lambda),\phi)$, such that
\begin{itemize}
\item[(i)] $\Lambda$ is a finite connected graph, and
\item[(ii)] $K=\kernel(\phi)$ is finitely generated as a normal subgroup of $\Pi=\pi_1(\caA,\Lambda,\Xi)$.
\end{itemize}
Condition (i) implies that such a group $G$ must be compactly generated.
Hence $G$ is of type $\FP_1$ (cf. Thm.~\ref{thm:FP1}). From condition (ii) follows that $R(\phi)$
is a finitely generated rational discrete $\QG$-module. Thus, as $R(\phi)\simeq\boL(\Gamma)$, where 
$\caT=\caT(\caA,\Lambda,\Xi,\caE^+)$ and $\Gamma=\caT/\!\!/K$ (cf. \eqref{eq:graph2}),
$G$ must be of type $\FP_2$ (cf. Prop.~\ref{prop:FPn}).

It has been an open problem for a long time whether discrete groups of type $\FP_2$ are indeed
finitely presented. Finally, it has been answered negatively by M.~Betsvina and N.~Brady in \cite{bebr:fin}.
In our context the following question arises.

\begin{ques}
\label{ques:pres}
Does there exist a non-discrete t.d.l.c. group $G$ which is of type $\FP_2$,
but which is not compactly presented?
\end{ques}

%%%%%%%%%
%%% 4/3/2015
%%%%%%%%%

\section{$G$-spaces for topological groups of type \OS}
\label{s:classspace}
Throughout this section we assume that $G$ is a topological group of type $\OS$, and that
$\euF\subseteq\Osgrp(G)$ is a non-empty subset of open subgroups of $G$ satisfying
\begin{itemize}
\item[($F_1$)] for $A\in\euF$ and $g\in G$ one has ${}^gA=gAg^{-1}\in\euF$;
\item[($F_2$)] for $A,B\in\euF$ one has $A\cap B\in\euF$;
\end{itemize}
In some cases it will be helpful to assume that the class $\euF$ satisfies additionally the following 
condition.
\begin{itemize}
\item[($F_3$)] for $A,B\in\Osgrp(G)$, $B\subseteq A$, $|A:B|<\infty$ and $B\in\euF$ one has $A\in\euF$.
\end{itemize} 

%%%%

\subsection{$\euF$-discrete $G$-spaces}
\label{ss:FdisG}
Let $G$ be a topological group of type \OS.
We fix a non-empty set of open subgroups $\euF$
satisfying $(F_1)$ and $(F_2)$.
A non-trivial topological space $X$ together with a continuous
left $G$-action 
$\argu\cdot\argu\colon G\times X\to X$
will be called a left {\it $G$-space}.
Such a space will be said to be {\it $\euF$-discrete}, if
$\stab_G(x)\in\euF$ for all $x\in X$.
A continuous map $f\colon X\to Y$ of left $G$-spaces which commutes with the $G$-action
is called a {\it mapping of $G$-spaces}.
The non-empty $G$-space $X$ together with an increasing filtration $(X^n)_{n\geq 0}$ of closed subspaces
$X^n\subseteq X$ is called a {\it $(G,\euF)$-CW-complex}, if
\begin{itemize}
\item[($C_1$)] $X=\bigcup_{n\geq 0} X^n$;
\item[($C_2$)] $X^0$ is an $\euF$-discrete subspace of $X$;
\item[($C_3$)] for $n\geq 1$ there is an $\euF$-discrete $G$-space $\Lambda_n$, $G$-maps
$f\colon S^{n-1}\times \Lambda_n\to X^{n-1}$ and $\widehat{f}\colon B^n\times\Lambda_n\to X^n$
such that
\begin{equation}
\label{eq:Gsp1}
\xymatrix{
S^{n-1}\times\Lambda_n\ar[d]\ar[r]^-{f}&X^{n-1}\ar[d]\\
B^n\times\Lambda_n\ar[r]^-{\widehat{f}}&X^n}
\end{equation}
is a pushout diagram,
where $S^{n-1}$ denotes the unit sphere and $B^n$ the unit ball in euclidean $n$-space
(with trivial $G$-action);
\item[($C_4$)] a subspace $Y\subset X$ is closed if, and only if, $Y\cap X^n$ is closed for all
$n\geq 0$.
\end{itemize}
The {\it dimension} of $X$ is defined by
\begin{equation}
\label{eq:dimCW}
\dim(X)=\min( \{\,k\in\N_0\mid X^k=X\,\}\cup\{\infty\}).
\end{equation}
Moreover, $X$ is said to be of {\it type $\tF_n$}, $n\in\N_0\cup\{\infty\}$,
if $G$ has finitely many orbits on $\Lambda_k$ for $0\leq k\leq n$.
%%%%%

\subsection{The $\euF$-universal $G$-CW-complex}
\label{ss:Funi}
Let $G$ be a topological group of type \OS, and let
$\euF\subseteq\Osgrp(G)$ be a non-empty set of open subgroups
satisfying $(F_1)$ and $(F_2)$.
A left $G$-CW-complex $X$ said to be
{\it $\euF$-universal} or an {\it $\uE_\euF(G)$-space}, if
\begin{itemize}
\item[($U_1$)] $X$ is contractable;
\item[($U_2$)] for all $x\in X$ one has $\stab_G(x)\in\euF$;
\item[($U_3$)] for all $A\in\euF$ the set of $A$-fixed points $X^A$ is non-empty and contractable.
\end{itemize} 
It is well known that an $\euF$-universal
$G$-CW-complex $\uE_\euF(G)$ exists and is unique up to $G$-homotopy.
The $G$-CW-complex $\uE_\euF(G)$ is universal
with respect to all $G$-CW-complexes with point stabilizers in $\euF$, i.e.,
if $X$ is a $G$-CW-complex with $\stab_G(x)\in\euF$ for all $x\in X$, then
there exists a $G$-map $X\to\uE_\euF(G)$ which is unique up to $G$-homotopy
(cf. \cite[Remark~2.5]{misva:proper}, \cite[Chap.~1, Ex.~6.18.7]{tdieck:trans}).
This space can be explicitly constructed as the union of the $n$-fold joins 
of the discrete $G$-space $\Omega_{\euF}=\bigsqcup_{A\in\euF} G/A$
(cf. \cite[\S2]{misva:proper}, \cite[Chap.~1, \S 6]{tdieck:trans}).
This notion leads to an obvious notion of dimension, i.e.,
\begin{equation}
\label{eq:tdim}
\dim_\euF(G)=\min\{\,\dim(X)\mid \text{$X$ an $\uE_\euF(G)$-space}\,\}\in\N_0\cup\{\infty\}
\end{equation}
may be considered as the {\it topological $\euF$-dimension of $G$}.
Although this might be the most natural dimension to be considered,
we prefered to introduce and study a combinatorial version of this
dimension, the {\it simplicial $\euF$-dimension} of $G$.
This choice is motivated by the fact, that for the examples
we are interested in the $\uE_\euF(G)$-spaces arise naturally as the
topological realization of certain simplicial complexes.

%%%%

\subsection{Simplicial $G$-complexes}
\label{ss:simpGcom}
A simplicial complex (cf. \S\ref{ss:simpcom}) $\Sigma$ with a left $G$-action will be called
a {\it simplicial $G$-complex}. The simplicial $G$-complex $\Sigma$
is said to be {\it of type \tF$_n$}, 
$n\in\N_0$, if $G$ has finitely many orbits on $\Sigma_q$ for all
$q\in\{0,\ldots,n\,\}$,
and {\it of type $\tF_\infty$}, if it is of type \tF$_n$ for all $n\in\N_0$.
For a simplicial $G$-complex $\Sigma$
and a $q$-simplex $\omega=\{x_0,\ldots,x_q\}\in\Sigma_q$ we denote by
\begin{equation}
\label{eq:Fdis}
\begin{aligned}
\stab_G(\omega)&=\textstyle{\bigcap_{0\leq j\leq q}\stab_G(x_j),}\\
\stab_G\{\omega\}&=\{\,g\in G\mid g\cdot\omega=g\,\},
\end{aligned}
\end{equation}
the {\it pointwise and setwise stabilizer} of $\omega$, respectively.
In particular, one has $\stab_G(\omega)\subseteq\stab_G\{\omega\}$, and
$|\stab_G\{\omega\}:\stab_G(\omega)|<\infty$.
The simplicial $G$-complex $\Sigma$ will be said to be {\it $G$-tame}\footnote{In case
that $\Sigma=\Sigma(\Gamma)$ for a graph $\Gamma$ (cf. Ex.~\ref{ex:graph}), then $\Sigma$ is a $G$-tame 
simplicial $G$-complex if, and only if, $G$ acts without inversion of edges on the graph $\Gamma$.}
if $\stab_G(\omega)=\stab_G\{\omega\}$ for all $\omega\in\Sigma$.
For the convenience of the reader some basic properties of simplicial complexes
are recollected in \S\ref{s:sc}.
%%%%

\subsection{$\euF$-discrete simplical $G$-complexes}
\label{ss:Fdissimp}
Let $G$ be a topological group of type \OS, and let
$\euF\subseteq\Osgrp(G)$ be a non-empty set of open subgroups
satisfying $(F_1)$, $(F_2)$ and $(F_3)$.
A simplicial $G$-complex $\Sigma$ will be said to be {\it $\euF$-discrete},
if $\stab_G(x)\in\euF$ for all $x\in\Sigma_0$.
Thus, by $(F_2)$, for an $\euF$-discrete simplicial $G$-complex
one has $\stab_G(\omega)\in\euF$ and, by $(F_3)$, $\stab_G\{\omega\}\in\euF$
for all $\omega\in\Sigma$.
This property has the following consequence. 

\begin{fact}
\label{fact:FdsGc}
Let $G$ be topological group of type \OS, and let $\euF\subseteq\Osgrp(G)$
be a non-empty set of open subgroups of $G$ satisfying $(F_1)$, $(F_2)$ and $(F_3)$.
Then for every 
$\euF$-discrete simplicial $G$-complex $\Sigma$, its topological realization $|\Sigma|$
(cf. \S\ref{ss:geore}) is an $\euF$-discrete $G$-CW-complex.
\end{fact}

\begin{proof}
The property of being a $G$-CW-complex is straightforward.
For any point $z\in |\Sigma|$, there exists a unique $q$-simplex
$\omega\in\Sigma$ such that $z$ is an interior point of $|\Sigma(\omega)|$ 
(cf. Fact~\ref{fact:simpcom0}). Hence
$\stab_G(\omega)\subseteq\stab_G(z)\subseteq\stab_G\{\omega\}$, and thus
$\stab_G(z)\in\euF$ by $(F_3)$.
\end{proof}

%%%%%

\subsection{Contractable $\euF$-discrete simplical $G$-complexes}
\label{ss:Fdis}
The $\euF$-discrete simplical $G$-complex $\Sigma$ will be said
to be {\it contractable}, if $|\Sigma|$ is contractable.
But note that we do not claim that $|\Sigma|$ is $G$-homotopic to a point
(cf. \cite[Thm.~2.2]{misva:proper}).

\begin{example}
\label{ex:FGcom}
Let $G$ be a topological group of type \OS,
and let $\euF\subseteq\Osgrp(G)$ be a non-empty set of open subgroups
satisfying $(F_1)$, $(F_2)$ and $(F_3)$. Let $A\in\euF$, and put
$\Omega(A)=\Sigma[G/A]$ (cf. Ex.~\ref{ex:simpcom}).
Then $\Omega(A)$ is an $\euF$-discrete simplical $G$-complex.
By Fact~\ref{fact:simpcom1}, $|\Omega(A)|$ is contractable, i.e.,
$\Omega(A)$ is a contractable $\euF$-discrete simplicial $G$-complex.
\end{example}

We define the {\it simplicial $\euF$-dimension of $G$} by
\begin{equation}
\label{eq:sdim}
\begin{aligned}
\sdim_\euF(G)&=\min(\{\,n\in\N_0\mid \text{$\exists$ 
a contractable $\euF$-discrete}\\
&\qquad \text{simplicial $G$-complex $\Sigma$: $\dim(\Sigma)=n$}\,\}\cup\{\infty\})
\end{aligned}
\end{equation}
(cf. \eqref{eq:simpcom4}). For $n\in\N_0\cup\{\infty\}$ 
we call $G$ to be {\it of type $\sF_n(\euF)$}, if there exists
a contractable $\euF$-discrete simplicial $G$-complex $\Sigma$ which is of type $\tF_n$ (cf. \S\ref{ss:simpGcom}).
Moreover, $G$ will be said to be {\it of type $\sF(\euF)$},
if there exists a contractable $\euF$-discrete simplicial $G$-complex $\Sigma$ of type $\tF_\infty$
satisfying $\dim(\Sigma)<\infty$.

%%%%%

\subsection{The Bruhat-Tits property}
\label{ss:BT}
The $\euF$-discrete simplicial $G$-complex $\Sigma$ is said to have the 
{\it Bruhat-Tits property}, if for all $A\in\euF$ one has $|\Sigma|^A\not=\emptyset$.
For every point $z\in|\Sigma|^A$ there exists a unique $q$-simplex $\omega(z)\in\Sigma$ such that
$z$ is contained in the interior of $|\Sigma(\omega(z))|$  (cf. Fact~\ref{fact:simpcom0}). Hence
$A\subseteq\stab_G(z)\subseteq\stab_G\{\omega(z)\}$.

%%%%%%%%%%

\subsection{Classes of open subgroups for t.d.l.c. groups}
\label{ss:Ftdlc}
Let $G$ be a t.d.l.c. group. Obviously, the class of all open subgroups
$\euO=\Osgrp(G)$ and the class $\euC=\CO(G)$ of all compact open subgroups
satisfy the conditions $(F_1)$, $(F_2)$ and $(F_3)$.
The same is true for 
\begin{equation}
\label{eq:Fnest}
\euvN=\{\,\caO\in\Osgrp(G)\mid \exists\, \caU\subseteq\caO,\,\text{$\caU$ open and nested,}\, |\caO:\caU|<\infty\,\},
\end{equation}
the class of all {\it virtually nested open subgroups}  of $G$ (cf. \S\ref{ss:nn}).

\begin{prop}
\label{prop:nestedclass}
Let $G$ be a t.d.l.c. group. Then $\euvN$ satisfies the conditions $(F_1)$, $(F_2)$ and $(F_3)$.
\end{prop}

\begin{proof}
Obviously, $\euvN$ satisfies $(F_1)$ and $(F_3)$.
Let $\caO\in \euvN$, and let $\caU\subseteq\caO$ be an open, nested subgroup of finite index, i.e.,
$\caU=\bigcup_{k\geq 0} \caU_k$, $\caU_k\subseteq\caU_{k+1}$,
$\caU_k\in\CO(G)$. Since $|\caO:\caU|<\infty$, we may also assume that $\caU$ is normal in $\caO$.
Let $\caV\in\euvN$. Then $\caV\cap\caU$ is a nested open subgroup of $G$. Moreover,
as $\caV\cap\caU=\kernel(\caV\to\caO/\caU)$, one has that $|\caV:\caV\cap\caU|<\infty$.
Hence $\euvN$ satisfies $(F_2)$ as well.
\end{proof}

\begin{rem}
\label{rem:nest} The proof of Proposition~\ref{prop:nested} show also that
$\euN$ - the class of all nested open subgroups of a t.d.l.c. group $G$ -
satisfies the conditions $(F_1)$ and $(F_2)$.
\end{rem}

\begin{rem}
\label{rem:thomp}
It is a remarkable fact, that although the Higman-Thompson groups $F_{n,r}$, $T_{n,r}$, $G_{n,r}$
are not of finite cohomological dimension, they still satisfy the
$\FP_\infty$-condition (cf. \cite{brown:fini}). The Neretin group $N_{n,r}$ of spheromorphism 
of a regular rooted tree $\caT_{n,r}$ can be seen as a t.d.l.c. analogue
of the Higman-Thompson groups (cf. Remark~\ref{rem:nere}). In particular, this group has a factorization
\begin{equation}
\label{eq:factN}
N_{n,r}=\caO\cdot\caF_{n,r},\qquad \caO\cap\caF_{n,r}=\{1\},
\end{equation}
where $\caF_{n,r}$ is a discrete subgroup isomorphic to the Richard Thompson group $F_{n,r}$
and $\caO$ is a nested open subgroup of $N_{n,r}$. Hence the following questions
arise.

\begin{ques}
\label{ques:ner}
Is $N_{n,r}$ of type $\sF(\euN)_\infty$ or even of type $\sF(\euC)_\infty$?
\end{ques}

\begin{ques}
\label{ques:ner2}
Does there exist an  $\uE_{\euN}(N_{n,r})$-space which is of type $F_\infty$?
\end{ques}
\end{rem}

The finiteness conditions 
we discussed in subsection \S\ref{ss:Fdis} for  $\euC$-discrete simplicial $G$-complexes
for a t.d.l.c. group $G$ have the following consequence for
its cohomological finiteness conditions.

\begin{prop}
\label{prop:ftopcd}
Let $G$ be a t.d.l.c. group. Then one has the following.
\begin{itemize}
\item[(a)] $\sdim_{\euC}(G)\geq \ccd_{\Q}(G)$.
\item[(b)] Let $n\in\N_0\cup\{\infty\}$.
If $G$ is of type $\sF_n(\euC)$, then $G$ is of type $\FP_n$.
\item[(c)] If $G$ is of type $\sF(\euC)$, then $G$ is of type $\FP$.
\end{itemize}
\end{prop}

\begin{proof}
Let $\Sigma$ be a $\euC$-discrete simplicial $G$-complex, let
$\omega=\{x_0,\ldots,x_q\}\in\Sigma_q$, and let 
$\uomega=x_0\wedge\cdots\wedge x_q\in\tSigma_q$ (cf. \eqref{eq:simpcom8}).
Then $\stab_G\{\omega\}$
is acting on $\{\pm\uomega\}$, i.e., one has a continuous group homomorphism
\begin{equation}
\label{eq:signch}
\sgn_\omega\colon\stab_G\{\omega\}\longrightarrow\{\pm1\}.
\end{equation}
Let $\Q_\omega$ denote the $1$-dimensional rational discrete left $\stab_G\{\omega\}$-module associated
to $\sgn_\omega$, i.e., for $g\in\stab_G\{\omega\}$ and $z\in\Q_\omega$ one has
$g\cdot z=\sgn_\omega(g)\cdot z$.
Thus if $\caR_q\subseteq\Sigma_q$ is a set of representatives for the $G$-orbits on $\Sigma_q$, one has
an isomorphism
\begin{equation}
\label{eq:simpiso1}
\textstyle{C_q(\Sigma)\simeq \coprod_{\omega\in\caR_q} \idn_{\stab_G\{\omega\}}^G(\Q_\omega)}
\end{equation}
of rational discrete left $\QG$-modules (cf. \S\ref{ss:osc}).
In particular, $C_q(\Sigma)$ is a projective rational discrete left $\QG$-module.

If $\Sigma$ is contractable, one has
$H_0(|\Sigma|,\Q)\simeq\Q$ and $H_k(|\Sigma|,\Q)=0$ for $k\not=0$.
Hence $(C_\bullet(\Sigma),\der_\bullet)$ is a
projective resolution of the rational discrete left $\QG$-module $\Q$ (cf. Fact~\ref{fact:simpcom3}).
In particular, $\ccd_{\Q}(G)\leq \dim(\Sigma)$. This yields (a).

If $\Sigma$ is contractable and of type $\tF_n$, $n\in\N_0\cup\{\infty\}$, then,
by \eqref{eq:simpiso1}, $C_k(\Sigma)$ is a finitely generated
projective rational discrete left $\QG$-module for $0\leq k\leq n$. This yields (b).
Part (c) follows by similar arguments.
\end{proof}

%%%%%

\subsection{Simplicial t.d.l.c. rational duality groups}
\label{ss:rattop}
As in the discrete case  it is pos\-sible to detect
from a group action
that a t.d.l.c. group $G$ is indeed a rational duality group. 

\begin{thm}
\label{thm:ratdual}
Let $G$ be a t.d.l.c. group, and let
$\Sigma$ be a $\euC$-discrete simplicial $G$-complex such that
\begin{itemize}
\item[($D_1$)] $\Sigma$ is contractable;
\item[($D_2$)] $\dim(\Sigma)<\infty$ and $\Sigma$ is of type $\tF_\infty$;
\item[($D_3$)] $\Sigma$ is locally finite.
\end{itemize}
Then $d=\ccd_{\Q}(G)<\dim(\Sigma)$ and one has
\begin{equation}
\label{eq:rattopthm}
{}^\times D_G=H^d_c(|\Sigma|,\Q)\otimes\Q(\Delta).
\end{equation}
In particular, if
\begin{itemize}
\item[($D_4$)] $H^k_c(|\Sigma|,\Q)=0$ for $k\not=d$,
\end{itemize}
then $G$ is a t.d.l.c. rational duality group of dimension $d$.
\end{thm}

\begin{proof}
By $(D_1)$, $(C_\bullet(\Sigma),\der_\bullet)$ is a projective resolution of $\Q$ in $\QGdis$,
and by $(D_2)$, $C_q(\Sigma)$ is a finitely generated projective rational discrete $\QG$-modules for all $q\geq 0$.
As $\Sigma$ is locally finite by $(D_3)$, $\der_q\colon C_q(\Sigma)\to C_{q-1}(\Sigma)$, $q\geq 1$, has
an adjoint map $\der_q^\ast\colon \Q[\utSigma^\ast_{q-1}]\to\Q[\utSigma^\ast_q]$, where one has to replace
$\tSigma_0$ by $\tSigma_0^{\pm}$ (cf. \S\ref{ss:ssets}). Moreover,
\begin{equation}
\label{eq:rattop1}
\argu^\vee_{\tSigma_q}\colon C_q(\Sigma)^{\odot}\longrightarrow \Q[\utSigma_q^\ast]
\end{equation}
is an isomorphism (cf. Prop.~\ref{prop:propG1}), and one has
the commutative diagram
\begin{equation}
\label{eq:rattop2}
\xymatrix@C=2cm{
C_q(\Sigma)^\odot\ar[r]^{\argu^\vee_{\tSigma_q}}\ar[d]_{\der_{q+1}^\odot}&\Q[\utSigma_q^\ast]\ar@{=}[r]
\ar[d]^{\uder^\ast_{q+1}}& C^q_c(\Sigma)\ar[d]^{\eth^q}\\
C_{q+1}(\Sigma)^\odot\ar[r]^{\argu^\vee_{\tSigma_{q+1}}}&\Q[\utSigma_{q+1}^\ast]\ar@{=}[r]&C^{q+1}_c(\Sigma)
}
\end{equation}
(cf. Prop.~\ref{prop:propG2}). In particular, one obtains canonical isomorphisms (cf. Fact~\ref{fact:simpcom4})
\begin{equation}
\label{eq:rattop3}
\dH^k(G,\caC_c(G))\simeq H^k_c(|\Sigma|,\Q)^{\times}
\end{equation}
of rational discrete right $\QG$-modules, and non-canonical isomorphisms
\begin{equation}
\label{eq:rattop4}
{}^\times\dH^k(G,\biB(G))\simeq H^k_c(|\Sigma|,\Q)\otimes\Q(\Delta)
\end{equation}
of rational discrete left $\QG$-modules (cf. Remark~\ref{rem:biC}).
In particular, as $G$ is of type $\FP$ (cf. Prop.~\ref{prop:ftopcd}(c)), one has
\begin{equation}
\label{eq:rattop5}
\ccd_{\Q}(G)=\max\{\,k\in\N_0\mid H^k_c(|\Sigma|,\Q)\not=0\,\}
\end{equation}
(cf. Prop.~\ref{prop:FP}). Thus for $d=\ccd_{\Q}(G)$ one has an isomorphism
\begin{equation}
\label{eq:rattop6}
{}^\times D_G=H^d_c(|\Sigma|,\Q)\otimes\Q(\Delta).
\end{equation}
This yields the claim.
\end{proof}

A t.d.l.c. group $G$ admitting a $\euC$-discrete simplicial $G$-complex
$\Sigma$ satisfying $(D_j)$, $1\leq j\leq 4$, will be called a {\it simplicial t.d.l.c. rational duality group}.
It should be mentioned that the dimension of $\Sigma$ as simplicial complex does not have to coincide
with $d=\ccd_{\Q}(G)$.

%%%%%
\subsection{Universal $\euC$-discrete simplicial $G$-complexes}
\label{ss:unisimp}
Let $\Sigma$ be a $\euC$-discrete simplical $G$-complex,
and let $\delta\colon |\Sigma|\times|\Sigma|\to\R^+_0$
be a $G$-invariant metric on its topological realization.
Then $(\Sigma,\delta)$ is called a {\it $\euC$-discrete $\CAT(0)$ simplicial  $G$-complex},
if 
\begin{itemize}
\item[(E$_1$)] $(|\Sigma|,\delta)$ is a complete $\CAT(0)$-space\footnote{For details on $\CAT(0)$-spaces see \cite{habr:cat0}.},
\item[(E$_2$)] $\Sigma$ is $G$-tame.
\end{itemize}
Note that property (E$_1$) implies that $|\Sigma|$ is contractable (cf. \cite[Prop.~11.7]{ab:build}).
Let $\caO$ be a compact open subgroup of $G$.
By hypothesis, any orbit $\caO\cdot x$, $x\in|\Sigma|$, is finite and thus in particular bounded.
In particular, $\caO$ has a fixed point $z$ in $|\Sigma|$ (cf. \cite[Thm.~11.23]{ab:build}).
Let $\omega(z)\in\Sigma$ denote the unique simplex such that $z$ is contained in the interior of $\omega(z)$
(cf. Fact~\ref{fact:simpcom0}). Then, by
property (E$_2$), 
$\caO\subseteq\stab_G(z)=\stab_G\{\omega(z)\}=\stab_G(\omega(z))$, i.e.,
$\Sigma$ has the Bruhat-Tits property (cf. \S\ref{ss:BT}). By property (E$_2$), 
$\Sigma^{\caO}$ is a simplicial subcomplex of $\Sigma$, and, by (E$_1$),
$(|\Sigma^{\caO}|,\delta)$ is a complete CAT(0)-space. In particular, $|\Sigma^{\caO}|$ is contractable.
This shows the following.

\begin{prop}
\label{prop:cat0}
Let $G$ be a t.d.l.c. group, and let
$\Sigma$ be a $\euC$-discrete $\CAT(0)$ simplical $G$-complex.
Then $|\Sigma|$ is an $\uE_{\euC}(G)$-space.
\end{prop}

%%%%%

\subsection{Compactly generated t.d.l.c. groups of rational discrete cohomological dimension $1$}
\label{sss:fund}
Let $G$ be a compactly generated t.d.l.c. group satisfying $\ccd_{\Q}(G)=1$.
Then $G$ is of type $\FP$ (cf. Thm.~\ref{thm:FP1}).
As $G$ is not compact (cf. Prop.~\ref{prop:cd}(a)), one has
\begin{equation}
\label{eq:cd1ex}
\dH^0(G,\biB(G))=\Hom_G(\Q,\biB(G))=0
\end{equation}
(cf. \eqref{eq:invbiB}). 
Hence $G$ is a rational t.d.l.c. duality group.

Suppose that $\Pi=\pi_1(\caA,\Lambda,\Xi)$ is the fundamental group
of a finite graph of profinite groups $(\caA,\Lambda)$, i.e., $\Lambda$ is a finite connected graph.
Moreover, $\Pi$ acts on the tree $\caT=\caT(\caA,\Lambda,\Xi,\caE^+)$ without inversion
of edges. L
We also assume that $\Pi$ is not compact. Then $\ccd_{\Q}(G)=1$ (cf. Prop.~\ref{prop:fgrgr}).
Hence the simplicial version of the tree $\caT$ - which we will denote by the same symbol -
has the following properties:
\begin{itemize}
\item[(1)] $\caT$ is a locally finite simplicial complex of dimension $1$;
\item[(2)] $\caT$ is a $\euC$-discrete simplicial $\Pi$-complex;
\item[(3)] $\Pi$ has finitely many orbits on the simplicial complex $\caT$, i.e., $\caT$ is of type $\tF_\infty$;
\item[(4)] $\caT$ is $\Pi$-tame;
\item[(5)] the standard metric $\delta\colon |\caT|\times |\caT|\to\R^+_0$ gives
$(|\caT|,\delta)$ the structure of a complete $\CAT(0)$-space.
\item[(6)]  $H^0_c(|\caT|,\Q)=0$.
\end{itemize}
Hence (1), (2), (3) and (5) imply that $\Pi$ is a simplicial t.d.l.c. duality group.
Moreover, by Theorem~\ref{thm:ratdual}, one has
\begin{equation}
\label{eq:ratdual1}
{}^\times D_{\Pi}\simeq H^1_c(|\caT|,\Q)\otimes\Q(\Delta).
\end{equation}
From (4) and (5) one concludes that $\caT$ is a $\euC$-discrete $\CAT(0)$-simplicial $\Pi$-complex,
and thus an $\uE_{\euC}(\Pi)$-space (cf. Prop.~\ref{prop:cat0}).

%%%%%

\subsection{Algebraic groups defined over a non-discrete, non-archimedean local field}
\label{sss:alg}
Let $K$ be a non-discrete non-archimedean local field\footnote{By a local field we understand
a non-discrete locally compact field. Such a field comes equipped with a valuation $v$,
and, in case that $v$ is non-archimedean, $v$ is also discrete. Hence such a field is either
isomorphic to a finite extension of the $p$-adic numbers $\Q_p$ or isomorphic to a
Laurent power series ring $\F_q[t^{-1},t]\!]$ for some finite field $\F_q$.} with residue field $\F$;
in particular, $\F$ is finite.
Let $G$ be a semi-simple, simply-connected algebraic group defined over $K$,
and let $G(K)$ denote the group of $K$-rational points. In particular, $G(K)$
carries naturally the structure of a t.d.l.c. group. Suppose that $G(K)$ is not compact.
Then $G(K)$ has a type preserving action
on an affine building $(\caC,S)$, the Bruhat-Tits building, where $(W,S)$ is an affine Coxeter group.
Moreover, $d=|S|-1$ coincides with the algebraic $K$-rank of $G$. For further details see
 \cite{bt:1}, \cite{bt:2}, \cite{tits:1} and \cite{weiss:build}.

The {\it Tits realization} $\Sigma(\caC,S)$ of the Bruhat-Tits building $(\caC,S)$
is a discrete simplicial $G(K)$-complex of dimension $d$
with the following properties:
\begin{itemize}
\item[(1)] For $\omega\in\Sigma(\caC,S)$ its stabilizer $\stab_{G(K)}(\omega)$ is a parahoric subgroup
and hence compact, i.e., $\Sigma(\caC,S)$ is a $\euC$-discrete simplicial $G(K)$-complex.
\item[(2)] $G(K)$ has finitely many orbits on $\Sigma(\caC,S)$, i.e.,
$\Sigma(\caC,S)$ is of type $\tF_\infty$.
\item[(3)] As $G(K)$ acts type preserving, $\Sigma(\caC,S)$ is $G(K)$-tame.
\item[(4)] $\Sigma(\caC,S)$ is locally finite.
\item[(5)] There exists a $G(K)$-invariant metric $\delta\colon |\Sigma(\caC,S)|\times|\Sigma(\caC,S)|\to\R^+_0$
making $(|\Sigma(\caC,S)|,\delta)$ a complete $\CAT(0)$-space (cf. \cite[Thm.~11.16]{ab:build}).
In particular, $|\Sigma(\caC,S)|$ is contractable.
\item[(6)] By a theorem of A.~Borel and J-P.~Serre,
one has $H_c^k(|\Sigma(\caC,S)|,\Q)=0$ for $k\not=d$ 
and $H_c^d(|\Sigma(\caC,S)|,\Q)\not=0$ (cf. \cite{bose:cr}, \cite{bose:top}).
\end{itemize}
From these properties one concludes the following.

\begin{thm}
\label{thm:algeu}
Let $G$ be a semi-simple, simply-connected algebraic group defined over 
the non-discrete non-archimedean local field $K$.
\begin{itemize}
\item[(a)] The topological realization $|\Sigma(\caC,S)|$
of the Tits realization of the affine building $(\caC,S)$
is an $\uE_{\euC}(G(K))$-space.
\item[(b)] $G(K)$ is a rational t.d.l.c.\,duality group of dimension
$\rk(G)=\dim(\Sigma(\caC,S))$. 
\item[(c)] For $d=\ccd_{\Q}(G(K))$ one has ${}^\times D_{G(K)}\simeq H_c^d(|\Sigma(\caC,S)|,\Q)$.
\end{itemize}
\end{thm}

\begin{proof}
(a) is a direct consequence of
the properties (1), (3), (5) and Proposition~\ref{prop:cat0}.

\noindent
(b) and (c) follow from the properties (1), (2), (4), (5), (6) and Theorem~\ref{thm:ratdual}.
\end{proof}

Note that $G(K)$ is in particular a simplicial t.d.l.c. rational duality group
and
\begin{equation}
\label{eq:algloc}
\dim_\euC(G(K))=\sdim_\euC(G(K))=\ccd_{\Q}(G(K)).
\end{equation}

%%%%%

\subsection{Topological Kac-Moody groups}
\label{ss:km}
Let $\F$ be a finite field, and let $G(\F)$ denote the rational points
of an almost split Kac-Moody group defined over $\F$ (cf. \cite[12.6.3]{remy:kac})
with infinite Weyl group $(W,S)$.
Such a group acts naturally on a twin building $\Xi$ (cf. \cite[1.A]{rr:KM}).
Let $\Sigma(\Xi^+)$ denote the Davis-Moussang realization of the positive part of $\Xi$
(cf. \cite{dav:rel}, \cite{mou:rel}).
In particular,
\begin{itemize}
\item[(1)] $\Sigma(\Xi^+)$ is a locally finite simplicial complex of finite dimension.
\end{itemize}
Moreover, $G(\F)$ is acting on $\Xi^+$ and hence on $\Sigma(\Xi^+)$, and thus
one has a homomorphism of groups $\pi\colon G(\F)\to\Aut(\Sigma(\Xi^+))$.
As $\Sigma(\Xi^+)$ is locally finite, $\Aut(\Sigma(\Xi^+))$ carries naturally the structure
of a t.d.l.c. group. The closure $\bG(\F)$ of $\image(\pi)$ is called the
{\it topological Kac-Moody group} associate to $G(\F)$ (cf. \cite[1.B and 1.C]{rr:KM}).
By construction, 
\begin{itemize}
\item[(2)] $\Sigma(\Xi^+)$ is a $\euC$-discrete $\bG(\F)$-simplicial complex.
\end{itemize}
Indeed, if $\F$ is of positive characteristic $p$, then $\stab_{\bG(\F)}(\omega)$ is a virtual pro-$p$ group
for all $\omega\in\Sigma(\Xi^+)$ (cf. \cite[p.~198, Theorem]{rr:KM}).
Let $G(\F)^\circ$ denote the subgroup of $G(\F)$ which elements
are type preserving on $\Xi$, and let $\bG(\F)^\circ$ denote
the closure of $\pi(G(\F)^\circ)$. Then $\bG(\F)^\circ$
is open and of finite index in $\bG(\F)$. As
$G(\F)^\circ$ is type preserving, one concludes that
\begin{itemize}
\item[(3)] $\Sigma(\Xi^+)$ is $\bG(\F)^\circ$-tame.
\end{itemize}
Since $G(\F)$ (and thus $\bG(\F)$) acts $\delta$-$2$-transitive on chambers of $\Xi$, one has that
\begin{itemize}
\item[(4)] $\Sigma(\Xi^+)$ is a $\euC$-discrete simplicial $\bG(\F)^\circ$-complex of type $\tF_\infty$ (cf. \S\ref{ss:simpGcom}).
\end{itemize}
By a remarkable result of M.W.~Davis (cf. \cite{dav:rel}) the topological realization
$|\Sigma(\Xi^+)|$ admits a $\bG(\F)^\circ$-invariant $\CAT(0)$-metric $\met\colon
|\Sigma(\Xi^+)|\times|\Sigma(\Xi^+)|\to\R^+_0$ (cf. \cite[Thm.~12.66]{ab:build}). Hence
Proposition~\ref{prop:cat0} implies  the following.

\begin{thm}
\label{thm:ebarkm}
Let $\F$ be a finite field,
let $G(\F)$ denote the rational points of an almost split Kac-Moody group defined over $\F$
with infinite Weyl group, and let
$\bG(\F)$ be the associated topological Kac-Moody group.
Then the topological realization $|\Sigma(\Xi^+)|$ of the David-Moussang realization $\Sigma(\Xi^+)$
is an $\uE_{\euC}(\bG(\F)^\circ)$-space.
\end{thm}

Note that by Davis' theorem, 
\begin{itemize}
\item[(5)] $\Sigma(\Xi^+)$ is a contractable $\euC$-discrete simplicial $\bG(\F)$-complex.
\end{itemize}
Hence Theorem~\ref{thm:ratdual} implies that
\begin{equation}
\label{eq:kac1}
d=\ccd_{\Q}(\bG(\F))=\ccd_{\Q}(\bG(\F)^\circ)\leq\dim(\Sigma(\Xi^+))<\infty,
\end{equation}
$\bG(\F)$ is of type $\FP$, and since $\bG(\F)$ is unimodular (cf. \cite[p.~811, Thm.(i)]{rem:super}),
\begin{equation}
\label{eq:kac2}
{}^\times D_{\bG(\F)}\simeq H_c^d(|\Sigma(\Xi^+)|,\Q).
\end{equation}
Moreover, by \eqref{eq:rattop5},
\begin{equation}
\label{eq:kac3}
\ccd_{\Q}(\bG(\F))=\max\{\,k\geq 0\mid H^k_c(|\Sigma(\Xi^+)|,\Q)\not=0\,\}.
\end{equation}
However, $\bG(\F)$ may or may not be a rational t.d.l.c. duality group (cf. Remark~\ref{rem:hyp} and \ref{rem:notdu}).
The cohomology with compact support of $|\Sigma(\Xi^+)|$ with coefficients in $\Z$
was computed in \cite{ddjmo:comp}. In more detail,
\begin{align}
H^k_c(|\Sigma(\Xi^+)|,\Z)&=
\coprod_{T\in\caS} H^k(K,K^{S-T})\otimes \hA^T(\Xi^+),
\label{eq:kac4}\\
\intertext{and a similar formula holds for the Davis realization of the Coxeter complex $\Sigma(C)$
associated to $(W,S)$ (cf. \cite{ddjo:cox}), i.e.,}
H^k_c(|\Sigma(C)|,\Z)&=
\coprod_{T\in\caS} H^k(K,K^{S-T})\otimes \hA^T(C),\label{eq:kac5}
\end{align}
Here $\caS$ denotes the set of non-trivial spherical subsets of $S$, and $K$ is a 
simplicial complex associated to $(W,S)$ build up from the spherical subsets of $S$.
A more detailed analysis of the coeffcient modules $\hA_T(\Xi^+)$ and $\hA_T(C)$
yields the following.

\begin{thm}
\label{thm:kmcd}
Let $\F$ be a finite field,
let $G(\F)$ denote the rational points of an almost split Kac-Moody group defined over $\F$
with infinite Weyl group $(W,S)$, and let
$\bG(\F)$ be the associated topological Kac-Moody group.
Then
\begin{equation}
\label{eq:km4}
\ccd_{\Q}(\bG(\F)^\circ)=\ccd_{\Q}(W).
\end{equation}
Moreover, $\bG(\F)^\circ$ (and thus $\bG(\F)$) is a rational t.d.l.c. duality group if,
and only if, $W$ is a rational discrete duality group. 
\end{thm}

\begin{proof}
Note that $\hA^T(\Xi^+)$ and $\hA^T(C)$ are free $\Z$-modules. Moreover,
for $T\in\caS$ one has $\hA^T(C)\not=0$ and $\hA^T(\Xi^+)\not=0$
(cf. \cite[\S 6, p.~570, Remark and Definition~7.4]{ddjmo:comp}).
As the abelian groups $H^k(K,K^{S-T})$ are finitely generated,
the universal coefficient theorem implies that
\begin{align}
\ccd_{\Q}(W)&=\max\{\,k\geq 0\mid \exists\, T\in\caS:\ H^k(K,K^{S-T})\ \text{is not torsion}\,\}\notag\\
&=\max\{\,k\geq 0\mid H^k_c(|\Sigma(\Xi^+)|,\Q)\not=0\,\}\notag\\
&=\ccd_{\Q}(\bG(\F))\label{eq:km5}
\end{align}
(cf. Thm.~\ref{thm:ratdual}). Let $d=\ccd_{\Q}(W)=\ccd_{\Q}(\bG(\F))$.
Then $W$ is a rational duality group if, and only if,
$H^k(K,K^{S_T})$ is a torsion group for all $T\in\caS$ and all $k\not=d$.
By \eqref{eq:kac5}, this is equivalent to $H^k_c(|\Sigma(\Xi^+)|,\Q)=0$ for all $k\not=d$.
Thus Theorem~\ref{thm:ratdual} yields the claim.
\end{proof}

\begin{rem}
\label{rem:hyp}
Let $\F$ be a finite field,
let $G(\F)$ denote the rational points of a split Kac-Moody group defined over $\F$, and
let $\bG(\F)$ be the associated topological Kac-Moody group. Suppose further that
\begin{itemize}
\item[(1)] the associated generalized Cartan matrix is symmetrizable, and that
\item[(2)] the associated Weyl group $(W,S)$ is hyperbolic (cf. \cite[\S 6.8]{hum:refl}).
\end{itemize}
Condition (2) implies that $W$ is a lattice in the real Lie group $\Ort(n,1)$, $|S|=n+1$ (cf. \cite[p.~140, and the references therein]{hum:refl}). 
Moreover, as $W$ is the Weyl group of a (split) Kac-Moody Lie algebra with property (1),
the Tits representation of the Weyl group $W$ is integral (cf. \cite[\S16.2]{car:kac}). Hence 
$W$ is an arithmetic lattice in $\Ort(n,1)$. Thus
by A.~Borel's and J-P.~Serre's theorem (cf. \cite{bose:coar}), 
$W$ is a virtual duality group, and thus, in particular, a rational duality group. Hence Theorem~\ref{thm:kmcd} implies that
$\bG(\F)$ is a rational discrete t.d.l.c. duality group.
\end{rem}

\begin{rem}
\label{rem:notdu}
Consider the Coxeter group $(W,S)$ associated to the Coxeter diagram
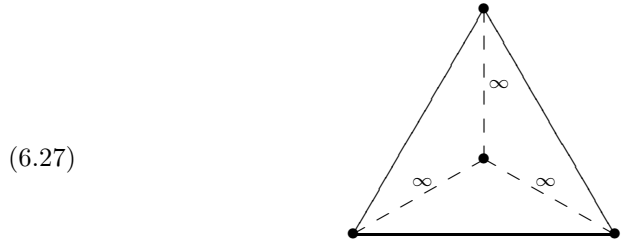
\begin{equation}
\label{eq:coxdia}
\begin{xy} 0;<2cm,0cm>:
(0,1)*=0{\bullet}="0", 
(0,0)*=0{\bullet}="1",
(-0.87,-0.5)*=0{\bullet}="2",
(0.87,-0.5)*=0{\bullet}="3",
(-0.41,-0,15)*=0{\scriptstyle{\infty}},
(0.41,-0,15)*=0{\scriptstyle{\infty}},
(0.1,0.5)*=0{\scriptstyle{\infty}},
\ar@{-} "0";"2",
\ar@{--} "1";"3",
\ar@{--} "1";"2",
\ar@{-} "2";"3",
\ar@{--} "0";"1",
\ar@{-} "0";"3"
\end{xy}
\end{equation}
Then $W\simeq W(\widetilde{A}_2)\coprod W(A_1)$. In particular, $\ccd_{\Q}(W)=2$.
By Stallings' decomposition theorem (cf. \cite{stall:ends}), $H^1(W,\Q[W])\not=0$. Thus 
$W$ is not a rational duality group. 
Let $G(\F)$ denote the rational points of a split Kac-Moody group defined over $\F$
with associated Coxeter group is $(W,S)$, and let
$\bG(\F)$ be the associated topological Kac-Moody group. Then Theorem~\ref{thm:kmcd} implies,
that $\bG(\F)$ is of rational cohomological dimension $2$, but that $\bG(\F)$ is not a 
rational t.d.l.c. duality group, i.e., $\dH^1(\bG(\F),\biB(\bG(\F))\not=0$.
The t.d.l.c. version of Stalling's decomposition theorem established in \cite{ila:stall} implies that
$\bG(\F)$ can be decomposed non-trivially as a free product with amalgamation in a compact open subgroup
(where one factor is a compact open subgroup as well), or as an HNN-extension with amalgamation in a compact
open subgroup. It would be interesting to understand how this decomposition is related to the root datum
of $G(\F)$.
\end{rem}

We close this section with a question which is motivated by
\eqref{eq:algloc}.

\begin{ques}
\label{ques:km}
For which topological Kac-Moody groups $\bG(\F)$, $\F$ a finite field,
does $\dim_{\euC}(\bG(\F))=\ccd_{\Q}(\bG(\F))$ hold (cf. \S\ref{eq:tdim})?
\end{ques}

%%%%%%%%
%% 5/3/2015
%%%%%%%%

\section{The Euler-Poincar\'e characteristic\\ of a 
unimodular t.d.l.c. group of type $\FP$}
\label{s:euler}
In this subsection we indicate how one can associate to any
unimodular t.d.l.c. group $G$ of type $\FP$ an {\it Euler-Poincar\'e characteristic}
\begin{equation}
\label{eq:euler1}
\chi(G)\in\boh(G)=\Q\cdot\mu_\caO,
\end{equation}
where $\caO$ is a compact open subgroup, and $\mu_\caO$ denotes the left invariant
Haar measure on $G$ satisfying $\mu_\caO(\caO)=1$.
If $h\in\Q^+\cdot \mu_\caO$ we simply write $h>0$; similarly
$h<0$ shall indicate that $h\in\Q^-\cdot \mu_\caO$.

In some particular cases discussed in \S\ref{sss:fund} and \S\ref{sss:alg} it will be possible to calculate
the Euler-Poincar\'e characteristic $\chi(G)$ explicitly. It is quite likely, that similar calculation can be done
also for the examples decribed in \S\ref{ss:km}. 
However, one of the most interesting question from our perspective which was our original motivation
to introduce and study this notion we were not able to answer.

\begin{ques}
\label{ques:cd1}
Let $G$ be a compactly generated t.d.l.c. group satisfying $\ccd_{\Q}(G)=1$.
Does this imply that $\chi(G)\leq 0$?
\end{ques}

An affirmative answer of Question~\ref{ques:cd1} would resolve the problem of
accessibility for compactly generated t.d.l.c. groups
of rational discrete cohomological dimension $1$ in analogy to the discrete case (cf. \cite{linn:acc}).

%%%%%

\subsection{The Hattori-Stallings rank of a finitely generated projective rational discrete $\QG$-module}
\label{ss:HS}
Let $G$ be a t.d.l.c. group, and let
$P$ be a finitely generated projective rational discrete $\QG$-module.
The evaluation map 
\begin{equation}
\label{eq:euler2}
\begin{gathered}
\ev_P\colon \Hom_G(P,\biB(G))\otimes_{\Q}P\to\biB(G),\qquad
\ev_P(\phi\otimes p)=\phi(p),\\
\phi\in\Hom_G(P,\biB(G)),\qquad p\in P,
\end{gathered}
\end{equation} 
induces a mapping $\uev_P\colon \Hom_G(P,\biB(G))\otimes_{G} P\to\ubG$
such that the diagram
\begin{equation}
\label{eq:euler3}
\xymatrix{
\Hom_G(P,\biB(G))\otimes_{\Q} P\ar[r]\ar[d]_{\ev_P}&\Hom_G(P,\biB(G))\otimes_G P\ar[d]^{\uev_P}\\
\biB(G)\ar[r]&\ubG
}
\end{equation}
commutes, where the horizontal maps are the canonical ones
and $\ubG$ is as defined in  \S\ref{ss:trace}. Then - provided that $G$ is unimodular -
one has a map
\begin{equation}
\label{eq:euler4}
\xymatrix{
\rho_P\colon\ 
\Hom_G(P,P)\ar[r]^-{\zeta_{P,P}^{-1}}&\Hom_G(P,\biB(G))\otimes_G P\ar[r]^-{\uev_P}&
\ubG\ar[r]^-{\utr}&\boh(G).}
\end{equation}
(cf. Prop.~\ref{prop:homten}, Prop.~\ref{prop:trace}).
The value $r_P=\rho_P(\iid_P)\in\boh(G)$ will be called the {\it Hattori-Stallings rank} of $P$.

\begin{prop}
\label{prop:hs}
Let $G$ be a unimodular t.d.l.c. group, and let $\caO\subseteq G$ be a com\-pact open subgroup.
Then $r_{\Q[G/\caO]}=\mu_{\caO}$.
\end{prop}

\begin{proof}
Let $\eta_{\caO}\colon\Q[G/\caO]\to\biB(G)$ be the canonical embedding (cf. \eqref{eq:homt1}).
Then $j_{\caO}(\eta_{\caO})=\caO\in \caO\backslash G$ (cf. \eqref{eq:ecky3}).
Hence,
\begin{align}
(j_\caO\otimes_G\iid_{\Q[G/\caO]})(\eta_\caO\otimes\caO)&=\caO\otimes_G\caO\label{eq:euler5}\\
\intertext{and}
\phi^{-1}_{\Q[G/\caO],\caO}(\caO\otimes_G\caO)&=\caO\in\Q[G/\caO]^\caO.\label{eq:euler6}
\end{align}
Under the Eckmann-Shapiro isomorphism
$\alpha\!\colon\Q[G/\caO]^{\caO}\to\Hom_G(\Q[G/\caO],\Q[G/\caO])$
the element $\caO$ is mapped to $\iid_{\Q[G/\caO]}$.
Hence, by \eqref{dia:homten}, 
\begin{equation}
\label{eq:euler7}
\zeta_{\Q[G/\caO],\Q[G/\caO]}(\eta_\caO\otimes_G\caO)=\iid_{\Q[G/\caO]},
\end{equation}
and $\uev_{\Q[G/\caO]}(\zeta^{-1}_{\Q[G/\caO],\Q[G/\caO]}(\iid_{\Q[G/\caO]}))=\underline{\caO}$,
where $\underline{\caO}$ denotes the canonical image of $\caO$ in $\ubG$. Thus,
by the definition of $\utr\colon\ubG\to\boh(G)$ (cf. \eqref{eq:trace3}),
\begin{equation}
\label{eq:euler8}
r_{\Q[G/\caO]}=\tr(\caO)=\mu_{\caO}
\end{equation}
and this yields the claim.
\end{proof}

\begin{rem}
\label{rem:HS}
Let $G$ be a unimodular t.d.l.c. group, and let $P\in\ob(\QGdis)$ be
finitely generated and projective. In \cite{it:hecke} it is shown, that
$r_P\geq 0$, and $r_P=0$ if, and only if, $P=0$.
\end{rem}

%%%%

\subsection{The Euler-Poincar\'e characteristic}
\label{ss:euler}
Let $G$ be a unimodular t.d.l.c. group,
and let $P_1,P_2\in\ob(\QGdis)$ be finitely generated and projective.
Then, by \eqref{eq:euler4}, one has
$\rho_{P_1\oplus P_2}(\iid_{P_1\oplus P_2})=\rho_{P_1}(\iid_{P_1})+\rho_{P_2}(\iid_{P_2})$, i.e.,
\begin{equation}
\label{eq:euler9}
r_{P_1\oplus P_2}=r_{P_1}+r_{P_2}.
\end{equation}
Let $(P_\bullet,\delta_\bullet)$ be a finite projective resolution of $\Q$ in $\QGdis$, i.e.,
there exists $N\geq 0$ such that $P_k=0$ for $k>N$, and $P_k$ is finitely generated for all $k\geq 0$.
From the identity \eqref{eq:euler10} one concludes that the value
\begin{equation}
\label{eq:euler10}
\chi(G)=\textstyle{\sum_{k\geq 0} (-1)^k\cdot r_{P_k}\in \boh(G)}
\end{equation}
is independent from the choice of the projective resolution $(P_\bullet,\delta_\bullet)$
(cf. \cite[Proposition~4.1]{toto:eulhec}) for a similar argument). We will call $\chi(G)$ the {\it Euler-Poincar\'e characteristic}
of $G$.

\subsection{Examples}
\label{ss:ExEuler}
\subsubsection{Compact t.d.l.c. groups}
\label{sss:coeul}
Let $\caO$ be a compact t.d.l.c. group.
Then Proposition~\ref{prop:hs} implies that $\chi(G)=r_{\Q}=\mu_{\caO}$.

%%%

\subsubsection{Fundamental groups of finite graphs of profinite groups}
\label{sss:fundeul}
Let $(\caA,\Lambda)$ be a finite graph of profinite groups.
Assume further that $\Pi=\pi_1(\caA,\Lambda,x_0)$ is unimodular\footnote{For a given
finite graph of profinite groups it is possible to decide when $\Pi$ is unimodular.
However, the complexity of this problem will depend on the rank of $H^1(|\Lambda|,\Z)$.
E.g., if $\Lambda$ is a finite tree, then $\Pi$ will be unimodular.}. Then,
by Proposition~\ref{prop:hs},
\begin{equation}
\label{eq:fundeul}
\chi(\Pi)=\sum_{v\in\caV(\Lambda)} \mu_{\caA_v}-\sum_{\eue\in\caE^+(\Lambda)} \mu_{\caA_\eue}.
\end{equation}
In particular, if $(\caA,\Lambda)$ is a finite graph of finite groups one obtains
\begin{equation}
\label{eq:fundeul2}
\chi(\Pi)=\Big(\sum_{v\in\caV(\Lambda)} \frac{1}{|\caA_v|}-\sum_{\eue\in\caE^+(\Lambda)} \frac{1}{|\caA_\eue|}\Big)\cdot\mu
=\chi_\Pi\cdot\mu_{\{1\}},
\end{equation}
where $\chi_\Pi$ denotes the Euler characteristic of the discrete group $\Pi$
(cf. \cite[\S II.2.6, Ex.~3]{serre:trees}). One has the following:

\begin{prop}
\label{prop:muless}
Let $(\caA,\Lambda)$ be a finite graph of profinite groups, such that
$\Pi=\pi_1(\caA,\Lambda,x_0)$ is unimodular and non-compact. Then
$\chi(\Pi)\leq 0$.
\end{prop}

\begin{proof}
Suppose that one of the open embeddings
$\alpha_\eue\colon\caA_\eue\to\caA_{t(\eue)}$ is surjective.
Then removing the edge $\eue$ form $\Lambda$,
idetifying $t(\eue)$ with $o(\eue)$ and taking the induced
graph of profinite groups $(\caA^\prime,\Lambda^\prime)$
does not change the fundamental group, i.e.,
one has a topological isomorphism $\pi_1(\caA,\Lambda,x_0)\simeq
\pi_1(\caA^\prime,\Lambda^\prime,x_0^\prime)$.
Thus we may assume that the finite graph of profinite groups has the property
that none of the injections $\alpha_\eue\colon\caA_\eue\to\caA_{t(\eue)}$
is surjective. Since $\Pi$ is not compact, $\Lambda$ cannot be a single vertex.
Thus, as $|\caA_{t(\eue)}:\alpha_\eue(\caA_{\eue})|\geq 2$ one has 
$\mu_{\caA_\eue}\geq\mu_{\caA_{t(\eue)}}+\mu_{\caA_{o(\eue)}}$
for any edge $\eue\in\caE(\Lambda)$.
Choosing a maximal subtree of $\Lambda$ then yields the claim.
\end{proof}

%%%%

\subsubsection{The automorphism group of a locally finite regular tree}
\label{sss:auttree}
Let $\caT_{d+1}$ be a $(d+1)$-regular tree, $1\leq d<\infty$, and let
$G=\Aut(\caT_{d+1})^\circ$ denote the group of automorphism
not inverting the orientation of edges. Then, as $G$ is
the fundamental group of a finite graph of profinite groups based on a tree
with $2$ vertices, $G$ is unimodular.
One concludes from Bass-Serre theory and \S\ref{sss:fundeul} that
\begin{equation}
\label{eq:auttree1}
\chi(\Aut(\caT_{d+1})^\circ)=\frac{1-d}{1+d}\cdot\mu_{G_\eue},
\end{equation}
where $G_\eue$ is the stabilizer of an edge.
Moreover, $\chi(\Aut(\caT_{d+1})^\circ)<0$.

%%%

\subsubsection{Chevalley groups over non-discrete non-archimedean local fields}
\label{sss:cheveul}
Let $X$ be a simply-con\-nec\-ted simple Chevalley group scheme,
and let $K$ be a non-archimedean local field with finite residue field $\F$.
Put $q=|\F|$. Let $W$ be the finite (or spherical) Weyl group associated to $X$,
and let $\tW$ denote the associated affine Weyl group associated to $W$.
We also fix a Coxeter generating system $\bDelta$ of $\tW$.
Let $n=\rk(W)=\rk(\tW)-1=|\bDelta|-1$ denote the rank of $W$ as Coxeter  group.
It is well-known that $X(K)$ modulo its center is simple. In particular, $X(K)$
is a unimodular t.d.l.c. group.

Let $\Sigma$ be the affine building associated to $X(K)$ (cf. \S\ref{sss:alg}).
The stablizer $\Iw=\stab_{X(K)}(\omega)$ of a simplex $\omega\in\Sigma_n$ of maximal dimension
is also called a {\it Iwahori subgroup} of $X(K)$.
Let $\Sigma(\omega)=\{\,\varpi\in\Sigma\mid \varpi\subseteq\omega\,\}=\caP(\omega)\setminus\{\emptyset\}$.
For $\varpi\in\Sigma(\omega)$ let $P_\varpi=\stab_{X(K)}(\varpi)$ denote the {\it parahoric subgroup}
associated to $\varpi$, i.e.,
$P_\omega=\Iw$. Then
\begin{equation}
\label{eq:algeul}
\chi(X(K))=\sum_{\varpi\in\Sigma(\omega)} (-1)^{|\varpi|}\cdot\mu_{P_{\varpi}},
\end{equation}
where $|\varpi|$ is the degree of $\varpi$, i.e., $\varpi\in\Sigma_{|\varpi|}$ (cf. \S\ref{sss:alg}).
Any parahoric subgroup $P_\varpi$, $\varpi\in\Sigma(\omega)$, corresponds
to a unique proper subset $\bDelta(\varpi)$ of $\bDelta$ of cardinality
$|\omega|-|\varpi|$, i.e., $\bDelta(\omega)=\emptyset$. Moreover,
\begin{equation}
\label{eq:algeul2}
|P_\varpi:\Iw|=p_{W(\bDelta(\varpi))}(q),
\end{equation}
where $p_{W(\bDelta(\varpi))}(t)\in\Z[t]$ is the Poincar\'e polynomial associated
to the finite Weyl group $W(\bDelta(\varpi))$. Thus, by \cite[\S5.12, Proposition]{hum:refl}, one obtains that
\begin{align}
\chi(X(K))&=\Big(\sum_{\varpi\in\Sigma(\omega)} \frac{(-1)^{|\varpi|}}{|P_\varpi:\Iw|}\Big)\cdot\mu_{\Iw}\notag\\
&=(-1)^{n+1}\cdot \Big(\sum_{\substack{I\subseteq \bDelta\\ I\not=\bDelta}} \frac{(-1)^{|I|-1}}{p_{W(I)}(q)}\Big)\cdot\mu_{\Iw}
=\frac{(-1)^{n+1}}{p_{\tW}(q)}\cdot\mu_{\Iw}.\label{eq:algeul3}\\
\intertext{By R.~Bott's theorem, one has
$p_{\tW}(t)=p_W(t)/\prod_{1\leq i\leq n} (1-t^{m_i})$, where $m_i=d_i-1$ are the exponents
of $W$ (cf. \cite[\S 8.9, Theorem]{hum:refl}). 
Thus $\chi(X(K))$ can be rewritten as}
\chi(X(K))&=-\frac{1}{p_W(q)}\cdot \prod_{1\leq i\leq n} (q^{m_i}-1)\cdot\mu_{\Iw}.\label{eq:algeul4}
\end{align}
In particular, $\chi(X(K))<0$.

\begin{rem}
\label{rem:zeta}
One may speculate whether in a situation which is controlled by
geometry the Euler-Poincar\'e characteristic should be related
to a meromorphic {\it {$\zeta$-function}} evaluated in $-1$.
There are some phenomenon supporting this idea which arise
even in a completely different context (cf. \cite{brown:zeta}).
Indeed, using {\it Hecke algebras} one may define
a formal Dirichlet series $\zeta_{G,\caO}(s)$ for
certain t.d.l.c. groups $G$ and any compact open subgroup $\caO$ of 
such a group $G$. It turns out that for the example discussed in \S\ref{sss:cheveul}
one obtains $\chi(X(K))=(-1)^{n+1}\cdot\zeta_{X(K),\Iw}(-1)^{-1}\cdot\mu_{\Iw}$ (cf. \cite{it:hecke}).
Nevertheless, further investigations seem necessary in order to shed light on
such a possible connection.
\end{rem}

%%%%

We close the discussion of the Euler-Poincar\'e characteristic
with a question which is motivated by the examples discussed
in \S\ref{sss:fundeul} and \S\ref{sss:cheveul}.

\begin{ques}
\label{ques:cat0eul}
Suppose that $G$ is a t.d.l.c. group admitting a $\euC$-discrete simplicial $G$-complex $\Sigma$
such that
\begin{itemize}
\item[(1)] $\dim(\Sigma)=\ccd_{\Q}(G)<\infty$;
\item[(2)] $\Sigma$ is of type $\tF_\infty$;
\item[(3)] $\Sigma$ is locally finite;
\item[(4)] $\Sigma$ is $G$-tame;
\item[(5)] $|\Sigma|$ admits a $G$-invariant metric
$\delta\colon |\Sigma|\times |\Sigma|\to\R^+_0$ making 
$(|\Sigma|,\delta)$ a complete $\CAT(0)$-space (cf. \S\ref{ss:unisimp}).
\end{itemize}
What further condition ensures that $\chi(G)\leq 0$?
\end{ques}

The most trivial examples (cf. \S\ref{sss:coeul}) show that the conditions (1)-(5) are not sufficient
(for $G$ compact $\Sigma$ can be chosen to consist of 1 point).

%%%%%%%%%%%%%%%%%%%%%%%%%%%%%%
%%%% last update 3/3/2015 %%%%%%%%%%%%%%%
%%%%%%%%%%%%%%%%%%%%%%%%%%%%%%

\appendix
\section{Simplicial complexes}
\label{s:sc}

\subsection{Simplicial complexes}
\label{ss:simpcom}
A {\it simplicial complex} $\Sigma$ is a non-empty set of  
non-trivial finite subsets of a set $X$ with the property
that if $B\in\Sigma$ and $A\subseteq B$ then $A\in\Sigma$.
Any simplicial complex $\Sigma$ carries a canonical
grading $\Sigma=\bigsqcup_{q\geq 0}\Sigma_q$, where
\begin{equation}
\label{eq:simpcom1}
\Sigma_q=\{\, A\in\Sigma\mid \card(A)=q+1\,\}.
\end{equation}
\begin{example}
\label{ex:simpcom}
Let $X$ be a set, and let $\Sigma[X]\subseteq\caP(X)\setminus\{\emptyset\}$ denote the set of all non-trivial finite subsets of $X$.
Then $\Sigma[X]$ is a simplicial complex -
the simplicial complex generated by $X$.
\end{example}

\begin{example}
\label{ex:graph}
Let $\Gamma=(\caV(\Gamma),\caE(\Gamma))$ be a combinatorial graph (cf. \S\ref{ss:tdlctree}).
Then $\Sigma(\Gamma)$ given by 
$\Sigma_0(\Gamma)=\caV(\Gamma)$ and $\Sigma_1(\Gamma)=\{\,\{o(\eue),t(\eue)\}\mid
\eue\in\caE(\Gamma)\,\}$ is a
simplicial complex - the {\it simplicial complex associated with the graph} $\Gamma$.
\end{example}

In a simplicial complex $\Sigma$ and $n\geq 0$ the subset 
\begin{equation}
\label{eq:simpcom2}
\Sigma^{(n)}=\{\,A\in\Sigma\mid \card(A)\leq n+1\,\}
\end{equation}
is again simplicial subcomplex of $\Sigma$ - the {\it $n$-skeleton} 
of $\Sigma$.
A subset $Y\subseteq\Sigma_0$ defines a simplicial subcomplex
\begin{equation}
\label{eq:simpcom3}
\Sigma(Y)=\{\,C\in\Sigma\mid C\subseteq Y\,\}.
\end{equation}
For a simplicial complex $\Sigma$ one defines
the {\it dimension} by
\begin{equation}
\label{eq:simpcom4}
\dim(\Sigma)=\min(\{\,q\in\N_0\mid \Sigma_{q+1}=\emptyset\,\}\cup\{\infty\}).
\end{equation} 
In particular, $\dim(\Sigma(X))=\card(X)-1$.
A simplicial complex of dimension $0$ is just a set,
while simplicial complexes of dimension $1$ are unoriented graphs (cf. Ex.~\ref{ex:graph}).
For $A\in\Sigma_q$ one defines the {\it set of neighbours} $N_\Sigma(A)$
of $A$ by
\begin{equation}
\label{eq:simpcom5}
N_\Sigma(A)=\{\,B\in\Sigma_{q+1}\mid A\subset B\,\}.
\end{equation}
Moreover, $\Sigma$ is called {\it locally finite}
if for all $A\in\Sigma_q$ one has $\card(N_\Sigma(A))<\infty$.

%%%%%

\subsection{The topological realization of a simplicial complex}
\label{ss:geore}
With any simplicial complex $\Sigma\subseteq\caP(X)\setminus\{\emptyset\}$
one can associate a topological space $|\Sigma|$ -
the {\it topological realization} of $\Sigma$ - which is a {\it CW-complex}.
The space $|\Sigma|$ can be constructed as follows:
Let $V=\R[X]$ denote the free $\R$-vector space over the set $X$,
and for $A\in\Sigma_q$, $A=\{\,a_0,\ldots,a_q\}$, put
\begin{equation}
\label{eq:simpcom6}
\textstyle{|A|=\{\,\sum_{j=0}^q t_j\cdot a_j\mid
0\leq t_j\leq 1,\ \sum_{j=0}^q t_j=1\,\}\subset\R[X].}
\end{equation}
Then $|\Sigma|=\bigcup_{A\in\Sigma} |A|\subset\R[X]$,
and $|\Sigma|$ carries the weak topology with respect to the
embeddings $i_A\colon |A|\to|\Sigma|$, where $|A|\subset\R[A]$ carries
the induced topology (cf. \cite[\S 3.1]{span:top}).
By construction, one has the following property.

\begin{fact}
\label{fact:simpcom0}
Let $\Sigma$ be a simplical complex, and let $z\in |\Sigma|$.
Then there exists a unique simplex $A\in\Sigma$ such that
$z$ is an interior point of $|\Sigma(A)|\subseteq |\Sigma|$.
\end{fact}

The topological space $|\Sigma|$ has the following well-known property.

\begin{fact}
\label{fact:simpcom1}
Let $\Sigma$ be a simplicial complex, and let $C$
be a compact subspace of $|\Sigma|$. Then there exists
a finite subset $Y\subseteq\Sigma_1$ such that $C\subseteq |\Sigma(Y)|$.
\end{fact}

\begin{proof}
(cf. \cite[p.~520, Prop.~A.1]{hatcher:at}).
\end{proof}

The just mentioned property has the following consequence.

\begin{fact}
\label{fact:simpcom2}
Let $X$ be a non empty set. Then $|\Sigma[X]|$ is contractable.
\end{fact}

\begin{proof}
Put $\Sigma=\Sigma[X]$.
As $|\Sigma|$ is a connected CW-complex, one concludes from 
Whitehead's theorem
(cf. \cite[Thm.~4.5]{hatcher:at}) that
it suffices to show that $\pi_i(|\Sigma|,x_0)=1$ for all $i\geq 1$.
Note that $\Sigma=\bigcup_{A\in \caP_{x_0}} \Sigma(A)$,
where 
\begin{equation}
\label{eq:simpcom7}
\caP_{x_0}=\{\,A\subseteq X\mid x_0\in A,\ \card(A)<\infty\,\}.
\end{equation}
By Fact~\ref{fact:simpcom1}, for any compact subset $C$ of $|\Sigma|$, there
exists a finite set $A\in\caP_{x_0}$ such that $C\subseteq|\Sigma(A)|$.
Hence $\pi_i(|\Sigma|,x_0)=\varinjlim_{A\in\caP_{x_0}} \pi_i(|\Sigma(A)|,x_0)$.
As $|\Sigma(A)|$ coincides with the standard $(\card(A)-1)$-simplex,
one has $\pi_i(|\Sigma(A)|,x_0)=1$. 
\end{proof}

%%%%

\subsection{The chain complex of a simplicial complex}
\label{ss:osc}
Let $\Sigma\subseteq\caP(X)\setminus\{\emptyset\}$
be a simplicial complex. For $q\geq 0$ put
\begin{equation}
\label{eq:simpcom8}
\tSigma_q=\{\,x_0\wedge\cdots \wedge x_q\mid \{x_0,\ldots,x_q\}\in\Sigma_q\,\}\subset
\Lambda_{q+1}(\Q[X]),
\end{equation}
where $\Lambda_\bullet(\Q[X])$ denotes the 
{\it exterior algebra} of the free $\Q$-vector space over the set $X$,
and define
\begin{equation}
\label{eq:simpcom9}
C_q(\Sigma)=\spn_{\Q}(\tSigma_q)\subseteq\Lambda_{q+1}(\Q[X]).
\end{equation}
Then $(C_\bullet(\Sigma),\der_\bullet)$, where
\begin{equation}
\label{eq:simpcom10}
\textstyle{
\der_q(x_0\wedge\cdots\wedge x_q)=\sum_{0\leq j\leq q} (-1)^j\,
x_0\wedge\cdots\wedge x_{j-1}\wedge x_{j+1}\wedge\cdots\wedge x_q
}
\end{equation}
is a chain complex of $\Q$-vector spaces - the {\it rational chain complex of the 
simplicial complex} $\Sigma$ (cf. \cite[\S 4.1]{span:top}).
It has the following well-known property (cf. \cite[\S4.3, Thm.~8]{span:top}).

\begin{fact}
\label{fact:simpcom3}
Let $\Sigma$ be a simplicial complex. Then 
one has a canonical isomorphism
\begin{equation}
\label{eq:simpcom11}
H_\bullet(C_\bullet(\Sigma),\der_\bullet)\simeq H_\bullet(|\Sigma|,\Q).
\end{equation}
\end{fact}

%%%%

\subsection{Cohomology with compact support}
\label{ss:cosup}
Let $\Sigma\subseteq\caP(X)\setminus\{\emptyset\}$ be a locally finite simplicial complex.
Put $X^\ast=\{\,x^\ast\mid x\in X\,\}$, and think of $X^\ast$ as a subset
of the $\Q$-vector space $\caC_c(X,\Q)$, the functions from $X$ to $\Q$ with finite support,
i.e., for $x,y\in X$ one has $x^\ast(y)=\delta_{x,y}$. Here $\delta_{.,.}$ denotes Kronecker's function.
For $q\geq 0$ put
\begin{equation}
\label{eq:simpcom12}
\hSigma^q=\{\,x_0^\ast\wedge\cdots \wedge x_q^\ast\mid 
\{x_0,\ldots,x_q\}\in\Sigma_q\,\}\subset
\Lambda_q(\Q[X^\ast]),
\end{equation}
and define
\begin{equation}
\label{eq:simpcom13}
C^q_c(\Sigma)=\spn_{\Q}(\hSigma^q)\subseteq\Lambda_q(\Q[X^\ast]).
\end{equation}
By definition, for $A\in\Sigma_q$ the set
\begin{equation}
\label{eq:simpcom14}
I_\Sigma(A)=\{\,x\in\Sigma_1\setminus A\mid A\cup\{x\}\in\Sigma_{q+1}\,\}
\end{equation}
is finite.
Then $(C^\bullet_c(\Sigma),\eth^\bullet)$, where
\begin{equation}
\label{eq:simpcom15}
\textstyle{
\eth^q(x_0^\ast\wedge\cdots\wedge x_q^\ast)=
\sum_{z\in I_\Sigma(\{x_0,\ldots,x_q\})} 
z\wedge x_0^\ast\wedge\cdots\wedge x_q^\ast\,
}
\end{equation}
$x_0^\ast\wedge\cdots\wedge x_q^\ast\in\hSigma^q$, is a cochain complex.
It has the following well-known property (cf. \cite[p.~242ff]{hatcher:at}).

\begin{fact}
\label{fact:simpcom4}
Let $\Sigma$ be a locally finite simplicial complex. Then one has a
canonical isomorphism
$H^\bullet(C_c^\bullet(\Sigma),\eth^\bullet)\simeq H^\bullet_c(|\Sigma|,\Q)$,
where $H^\bullet_c(\argu,\Q)$ denotes cohomology with compact support
with coefficients in $\Q$.
\end{fact}

%%%%

\subsection{Signed sets}
\label{ss:ssets}
A set $X$ together with a map
$\bar{\argu}\colon X\to X$ satisfying $\bar{\bar{x}}=x$
and $\bar{x}\not=x$ for all $x\in X$ will be called a {\it signed set}.
A map of signed sets $\phi\colon X\to Y$ is a map satisfying
$\phi(\bar{x})=\overline{\phi(x)}$ for all $x\in X$. Every signed set
$X$ defines a $\Q$-vector space
\begin{equation}
\label{eq:sset1}
\Q[\uX]=\Q[X]/\spn_{\Q}\{\,x+\bar{x}\mid x\in X\,\}.
\end{equation}
For a signed set $X$ let $X^\ast=\{\,x^\ast\mid x\in X\,\}$
denote the signed set satisfying $\bar{x}^\ast=\overline{x^\ast}$.
Then we may consider $\Q[\uX^\ast]$ as a $\Q$-subspace of $\Q[\uX]^\ast=\Hom_{\Q}(\Q[\uX],\Q)$, i.e.,
\begin{equation}
\label{ss:set2}
x^\ast(y)=\begin{cases}
\hfil 1\hfil &\ \text{if $y=x$,}\\
\hfil-1\hfil &\ \text{if $y=\bar{x}$,}\\
\hfil 0\hfil &\ \text{if $y\not\in\{x,\bar{x}\}$.}
\end{cases}
\end{equation}
Any map of signed sets $\phi\colon X\to Y$ defines
a $\Q$-linear map $\phi_\Q\colon\Q[\uX]\to\Q[\uY]$, but it does not necessarily
possess an adjoint map $\phi_\Q^\ast\colon \Q[\uY^\ast]\to\Q[\uX^\ast]$.
However, if $\phi$ is proper, i.e., $\phi$ has finite fibres,
then there exists a unique map $\phi_\Q^\ast$ satisfying
\begin{equation}
\label{eq:sset3}
\langle v^\ast,\phi_\Q(u)\rangle_Y=\langle \phi^\ast_\Q(v),u\rangle_X,\qquad u\in\Q[\uX],\ v\in\Q[\uY^\ast],
\end{equation}
where $\langle.,.\rangle_X$ and $\langle .,.\rangle_Y$ denote the evaluation mappings, respectively.
Let $X^+\subset X$ and $Y^+\subset Y$ be a set of representative for the $\Z/2\Z$-orbits on $X$ and $Y$, respectively.
For a map $\psi\colon \Q[\uX]\to\Q[\uY]$ and $\ux\in\Q[\uX]$ - the canonical image of $x\in X^+$ in $\Q[\uX]$ -
one has $\psi(\ux)=\sum_{y\in Y^+} \lambda_y(x)\cdot \uy$ for $\lambda_y(x)\in\Q$. Put
\begin{equation}
\label{eq:sset4}
\supp(\psi,x)=\{\,y\in Y^+\mid \lambda_y(x)\not=0\,\}.
\end{equation}
We call the map $\psi\colon\Q[\uX]\to\Q[\uY]$ {\it proper}, if
for all $y\in Y^+$ the set 
\begin{equation}
\label{eq:sset5}
\csupp(\psi,y)=\{\,x\in X^+\mid y\in\supp(\phi,x)\,\}
\end{equation} 
is finite.
It is straightforward to verify that the mapping $\psi^\ast\colon \Q[\uY^\ast]\to\Q[\uX^\ast]$,
\begin{equation}
\label{eq:sset6}
\psi^\ast(\uy^\ast)=\sum_{x\in\csupp(\psi,y)} \lambda_y(x)\cdot \ux^\ast,\qquad y\in Y^+
\end{equation}
satisfies \eqref{eq:sset3} and thus can be seen as the adjoint of $\psi$.

If $\Sigma$ is a simplicial complex, then $\tSigma_q$, $q\geq 1$, are signed sets, and one has a
canonical isomorphism
\begin{equation}
\label{eq:sset7}
\Q[\underline{\tSigma}_q]\simeq C_q(\Sigma).
\end{equation}
The same applies also for $q=0$ replacing
$\tSigma_0$ by 
\begin{equation}
\label{eq:sset9}
\tSigma_0^{\pm}=\{\,\pm x_0\mid x_0\in\Sigma_0\,\}\subseteq \Lambda_1(\Q[\Sigma_0]),
\end{equation}
i.e., $\Q[\underline{\tSigma}_0^{\pm}]=C_0(\Sigma)$.
Moreover, if $\Sigma$ is locally finite, then
$\der_q\colon C_q(\Sigma)\to C_{q-1}(\Sigma)$ is proper for $q\geq 2$. By construction,
one has for $q\geq 1$ a canonical isomorphism
\begin{equation}
\label{eq:sset8}
C^q_c(\Sigma)\simeq\Q[\underline{\tSigma}_q^\ast],
\end{equation}
and 
$C^0_c(\Sigma)\simeq \Q[(\underline{\tSigma_0^{\pm}})^\ast]$. Moreover,
$\eth^q=\der_{q+1}^\ast$. 
%%%%%%%%%%%%%%%%%%%%%
%%%%%%%%%%%%%%%%%%%%%

\providecommand{\bysame}{\leavevmode\hbox to3em{\hrulefill}\thinspace}
\providecommand{\MR}{\relax\ifhmode\unskip\space\fi MR }
\providecommand{\MRhref}[2]{%
  \href{http://www.ams.org/mathscinet-getitem?mr=#1}{#2}
}
\providecommand{\href}[2]{#2}

%%%%%%%%%%%%%%%%%%%%%%%%%%%%%%%%%%%%%%%%%%
%%%%%%%%%%%%%%%%%%%%%%%%%%%%%%%%%%%%%%%%%
%\bibliography{qratbib}
%\bibliographystyle{amsplain}

\end{document}

%%%%%%%%%%%%%%%%%%%%%%%%%%%%%%%%%%%%%%%%%
%%%%%%% This is the end %%%%%%%%%%%%%%%%%%%%%%%%%%
%%%%%%%%%%%%%%%%%%%%%%%%%%%%%%%%%%%%%%%%%